\documentclass[12pt,a4paper]{amsart}
\usepackage{amssymb}
\usepackage[T2A]{fontenc}
\usepackage[cp1251]{inputenc}
\usepackage[english]{babel}
\usepackage{srcltx}
\usepackage{upgreek}

\advance\textwidth20mm \advance\hoffset-10mm
\advance\textheight20mm \advance\voffset-10mm
\theoremstyle{plain}
\newtheorem{theorem}{Theorem}[section]
\newtheorem{corollary}[theorem]{Corollary}
\newtheorem{lemma}[theorem]{Lemma}

\newtheorem{proposition}[theorem]{Proposition}
\theoremstyle{remark}
\newtheorem{remark}[theorem]{Remark}
\newtheorem{definition}[theorem]{Definition}

\newtheorem*{example}{Example}
\numberwithin{equation}{section}

\newcommand\R{\mathbb{R}}
\newcommand\N{\mathbb{N}}
\newcommand\Z{\mathbb{Z}}

\newcommand\bsigma{\boldsymbol\upsigma}
\newcommand\bmu{\boldsymbol\upmu}
\newcommand\ba{\mathbf{a}}
\newcommand\<{\langle}
\renewcommand\>{\rangle}

\newcommand\esssup{\operatorname{ess\,sup}}
\newcommand\supp{\operatorname{supp}}
\newcommand\Cdot{{\mskip2mu{\cdot}\mskip2mu}}
\newcommand\Ast[1]{\ast_{\scriptscriptstyle\mskip-2mu #1}}
\newcommand\aomega{{^{*\!}}\omega}
\newcommand\aaomega{{^{*\!*}}\omega}
\newcommand\aDelta{{^{*\!}\!}\varDelta}
\newcommand\aaDelta{{^{*\!*\!}\!}\varDelta}
\newcommand\aj{j^{*}}
\newcommand\aaj{j^{*\!*}}

\begin{document}

\title[$L^p$-bounded Dunkl-type generalized translation operator]
{Positive $L^p$-bounded Dunkl-type generalized\\ translation operator and its applications}

\author{D.~V.~Gorbachev}
\address{D.~Gorbachev, Tula State University,
Department of Applied Mathematics and Computer Science,
300012 Tula, Russia}
\email{dvgmail@mail.ru}

\author{V.~I.~Ivanov}
\address{V.~Ivanov, Tula State University,
Department of Applied Mathematics and Computer Science,
300012 Tula, Russia}
\email{ivaleryi@mail.ru}

\author{S.~Yu.~Tikhonov}
\address{S. Tikhonov, ICREA, Centre de Recerca Matem\`{a}tica, and UAB\\
Campus de Bellaterra, Edifici~C
08193 Bellaterra (Barcelona), Spain}
\email{stikhonov@crm.cat}

\date{\today}

\thanks{The work of D.V.~Gorbachev and V.I.~Ivanov is supported by the Russian Science Foundation under grant N\,18-11-00199 and performed in Tula State University. The work of S.Yu.~Tikhonov is partially supported by MTM 2017-87409-P, 2017 SGR 358, and by the
CERCA Programme of the Generalitat de Catalunya.}

\begin{abstract}
We prove that the spherical mean value of the Dunkl-type generalized
translation operator $\tau^y$ is a positive $L^p$-bounded generalized
translation operator $T^t$. As application, we prove the Young inequality for a
convolution defined by $T^t$, the $L^p$-boundedness of $\tau^y$ on a radial
functions for $p>2$, the $L^p$-boundedness of the Riesz potential for the Dunkl
transform and direct and inverse theorems of approximation theory in
$L^p$-spaces with the Dunkl weight.
\end{abstract}

\maketitle
%\bigskip
\tableofcontents

\section{Introduction}
During the last three decades, many important elements of harmonic analysis
with Dunkl weight on $\R^d$ and ${\mathbb{S}}^{d-1}$ were proved; see, e.g.,
the papers by C.F.~Dunkl \cite{Dun89,Dun91,Dun92}, M.~R\"{o}sler
\cite{Ros98,Ros99,Ros03,Ros02}, M.F.E.~de~Jeu \cite{Jeu93,Jeu06},
K.~Trim\`{e}che \cite{Tri01,Tri02}, Y.~Xu \cite{Xu97,Xu00}, and the recent
works \cite{anker, Dai13, Dai15, iv1, iv2}.

Yet there are still several gaps in our knowledge of Dunkl harmonic analysis. In particular, Young's convolution inequality, several important polynomial inequalities, and basic approximation estimates are not established in the general case. One of the main reasons is the lack of tools  related to   the translation operator. Needless to say,  the standard translation operator $f\mapsto f(\Cdot+y)$ plays a crucial
role both in classical approximation theory and harmonic analysis, in particular, to
introduce several smoothness characteristics of $f$. In Dunkl  analysis,
its analogue is the generalized translation operator $\tau^y$ defined by
M.~R\"{o}sler \cite{Ros98}. Unfortunately, the $L^p$-boundedness of $\tau^y$ is
not obtained in general.

%This is the main reason why only part of the
%results in the Fourier analysis has been extended to the Dunkl analysis at the moment.

To overcome this difficulty,
the spherical mean value of the %generalized
translation operator $\tau^y$ was introduced in \cite{MejTri01} and it was
studied in \cite{Ros03}, where, in particular, its positivity was shown.
Our main goal in this paper is to prove that this operator %e spherical mean value of the generalized translation operator $\tau^y$
is a positive $L^p$-bounded %generalized translation
operator $T^t$, which may be considered as a generalized translation operator.
%The operator $T^t$ was first considered in \cite{MejTri01} and
%it was studied in \cite{Ros03}, where, in particular, its positivity was shown.
It is worth mentioning that this operator can be applied to problems where %notions are defined with help of the
it is essential to deal with radial multipliers. This is because by virtue of
$T^t$ we can define the convolution operator which coincides with the known
convolution introduced by S.~Thangavelu and Y.~Xu in \cite{ThaXu05} using the
operator $\tau^y$.

For this convolution we prove the Young inequality and, subsequently, an
$L^p$-boundedness of the operator $\tau^y$ on a radial functions for $p>2$. For
$1\leq p\leq2$ it was proved in \cite{ThaXu05}.

Let us mention here two applications of the operator $T^t$. The first one is
the Riesz potential defined in \cite{ThaXu07}, where its boundedness properties
were obtained for the reflection group $\Z_2^d$. For the general case see
\cite{HasMusSif09}. Using the $L^p$-boundedness of the operator $T^t$ allows us
to give a different simple proof, which follows ideas of \cite{ThaXu07}.
Another application is basic inequalities of approximation theory in the
weighted $L^p$ spaces. With the help of the operator $T^t$ one can define
%smoothness characteristics
moduli of smoothness, which are equivalent to the $K$-functionals, and prove
the direct and inverse approximation theorems. For the reflection group
$\Z_2^d$, basic approximation inequalities were studied in \cite{Dai13, Dai15}.

The paper is organized as follows. In the next section, we give some basic
notation and facts of Dunkl harmonic analysis. In Section 3, we study the
operator $T^t$, define a convolution operator and prove the Young inequality.
As a consequence, we obtain an $L^p$-boundedness of the operator $\tau^y$ on a
radial functions. The weighted Riesz potential is studied in Section 4. Section
5 consists of a study of interrelation between several classes of entire
functions. We also obtain multidimensional weighted analogues of
Plancherel--Polya--Boas inequalities, which are of their own interest. In
Section~6 we introduce moduli of smoothness and the $K$-functional, associated
to the Dunkl weight, and prove equivalence between them as well as the Jackson
inequality. Section~7 consists of weighted analogues of Nikol'ski\v{i},
Bernstein, and Boas inequalities for entire functions of exponential type. In
Section~8, we obtain that moduli of smoothness are equivalent to the
realization of the $K$-functional. We conclude with Section~9, where we prove
the inverse theorems in $L^p$-spaces with the Dunkl weight.

\bigskip
\section{Notation}
In this section, we recall the basic notation and results of  Dunkl harmonic analysis, see, e.g., % one can find in
\cite{Ros02}.

Throughout the paper, $\<x,y\>$ denotes the standard Euclidean scalar product in
 $d$-dimensional Euclidean space $\R^{d}$, $d\in \N$,
equipped with a norm $|x|=\sqrt{\<x,x\>}$.
  For $r>0$ we write
$B_r=\{x\in\R^d\colon |x|\leq r\}$. Define the following function spaces:
\begin{itemize}\renewcommand\labelitemi{{\boldmath$\cdot$}}
  \item $C(\R^d)$ the space of continuous
functions,
  \item $C_b(\R^d)$ the space of bounded continuous functions with the norm
$\|f\|_{\infty}=\sup_{\R^d}|f|$,
  \item $C_0(\R^d)$ the space of continuous functions
which vanish at infinity,
  \item $C^{\infty}(\R^d)$ the space of infinitely
differentiable functions,
 \item $C^{\infty}_{\Pi}(\mathbb{R}^d)$ the space of infinitely
differentiable functions whose derivatives have polynomial growth at infinity,
  \item $\mathcal{S}(\R^d)$ the Schwartz space,
  \item $\mathcal{S}'(\R^d)$
  the space of tempered distributions,
  \item $X(\R_+)$ the space of even functions from $X(\R)$, where $X$ is one of the spaces above,
  \item
 $X_\mathrm{rad}(\R^d)$ the subspace of $X(\R^d)$ consisting of radial functions
 $f(x)=f_{0}(|x|)$.

\end{itemize}

Let a finite subset $R\subset \mathbb{R}^{d}\setminus\{0\}$ be a root system,
$R_{+}$ a positive subsystem of $R$, $G(R)\subset O(d)$ the finite reflection
group, generated by reflections $\{\sigma_{a}\colon a\in R\}$, where
$\sigma_{a}$ is a reflection with respect to hyperplane $\<a,x\>=0$, $k\colon
R\to \mathbb{R}_{+}$ a $G$-invariant multiplicity function. Recall that a
finite subset $R\subset \mathbb{R}^{d}\setminus\{0\}$ is called a root system,
if
\[
R\cap\R a=\{a, -a\}\quad \text{and}\quad \sigma_{a}R=R\ \text{for all}\ a\in R.
\]

Let
\[
v_{k}(x)=\prod_{a\in R_{+}}|\<a,x\>|^{2k(a)}
\]
be the Dunkl weight,
\[
c_{k}^{-1}=\int_{\mathbb{R}^{d}}e^{-|x|^{2}/2}v_{k}(x)\,dx,\quad d\mu_{k}(x)=c_{k}v_{k}(x)\,dx,
\]
and $L^{p}(\mathbb{R}^{d},d\mu_{k})$, $0<p<\infty$, be the space of
complex-valued Lebesgue measurable functions $f$ for which
\[
\|f\|_{p,d\mu_{k}}=\Bigl(\int_{\mathbb{R}^{d}}|f|^{p}\,d\mu_{k}\Bigr)^{1/p}<\infty.
\]
We also assume that $L^{\infty}\equiv C_b$ and $\|f\|_{\infty,d\mu_{k}}=\|f\|_{\infty}$.

\begin{example} If the root system $R$ is $\{\pm e_1,\dots,\pm e_d\}$, where
$\{e_1,\dots,e_d\}$ is an orthonormal basis of $\mathbb{R}^{d}$, then
$v_{k}(x)=\prod_{j=1}^d|x_j|^{2k_j}$, $k_{j}\ge 0$, $G=\mathbb{Z}_2^d$.
\end{example}

Let
\[
D_{j}f(x)=\frac{\partial f(x)}{\partial x_{j}}+
\sum_{a\in R_{+}}k(a)\<a,e_{j}\>\,\frac{f(x)-f(\sigma_{a}x)}{\<a,x\>},\quad
j=1,\dots,d
\]
be differential-differences Dunkl operators and $\Delta_k=\sum_{j=1}^dD_j^2$ be the
Dunkl Laplacian. The Dunkl kernel $e_{k}(x, y)=E_{k}(x, iy)$ is a unique
solution of the system
\[
D_{j}f(x)=iy_{j}f(x),\quad j=1,\dots,d,\qquad f(0)=1,
\]
and it plays the role of a generalized exponential function. Its properties are
similar to those of the classical exponential function $e^{i\<x, y\>}$. Several basic properties
follow from an integral representation % M.~R\"{o}sler
\cite{Ros99}:
\[
e_k(x,y)=\int_{\mathbb{R}^d}e^{i\<\xi,y\>}\,d\mu_x^k(\xi),
\]
where $\mu_x^k$ is a probability Borel measure,
whose support is contained in
\[
\mathrm{co}{}(\{gx\colon g\in G(R)\}),
\]
the convex hull of the $G$-orbit of $x$ in $\R^d$.
In particular, $|e_k(x,y)|\leq 1$.

For $f\in L^{1}(\mathbb{R}^{d},d\mu_{k})$, the Dunkl transform is defined by the equality
\[
\mathcal{F}_{k}(f)(y)=\int_{\mathbb{R}^{d}}f(x)\overline{e_{k}(x,
y)}\,d\mu_{k}(x).
\]
For $k\equiv0$, \ $\mathcal{F}_{0}$ is the classical Fourier transform $\mathcal{F}$. We
also note that $\mathcal{F}_k(e^{-|\,\cdot\,|^2/2})(y)=e^{-|y|^2/2}$ and
$\mathcal{F}_{k}^{-1}(f)(x)=\mathcal{F}_{k}(f)(-x)$.
Let
\begin{equation}\label{zvzv}
\mathcal{A}_k=\Big\{f\in L^{1}(\mathbb{R}^{d},d\mu_{k})\cap C_0(\R^d)\colon
\mathcal{F}_{k}(f)\in L^{1}(\mathbb{R}^{d},d\mu_{k})\Big\}.
\end{equation}
Let us now list several basic properties of the Dunkl transform.

\begin{proposition}
\textup{(1)} For $f\in L^{1}(\mathbb{R}^{d},d\mu_{k})$, $\mathcal{F}_{k}(f)\in C_0(\R^d)$.

\smallbreak
\textup{(2)} If $f\in\mathcal{A}_k$, we have the pointwise inversion formula
\[
f(x)=\int_{\mathbb{R}^{d}}\mathcal{F}_{k}(f)(y)e_{k}(x, y)\,d\mu_{k}(y).
\]

\textup{(3)} The Dunkl transform leaves the Schwartz space $\mathcal{S}(\R^d)$ invariant.

\smallbreak
\textup{(4)} The Dunkl transform extends to a unitary operator in $L^{2}(\mathbb{R}^{d},d\mu_{k})$.
\end{proposition}

Let $\lambda\ge -1/2$ and $J_{\lambda}(t)$ be the classical Bessel function of degree $\lambda$ and
\[
j_{\lambda}(t)=2^{\lambda}\Gamma(\lambda+1)t^{-\lambda}J_{\lambda}(t)
\]
be the normalized Bessel function. Set
\[
b_{\lambda}^{-1}=\int_{0}^{\infty}e^{-t^{2}/2}t^{2\lambda+1}\,dt=2^{\lambda}\Gamma(\lambda+1),\qquad
d\nu_{\lambda}(t)=b_{\lambda}t^{2\lambda+1}\,dt,\quad t\in \R_{+}.
\]
The norm in $L^{p}(\R_{+},d\nu_{\lambda})$, $1\le p<\infty$,
is given by
\[
\|f\|_{p,d\nu_{\lambda}}=\Bigl(\int_{\R_{+}}|f(t)|^{p}\,d\nu_{\lambda}(t)\Bigr)^{1/p}.
\]
Define
$\|f\|_{\infty}=\esssup_{t\in \R_{+}}|f(t)|$.

The Hankel transform is defined as follows
\[
\mathcal{H}_{\lambda}(f)(r)=\int_{\R_{+}}f(t)j_{\lambda}(rt)\,d\nu_{\lambda}(t),\quad r\in \mathbb{R}_{+}.
\]
It is a unitary operator in $L^{2}(\R_{+},d\nu_{\lambda})$ and
$\mathcal{H}_{\lambda}^{-1}=\mathcal{H}_{\lambda}$ \cite[Chap.~7]{BatErd53}.

Note that if $\lambda=d/2-1$, the Hankel transform is a restriction of the
Fourier transform on radial functions and if
$\lambda=\lambda_k=d/2-1+\sum_{a\in R_+}k(a)$ of the Dunkl transform.

Let $\mathbb{S}^{d-1}=\{x'\in\R^d\colon |x'|=1\}$ be the Euclidean sphere and
$d\sigma_k(x')=a_kv_k(x')\,dx'$ be the probability measure on $\mathbb{S}^{d-1}$.
We have
\begin{equation}
\int_{\mathbb{R}^{d}}f(x)\,d\mu_{k}(x)=
\int_{0}^{\infty}\int_{\mathbb{S}^{d-1}}f(tx')\,d\sigma_k(x')\,d\nu_{\lambda_k}(t).\label{eq1}
\end{equation}

We need the following partial case of the Funk--Hecke formula \cite{Xu00}
\begin{equation}
\int_{\mathbb{S}^{d-1}}e_{k}(x, ty')\,d\sigma_k(y')=j_{\lambda_k}(t|x|).\label{eq2}
\end{equation}

Throughout the paper, we will assume that $A\lesssim B$ means that $A\leq C B$
with a constant $C$ depending only on nonessential parameters.

\bigskip
\section{Generalized translation operators and convolutions}

Let $y\in \mathbb{R}^d$ be given. M. R\"{o}sler \cite{Ros98} defined a
generalized translation operator $\tau^y$ in $L^{2}(\mathbb{R}^{d},d\mu_{k})$
by the equation
\[
\mathcal{F}_k(\tau^yf)(z)=e_k(y, z)\mathcal{F}_k(f)(z).
\]
Since $|e_k(y, z)|\leq 1$ then $\|\tau^y\|_{2\to 2}\leq1$. If
$f\in\mathcal{A}_k$ (recall that $\mathcal{A}_k$ is given by \eqref{zvzv}), then, for any $x, y\in \R^d$,
\begin{equation}
\tau^yf(x)=\int_{\mathbb{R}^d}e_k(y, z)e_k(x, z)\mathcal{F}_k(f)(z)\,d\mu_k(z).\label{eq3}
\end{equation}
Note that $\mathcal{S}(\R^d)\subset\mathcal{A}_k\subset
L^{2}(\mathbb{R}^{d},d\mu_{k})$. K. Trim\`{e}che \cite{Tri02} extended the operator $\tau^y$
on $C^{\infty}(\R^d)$.

The explicit expression of $\tau^yf$ is known only in the case of the
reflection group $\Z_2^d$. In particular, in this case $\tau^yf$ is not a positive operator \cite{Ros95}.
Note that in the case of symmetric group $S_d$ the operator
$\tau^yf$ is also not positive \cite{ThaXu05}.

It remains an open question whether $\tau^yf$ is an $L^p$ bounded operator on $\mathcal{S}(\R^d)$ for
$p\neq 2$.
It is known (\cite{Ros95,ThaXu05}) only
  for $G=\Z_2^d$.
 Note that a positive answer would follow from the $L^1$-boundedness.

Let
\[
\lambda_k=d/2-1+\sum_{a\in R_+}k(a).
\]
We have $\lambda_k\geq-1/2$ and, moreover, $\lambda_k=-1/2$ only if $d=1$ and
$k\equiv0$. In what follows we assume that $\lambda_k>-1/2$.

Define another generalized translation operator $T^t\colon
L^{2}(\mathbb{R}^{d},d\mu_{k})\to L^{2}(\mathbb{R}^{d},d\mu_{k})$, $t\in \mathbb{R}$,
by the relation
\[
\mathcal{F}_k(T^tf)(y)=j_{\lambda_k}(t|y|)\mathcal{F}_k(f)(y).
\]
Since $|j_{\lambda_k}(t)|\leq 1$, it is a bounded operator such that
$\|T^t\|_{2\to 2}\leq1$ and
\[
T^tf(x)=\int_{\mathbb{R}^d}j_{\lambda_k}(t|y|)e_k(x,y)\mathcal{F}_k(f)(y)\,d\mu_k(y).
\]
This gives $T^t=T^{-t}$. If $f\in \mathcal{A}_k$, then from \eqref{eq2} and \eqref{eq3} we have (pointwise)
\begin{equation}
T^tf(x)=\int_{\mathbb{R}^d}j_{\lambda_k}(t|y|)e_k(x,
y)\mathcal{F}_k(f)(y)\,d\mu_k(y)=\int_{\mathbb{S}^{d-1}}\tau^{ty'}f(x)\,d\sigma_k(y').\label{eq4}
\end{equation}
Note that the operator $T^t$ is self-adjoint. Indeed, if $f, g\in \mathcal{A}_k$, then
\begin{align*}
\int_{\mathbb{R}^d}T^tf(x)\,g(x)\,d\mu_k(x)&=\int_{\mathbb{R}^d}\int_{\mathbb{R}^d}j_{\lambda_k}(t|y|)e_k(x,
y)\mathcal{F}_k(f)(y)\,d\mu_k(y)\,g(x)\,d\mu_k(x)\\
&=\int_{\mathbb{R}^d}j_{\lambda_k}(t|y|)\mathcal{F}_k(f)(y)\mathcal{F}_k(g)(-y)\,d\mu_k(y)\\
&=\int_{\mathbb{R}^d}j_{\lambda_k}(t|y|)\mathcal{F}_k(g)(y)\mathcal{F}_k(f)(-y)\,d\mu_k(y)\\
&=\int_{\mathbb{R}^d}f(x)\,T^tg(x)\,d\mu_k(x).
\end{align*}

M. R\"{o}sler \cite{Ros03} proved that the spherical mean (with respect to the Dunkl weight) of the operator $\tau^y$, i.e.,
$\int_{\mathbb{S}^{d-1}}\tau^{ty'}f(x)\,d\sigma_k(y')$,
  is a positive operator on $C^{\infty}(\R^d)$ and obtained its integral representation. This implies that
$T^t$ is a positive operator on $C^{\infty}(\R^d)$ and, moreover, for any $t\in\R$, $x\in \R^d$,
\begin{equation}
T^tf(x)=\int_{\mathbb{R}^d}f(z)\,d\sigma_{x,t}^k(z), \label{eq5}
\end{equation}
{ where $\sigma_{x,t}^k$ is a probability Borel measure,
\begin{equation}
\supp \sigma_{x,t}^k\subset\bigcup\limits_{g\in G}\{z\in\R^d\colon |z-gx|\leq
t\} \label{eq5+}
\end{equation}
and the mapping $(x,
t)\to\sigma_{x,t}^k$ is continuous with respect to the weak topology on
probability measures.}

The representation \eqref{eq5} gives a natural extension
of the operator $T^t$ on $C_b(\R^d)$, namely, for $f\in C_b(\R^d)$ we define $T^tf(x)\in C(\R\times\R^d)$ by \eqref{eq5} and, moreover, the estimate
 $\|T^tf\|_{\infty}\leq \|f\|_{\infty}$
 holds.

Note that for
$k\equiv0$, $T^t$ is the usual spherical mean
\begin{equation}
T^tf(x)=S^tf(x)=\int_{\mathbb{S}^{d-1}}f(x+ty')\,d\sigma_0(y').\label{s-operator}
\end{equation}

\begin{theorem}%\label{thm3.1}
If $1\leq p\leq\infty$, then, for any $t\in\R$ and $f\in\mathcal{S}(\R^d)$,
\begin{equation}
\|T^tf\|_{p,d\mu_k}\leq \|f\|_{p,d\mu_k}.\label{eq6}
\end{equation}
\end{theorem}

\begin{remark}
(i) The inequality $\|T^tf\|_{p,d\mu_k}\leq c\|f\|_{p,d\mu_k}$ was proved in \cite{ThaXu05} for $G=\Z_2^d$.

\smallbreak
(ii) Since $\mathcal{S}(\R^d)$ is dense in $L^{p}(\mathbb{R}^{d},d\mu_{k})$, $1\leq p<\infty$, then for any $t\in \R_+$ the operator $T^t$
can be defined on $L^{p}(\mathbb{R}^{d},d\mu_{k})$ and estimate \eqref{eq6} holds.

\smallbreak
(iii) If $d=1$, $v_k(x)=|x|^{2\lambda+1}$, $\lambda>-1/2$, inequality
\eqref{eq6} was proved in \cite{Che09}. In this case the integral
representation of $T^t$ is of the form:
\begin{equation*}
T^tf(x)=\frac{c_{\lambda}}{2}\int_0^{\pi}\{f(A)(1+B)+f(-A)(1-B)\}\sin^{2\lambda}\varphi\,d\varphi,
%\label{c-lambda-}
\end{equation*}
where, for $(x,t)\neq (0,0)$,
\begin{equation}
c_{\lambda}=\frac{\Gamma(\lambda+1)}{\sqrt{\pi}\Gamma(\lambda+1/2)},\quad
A=\sqrt{x^2+t^2-2xt\cos\varphi},\quad B=\frac{x-t\cos\varphi}{A}.
\label{c-lambda}
\end{equation}
If $\lambda=-1/2$, i.e., $k\equiv0$, then $T^tf(x)=\frac12 \big(f(x+t)+f(x-t)\big)$.
\end{remark}

\begin{proof}
Let $t\in\R_+$ be given and the operator $T^t$ be defined on $\mathcal{S}(\R^d)$ by \eqref{eq5}. % or by \eqref{eq4} the same.
Using \eqref{eq4}, we have
\[
\sup\{\|T^tf\|_{2}\colon f\in \mathcal{S}(\R^d),\, \|f\|_{2}\leq 1\}\leq1
\]
and $T^t$ can be extended to the space $L^{2}(\mathbb{R}^{d},d\mu_{k})$ with
preservation of norm, moreover, this extension coincides with \eqref{eq4}. Moreover, \eqref{eq5} yields
\begin{equation}
\sup\{\|T^tf\|_{\infty}\colon f\in \mathcal{S}(\R^d),\, \|f\|_{\infty}\leq
1\}\leq1. \label{eq7}
\end{equation}
Since the operator $T^t$ is self-adjoint, then by \eqref{eq7}
\begin{align*}
&\sup\{\|T^tf\|_{1,d\mu_k}\colon f\in \mathcal{S}(\R^d),\, \|f\|_{1,d\mu_k}\leq
1\}\\
&\qquad =\sup\Bigl\{\int_{\mathbb{R}^d}T^tf\,g\,d\mu_k\colon f, g\in \mathcal{S}(\R^d),\
\|f\|_{1,d\mu_k}\leq 1,\ \|g\|_{\infty}\leq
1\Bigr\}\\
&\qquad =\sup\Bigl\{\int_{\mathbb{R}^d}f\,T^tg\,d\mu_k\colon f, g\in \mathcal{S}(\R^d),\
\|f\|_{1,d\mu_k}\leq 1,\ \|g\|_{\infty}\leq
1\Bigr\}\\
&\qquad =\sup\{\|T^tg\|_{\infty}\colon g\in \mathcal{S}(\R^d),\ \|g\|_{\infty}\leq 1\}\leq1.
\end{align*}
Hence, $T^t$ can be extended to $L^{1}(\mathbb{R}^{d},d\mu_{k})$ with
preservation of the norm such that this extension coincides with \eqref{eq4} on
$L^{1}(\mathbb{R}^{d},d\mu_{k})\cap L^{2}(\mathbb{R}^{d},d\mu_{k})$.

By the Riesz--Thorin interpolation theorem we obtain
\[
\sup\{\|T^tf\|_{p,d\mu_k}\colon f\in \mathcal{S}(\R^d),\ \|f\|_{p,d\mu_k}\leq
1\}\leq1,\quad 1\leq p\leq 2.
\]
Let $2<p<\infty$, $1/p+1/p'=1.$ As for $p=1$ we get
\begin{align*}
&\sup\{\|T^tf\|_{p,d\mu_k}\colon f\in \mathcal{S}(\R^d),\, \|f\|_{p,d\mu_k}\leq
1\}\\
&\qquad =\sup\{\|T^tg\|_{p',d\mu_k}\colon g\in \mathcal{S}(\R^d),\, \|g\|_{p',d\mu_k}\leq 1\}\leq1.
\qedhere
\end{align*}
\end{proof}

For any $f_0\in L^{p}(\R_{+},d\nu_{\lambda})$, $1\leq p\leq \infty$,
$\lambda>-1/2$, let us define the Gegenbauer-type translation operator (see,
e.g., \cite{Pla07,Pla09})
\begin{equation*}
R^tf_0(r)=c_\lambda\int_{0}^{\pi}f_0(\sqrt{r^2+t^2-2rt\cos\varphi})\sin^{2\lambda}\varphi\,d\varphi,
%\label{eq8}
\end{equation*}
where $c_\lambda$ is defined by \eqref{c-lambda}. We have that $\|R^t\|_{p\to
p}\leq 1$ and
$\mathcal{H}_{\lambda}(R^tf_0)(r)=j_\lambda(tr)\mathcal{H}_{\lambda}(f_0)(r)$,
where $f_0\in \mathcal{S}(\R_+)$. Taking into account \eqref{eq2} and
\eqref{eq4}, we note that for $\lambda=\lambda_k$ the operator $R^t$ is a
restriction of $T^t$ on radial functions, that is, for $f_0\in
L^{p}(\R_{+},d\nu_{\lambda_k})$,
\begin{equation*}
T^tf_0(|x|)=R^tf_0(r),\quad r=|x|.%\label{eq9}
\end{equation*}

We also mention the following useful properties of the generalized translation
operator $T^t$.

\begin{lemma} Let $t\in\R$.

\smallbreak
\textup{(1)} If $f\in L^{1}(\mathbb{R}^{d},d\mu_{k})$, then
$\int_{\mathbb{R}^d}T^tf\,d\mu_k=\int_{\mathbb{R}^d}f\,d\mu_k$.

\smallbreak
\textup{(2)} Let $r>0$, $f\in L^{p}(\mathbb{R}^{d},d\mu_{k})$, $1\leq
p<\infty$. If $\supp f\subset B_r$, then $\supp T^tf\subset B_{r+|t|}$. If
$\supp f\subset \R^d\setminus B_r$, $r>|t|$, then $\supp T^tf\subset
\R^d\setminus B_{r-|t|}$.
\end{lemma}

\begin{proof}
Due to the $L^p$-boundedness of $T^t$ and density of $\mathcal{S}(\R^d)$ in
$L^{p}(\mathbb{R}^{d},d\mu_{k})$ we can assume that $f\in \mathcal{S}(\R^d)$.

\smallbreak
(1) Let $s>0$. By integral representation of $j_{\lambda_k}(z)$ (see, e.g.,
\cite[Sect.~7.12]{BatErd53}) we have
\begin{align*}
T^t(e^{-s|\Cdot|^2})(x)&=R^t(e^{-s(\cdot)^2})(|x|)=
c_{\lambda_k}\int_{0}^{\pi}e^{-s(|x|^2+t^2-2|x|t\cos\varphi)}\sin^{2\lambda_k}\varphi\,d\varphi\\
&=e^{-s(|x|^2+t^2)}c_{\lambda_k}\int_{0}^{\pi}e^{2s|x|t\cos\varphi}\sin^{2\lambda_k}\varphi\,d\varphi
=e^{-s(|x|^2+t^2)}j_{\lambda_k}(2is|x|t),
\end{align*}
and,
in particular,
\[
T^t(e^{-s|\Cdot|^2})(x)\leq e^{-s(|x|^2+t^2)}e^{2s|x|t}=e^{-s(|x|-t)^2}\leq 1.
\]
Using the self-adjointness of $T^t$, we obtain
\[
\int_{\mathbb{R}^d}T^tf(x)\,e^{-s|x|^2}\,d\mu_k(x)=\int_{\mathbb{R}^d}f(x)T^t(e^{-s|\Cdot|^2})(x)\,d\mu_k(x).
\]
Since for any $t\in \R$, $x\in \R^d$,
\[
\lim\limits_{s\to 0}e^{-s|x|^2}=\lim\limits_{s\to 0}T^t(e^{-s|\Cdot|^2})(x)=1,
\]
then by Lebesgue's dominated convergence theorem we derive (1).

\smallbreak
(2) If $\supp f\subset B_r$ and $|x|>r+|t|$, then, in light of \eqref{eq5+} and
\eqref{eq5}, for $z\in \supp \sigma_{x,t}^k$ and $g\in G$, we have that
$$|z|\geq |gx|-|z-gx|=|x|-|z-gx|>r$$ and $f(z)=0$, which yields $T^tf(x)=0$.

If $\supp f\subset \R^d\setminus B_r$, $|x|<r-|t|$, then, for $z\in \supp
\sigma_{x,t}^k$ and $g\in G$, we similarly obtain $|z|\leq
|gx|+|z-gx|=|x|+|z-gx|<r$, $f(z)=0$, and $T^tf(x)=0$.
\end{proof}

Let $g$ be a radial function, $g(y)=g_0(|y|)$, where $g_0(t)$ is defined on
$\R_+$. Note that by virtue of \eqref{eq1}
\begin{equation}
\|g\|_{p,d\mu_k}=\|g_0\|_{p,d\nu_{\lambda_k}},\quad
\mathcal{F}_k(g)(y)=\mathcal{H}_{\lambda_k}(g_0)(|y|).\label{eq10}
\end{equation}

By means of operators $T^t$ and $\tau^y$ define two convolution operators
\begin{equation}
(f\Ast{\lambda_k}g_0)(x)=\int_0^{\infty}T^tf(x)g_0(t)\,d\nu_{\lambda_k}(t),\label{eq11}
\end{equation}
\begin{equation}
(f\Ast{k}g)(x)=\int_{\R^d}f(y)\tau^{x}g(-y)\,d\mu_{k}(y).\label{eq12}
\end{equation}
Note that operator \eqref{eq11} was defined in \cite{ThaXu05}, while \eqref{eq12}
was investigated in \cite{ThaXu05,Tri02}.

%Let $X_\mathrm{rad}(\R^d)$ be the subspace of $X(\R^d)$ consisting of radial functions.
S.~Thangavelu and Yu.~Xu \cite{ThaXu05} proved that if $f\in
L^{p}(\mathbb{R}^{d},d\mu_{k})$, $1\leq p\leq\infty$, $g\in L^{1}_\mathrm{rad}(\mathbb{R}^{d},d\mu_{k})$, and $g$ is bounded,
then
\begin{equation}
\|(f\Ast{k}g)\|_{p,d\mu_k}\leq \|f\|_{p,d\mu_k}\|g\|_{1,d\mu_k},\label{eq13}
\end{equation}
and if $1\leq p\leq 2$, $g\in L^{p}_\mathrm{rad}(\mathbb{R}^{d},d\mu_{k})$, then, for any $y\in\R^d$,
\begin{equation}
\|\tau^{y}g\|_{p,d\mu_k}\leq \|g\|_{p,d\mu_k}.\label{eq14}
\end{equation}

Note that additional condition of boundedness $g$ in \eqref{eq13} can be omitted. Indeed, by
H\"{o}lder's inequality
\begin{align*}
|(f\Ast{k}g)(x)|&=\Bigl|\int_{\mathbb{R}^d}f(y)\tau^{x}g(-y)\,d\mu_{k}(y)\Bigr|\\&\leq
\Bigl(\int_{\mathbb{R}^d}|f(y)|^p|\tau^{x}g(-y)|\,d\mu_{k}(y)\Bigr)^{1/p}
\Bigl(\int_{\mathbb{R}^d}|\tau^{x}g(-y)|\,d\mu_{k}(y)\Bigr)^{1-1/p},
\end{align*}
and by \eqref{eq14} for $p=1$ we get
\begin{align*}
\|(f\Ast{\lambda_k}g_0)\|_{p,d\nu_{\lambda_k}}&\leq
\Bigl(\int_{\mathbb{R}^d}\int_{\mathbb{R}^d}|f(y)|^p|\tau^{x}g(-y)|\,d\mu_{k}(y)\,d\mu_k(x)\Bigr)^{1/p}\|g\|_{1,d\mu_{k}}^{1-1/p}\\&
=\Bigl(\int_{\mathbb{R}^d}|f(y)|^p\int_{\mathbb{R}^d}|\tau^{-y}g(x)|\,d\mu_{k}(x)\,d\mu_k(y)\Bigr)^{1/p}\|g\|_{1,d\mu_{k}}^{1-1/p}
\\&\leq\|f\|_{p,d\mu_k}\|g\|_{1,d\mu_{k}}.
\end{align*}

\begin{lemma}\label{lem3.2}
If $f\in \mathcal{A}_k$, $g_0\in L^{1}(\R_{+},d\nu_{\lambda_k})$,
$g(y)=g_0(|y|)$, then, for any $x, y\in\R^d$,
\begin{equation}
(f\Ast{\lambda_k}g_0)(x)=(f\Ast{k}g)(x)=\int_{\R^d}\tau^{-y}f(x)g(y)\,d\mu_k(y), \label{eq15}
\end{equation}
\begin{equation}
\mathcal{F}_k(f\Ast{\lambda_k}g_0)(y)=\mathcal{F}_k(f\Ast{k}g)(y)=\mathcal{F}_k(f)(y)\mathcal{F}_k(g)(y).\label{eq16}
\end{equation}
\end{lemma}
\begin{proof}
Using \eqref{eq4} and \eqref{eq10}, we get
\begin{align*}
(f\Ast{\lambda_k}g_0)(x)&=\int_0^{\infty}T^tf(x)g_0(t)\,d\nu_{\lambda_k}(t)\\
&=\int_0^{\infty}\int_{\mathbb{R}^d}j_{\lambda_k}(t|y|)e_k(x,
y)\mathcal{F}_k(f)(y)\,d\mu_k(y)g_0(t)\,d\nu_{\lambda_k}(t)\\
&=\int_{\mathbb{R}^d}e_k(x, y)\mathcal{F}_k(f)(y)\mathcal{F}_k(g)(y)\,d\mu_k(y),
\end{align*}
which gives
\[
\mathcal{F}_k(f\Ast{\lambda_k}g_0)(y)=\mathcal{F}_k(f)(y)\mathcal{F}_k(g)(y).
\]
If $g\in \mathcal{A}_k$, then, by \eqref{eq3},
\begin{align*}
(f\Ast{k}g)(x)&=\int_{\R^d}f(y)\tau^{x}g(-y)\,d\mu_{k}(y)\\
&=\int_{\R^d}f(y)\int_{\mathbb{R}^d}e_k(-y, z)e_k(x,
z)\mathcal{F}_k(g)(z)\,d\mu_k(z)\,d\mu_{k}(y)\\
&=\int_{\mathbb{R}^d}e_k(x, z)\mathcal{F}_k(f)(z)\mathcal{F}_k(g)(z)\,d\mu_k(z).
\end{align*}
and hence the first equality in \eqref{eq15} and the second equality in
\eqref{eq16} are valid for $g\in \mathcal{A}_k$.

Assuming that $g_0\in L^{1}(\R_{+},d\nu_{\lambda})$, $(g_n)_0\in
\mathcal{S}(\R_+)$, $g_{n}\to g$ in $L^{1}(\mathbb{R}^{d},d\mu_{k})$, and taking
into account \eqref{eq7}--\eqref{eq12} and \eqref{eq14}, we
arrive at
\begin{align*}
\bigl|(f\Ast{\lambda_k}g_0)(x)-(f\Ast{k}g)(x)\bigr|&\leq\bigl|(f\Ast{\lambda_k}(g_0-(g_n)_0))(x)\bigr|
+\bigl|(f\Ast{k}(g-g_n))(x)\bigr|\\
&\leq 2\|f\|_{\infty}\,\|g-g_n\|_{1, d\mu_k}.
\end{align*}
Thus, the first equality in \eqref{eq15} holds.

Finally, using \eqref{eq3}, we get
\begin{align*}
\int_{\R^d}\tau^{-y}f(x)g(y)\,d\mu_{k}(y)
&=\int_{\R^d}g(y)\int_{\mathbb{R}^d}e_k(-y, z)e_k(x,
z)\mathcal{F}_k(f)(z)\,d\mu_k(z)\,d\mu_{k}(y)\\
&=\int_{\mathbb{R}^d}e_k(x, z)\mathcal{F}_k(f)(z)\mathcal{F}_k(g)(z)\,d\mu_k(z).
\end{align*}
and the second part in \eqref{eq15} is valid.
\end{proof}

Let $y\in \R^d$ be given. M. R\"{o}sler \cite{Ros03} proved that the operator
$\tau^y$ is positive
on
$C^{\infty}_\mathrm{rad}(\R^d)$, i.e., \ $\tau^y\geq 0$, and moreover,
 for
any $x\in \R^d$,
\begin{equation}
\tau^yf(x)=\int_{\mathbb{R}^d}f(z)\,d\rho_{x,y}^k(z), \label{eq17}
\end{equation}
where $\rho_{x,y}^k$ is a radial probability Borel measure such that $\supp
\rho_{x,y}^k\subset B_{|x|+|y|}$.

\begin{theorem}\label{thm3.3}
If $1\leq p\leq\infty$, then, for any $x\in\R^d$ and $f\in\mathcal{S}(\R^d)$,
\begin{equation}
\|T^tf(x)\|_{p,d\nu_{\lambda_k}}=
\Bigl(\int_{\R_{+}}|T^tf(x)|^{p}\,d\nu_{\lambda}(t)\Bigr)^{1/p}\leq
\|f\|_{p,d\mu_k}.\label{eq18}
\end{equation}
\end{theorem}
\begin{proof} Let $x\in\R^d$ be given. Let an operator $B^x$ be defined on $\mathcal{S}(\R^d)$ as follows (cf. \eqref{eq4} and \eqref{eq5}): for $f\in\mathcal{S}(\R^d)$,
\[
B^xf(t)=T^tf(x)=\int_{\mathbb{R}^d}j_{\lambda_k}(t|y|)e_k(x,
y)\mathcal{F}_k(f)(y)\,d\mu_k(y)=\int_{\mathbb{R}^d}f(z)\,d\sigma_{x,t}^k(z).
\]

Let $p=2$. We have
\[
T^tf(x)=\int_{0}^{\infty}j_{\lambda_k}(tr)\int_{\mathbb{S}^{d-1}}e_k(x,
ry')\mathcal{F}_k(f)(ry')\,d\sigma_k(y')\,d\nu_{\lambda_k}(r)
\]
and
\[
\mathcal{H}_{\lambda_k}(T^tf(x))(r)=\int_{\mathbb{S}^{d-1}}e_k(x, ry')\mathcal{F}_k(f)(ry')\,d\sigma_k(y').
\]
This, H\"{o}lder's inequality, and the fact that the operators $\mathcal{H}_{\lambda_k}$ and $\mathcal{F}_{k}$ are unitary imply
\begin{align*}
\|T^tf(x)\|_{2,d\nu_{\lambda_k}}^2&=\|\mathcal{H}_{\lambda_k}(T^tf(x))(r)\|_{2,d\nu_{\lambda_k}}^2\\
&=\int_{0}^{\infty}\Bigl|\int_{\mathbb{S}^{d-1}}e_k(x, ry')\mathcal{F}_k(f)(ry')\,d\sigma_k(y')
\Bigr|^2\,d\nu_{\lambda_k}(r)\\
&\leq
\int_{0}^{\infty}\int_{\mathbb{S}^{d-1}}|\mathcal{F}_k(f)(ry')|^2\,d\sigma_k(y')\,d\nu_{\lambda_k}(r)\\
&=\|\mathcal{F}_k(f)\|_{2,d\mu_k}^2=\|f\|_{2,d\mu_k}^2,
\end{align*}
which yields inequality \eqref{eq18} for $p=2$. Moreover, $B^x$ can be extended to the space $L^{2}(\mathbb{R}_+,d\nu_{\lambda_k})$ with preservation of norm, and, moreover, this extension coincides with \eqref{eq4}.
 %and $T^t$ can be extended to the space $L^{2}(\mathbb{R}^{d},d\mu_{k})$ with preservation of norm, moreover, this extension coincides with %\eqref{eq4}.

Let $p=1$. By \eqref{eq15} and \eqref{eq17}, we obtain
\begin{align*}
\|T^tf(x)\|_{1,d\nu_{\lambda_k}}&=\sup\Bigl\{\int_0^{\infty}T^tf(x)g_0(t)\,d\nu_{\lambda_k}(t)\colon g_0\in
\mathcal{S}(\R_+),\ \|g_0\|_{\infty}\leq 1\Bigr\}\\
&=\sup\Bigl\{\int_{\R^d}f(y)\tau^{x}g(-y)\,d\mu_{k}(y)\colon g\in
\mathcal{S}_\mathrm{rad}(\R^d),\ \|g\|_{\infty}\leq 1\Bigr\}\\
&\leq \|f\|_{1,d\mu_k}\sup\bigl\{\|\tau^{x}g(-y)\|_{\infty}\colon g\in
\mathcal{S}_\mathrm{rad}(\R^d),\ \|g\|_{\infty}\leq 1\bigr\}\leq\|f\|_{1,d\mu_k},
\end{align*}
which is the desired inequality \eqref{eq18} for $p=1$. Moreover, $B^x$ can be extended to $L^{1}(\mathbb{R}_+,d\nu_{\lambda_k})$ with preservation of norm such that the extension coincides with \eqref{eq4} on $L^{1}(\mathbb{R}_+,d\nu_{\lambda_k})\cap L^{2}(\mathbb{R}_+,d\nu_{\lambda_k})$.

By the Riesz--Thorin interpolation theorem we obtain \eqref{eq18} for $1<p<2$.

If $2<p<\infty$, $1/p+1/p'=1$, then %$1<p'<2$ and
 by \eqref{eq15} and \eqref{eq14},
\begin{align*}
\|T^tf(x)\|_{p,d\nu_{\lambda_k}}&=\sup\Bigl\{\int_0^{\infty}T^tf(x)g_0(t)\,d\nu_{\lambda_k}(t)\colon g_0\in
\mathcal{S}(\R_+),\ \|g_0\|_{p',d\nu_{\lambda_k}}\leq 1\Bigr\}\\
&=\sup\Bigl\{\int_{\R^d}f(y)\tau^{x}g(-y)\,d\mu_{k}(y)\colon g\in
\mathcal{S}_\mathrm{rad}(\R^d),\ \|g\|_{p',d\mu_k}\leq 1\Bigr\}\\
&\leq \|f\|_{p,d\mu_k}\sup\{\|\tau^{x}g(-y)\|_{p',d\mu_k}\colon g\in
\mathcal{S}_\mathrm{rad}(\R^d),\ \|g\|_{p',d\mu_k}\leq 1\}\\
&\leq \|f\|_{p,d\mu_k}.
\end{align*}

Finally, for $p=\infty$, \eqref{eq18} follows from representation \eqref{eq5}.
\end{proof}

We are now in a position to prove the Young inequality for the convolutions \eqref{eq11} and \eqref{eq12}.
\begin{theorem}%\label{thm3.4}
Let $1\leq p, q\leq\infty$, $\frac{1}{p}+\frac{1}{q}\geq 1$, and
$\frac{1}{r}=\frac{1}{p}+\frac{1}{q}-1$. We have that, for any $f\in\mathcal{S}(\R^d)$,
$g_0\in\mathcal{S}(\R_+)$ and $g\in\mathcal{S}_\mathrm{rad}(\R^d)$,
\begin{equation}
\|(f\Ast{\lambda_k}g_0)\|_{r,d\nu_{\lambda_k}}\leq
\|f\|_{p,d\mu_k}\|g_0\|_{q,d\nu_{\lambda_k}},\label{eq19}
\end{equation}
\begin{equation}
\|(f\Ast{k}g)\|_{r,d\mu_k}\leq \|f\|_{p,d\mu_k}\|g\|_{q,d\mu_k}.\label{eq20}
\end{equation}
\end{theorem}
\begin{proof}
Since for $g(y)=g_0(|y|)$ we have
\[
\|(f\Ast{\lambda_k}g_0)\|_{r,d\nu_{\lambda_k}}=\|(f\Ast{k}g)\|_{r,d\mu_k},\quad
\|g_0\|_{q,d\nu_{\lambda_k}}=\|g\|_{q,d\mu_k},
\]
it is enough to show inequality \eqref{eq19}. The proof is straightforward using
H\"{o}lder's inequality and estimates \eqref{eq6} and \eqref{eq18}.
For the sake of completeness, we give it here.
Let
$\frac{1}{\mu}=\frac{1}{p}-\frac{1}{r}$ and
$\frac{1}{\nu}=\frac{1}{q}-\frac{1}{r}$, then $\frac{1}{\mu}\geq 0$,
$\frac{1}{\nu}\geq 0$ and $\frac{1}{r}+\frac{1}{\mu}+\frac{1}{\nu}=1$.
In virtue of \eqref{eq18}, we have
% Using \eqref{eq6}, \eqref{eq18} and the H\"{o}lder inequality we get
\begin{align*}
&\Bigl|\int_0^{\infty}T^tf(x)g_0(t)\,d\nu_{\lambda_k}(t)\Bigr|\leq
\Bigl(\int_0^{\infty}|T^tf(x)|^p|g_0(t)|^q\,d\nu_{\lambda_k}(t)\Bigr)^{1/r}\\
&\qquad \times \Bigl(\int_0^{\infty}|T^tf(x)|^p\,d\nu_{\lambda_k}(t)\Bigr)^{1/\mu}
\Bigl(\int_0^{\infty}|g_0(t)|^q\,d\nu_{\lambda_k}(t)\Bigr)^{1/\nu}\\
&\qquad \leq\Bigl(\int_0^{\infty}|T^tf(x)|^p|g_0(t)|^q\,d\nu_{\lambda_k}(t)\Bigr)^{1/r}
\|f\|_{p,d\mu_k}^{p/\mu}\|g_0\|_{q,d\nu_{\lambda_k}}^{q/\nu}.
\end{align*}
 Using \eqref{eq6}, this gives
\begin{align*}
&\|(f\Ast{\lambda_k}g_0)\|_{r,d\nu_{\lambda_k}}\leq
\Bigl(\int_{\R^d}\int_0^{\infty}|T^tf(x)|^p|g_0(t)|^q\,d\nu_{\lambda_k}(t)\,d\mu_k(x)\Bigr)^{1/r}\\
&\qquad \times \|f\|_{p,d\mu_k}^{p/\mu}\|g_0\|_{q,d\nu_{\lambda_k}}^{q/\nu}\leq\|f\|_{p,d\mu_k}\|g_0\|_{q,d\nu_{\lambda_k}}.
\qedhere
\end{align*}
\end{proof}

\begin{theorem}%\label{thm3.5}
Let $1\leq p\leq\infty$ and $g\in\mathcal{S}_\mathrm{rad}(\R^d)$. We have that, for any $y\in\R^d$,
\begin{equation}
\|\tau^{y}g\|_{p,d\mu_k}\leq \|g\|_{p,d\mu_k}.\label{eq21}
\end{equation}
\end{theorem}
\begin{remark}
Since $\mathcal{S}(\R^d)$ is dense in $L^{p}(\mathbb{R}^{d},d\mu_{k})$,
$1\leq p<\infty$, the operator $\tau^{y}$ can be defined on
$L^{p}_\mathrm{rad}(\mathbb{R}^{d},d\mu_{k})$ so that
\eqref{eq21} holds.
\end{remark}

\begin{proof}
In the case $1\leq p\le 2$ this result was proved in \cite{ThaXu05}. The case $p=\infty$
follows from \eqref{eq17}.

Let $2<p<\infty$. Since $\mathcal{F}_k(g)$ is a radial function and
\begin{align*}
\tau^{y}g(-x)&=\int_{\mathbb{R}^d}e_k(y, z)e_k(-x,
z)\mathcal{F}_k(g)(z)\,d\mu_k(z)\\
&=\int_{\mathbb{R}^d}e_k(-y, z)e_k(x, z)\mathcal{F}_k(g)(z)\,d\mu_k(z)=\tau^{-y}g(x),
\end{align*}
then using \eqref{eq20} for $r=\infty$, $q=p$ we obtain
\begin{align*}
\|\tau^{-y}g\|_{p,d\mu_k}&=\sup \Bigl\{\int_{\R^d}\tau^{-y}g(x)f(x)\,d\mu_{k}(x)\colon f\in
\mathcal{S}(\R^d),\ \|f\|_{p',d\mu_k}\leq 1\Bigr\}\\
&\leq \sup\{\|(f\Ast{k}g)(y)\|_{\infty,d\mu_k}\colon f\in
\mathcal{S}(\R^d),\ \|f\|_{p',d\mu_k}\leq 1\}\leq\|g\|_{p,d\mu_k}.
\qedhere
\end{align*}
\end{proof}

Now we give an analogue of Lemma \ref{lem3.2} for the case when $f\in L^{p}$.
\begin{lemma}\label{lem3.6}
Let $1\leq p\leq\infty$, $f\in L^{p}(\mathbb{R}^{d},d\mu_{k})\cap C_b(\R^d)\cap
C^{\infty}(\R^d)$, $g_0\in \mathcal{S}(\R_{+})$, and $g(y)=g_0(|y|)$. Then, for any
$x\in\R^d$,
\begin{equation}
(f\Ast{\lambda_k}g_0)(x)=(f\Ast{k}g)(x)\in L^{p}(\mathbb{R}^{d},d\mu_{k})\cap C_b(\R^d)\cap C^{\infty}(\R^d) \label{eq22}
\end{equation}
and, in the sense of tempered distributions,
\begin{equation}
\mathcal{F}_k(f\Ast{\lambda_k}g_0)=\mathcal{F}_k(f\Ast{k}g)=\mathcal{F}_k(f)\mathcal{F}_k(g).\label{eq23}
\end{equation}
\end{lemma}
\begin{proof}
First, in light of \eqref{eq6} and \eqref{eq19}, we
note that the convolution \eqref{eq11} belongs to $L^{p}(\mathbb{R}^{d},d\mu_{k})$. Moreover, \eqref{eq5} implies that it is
in $C_b(\R^d)$.

Taking into account that
$g\in \mathcal{S}(\R^d)$ and $
(-\Delta_k)^re_k(\Cdot, z)=|z|^{2r}e_k(\Cdot, z),
$
we have
\[
(-\Delta_k)^r(f\Ast{k}g)(x)=\int_{\R^d}f(y)\int_{\R^d}e_k(x, z)e_k(-y,
z)|z|^{2r}\mathcal{F}_k(g)(z)\,d\mu_{k}(z)\,d\mu_{k}(y).
\]
%convolution \eqref{eq12} belongs to $C^{\infty}(\R^d)$, since the integral converges uniformly.
Let us show that the integral converges uniformly in $x$.
 We have
\[
\int_{\R^d}e_k(x, z)e_k(-y, z)|z|^{2r}\mathcal{F}_k(g)(z)\,d\mu_{k}(z)=\tau^xG(-y),
\]
where $G\in \mathcal{S}_\mathrm{rad}(\R^d)$ is such that $\mathcal{F}_k(G)(z)=|z|^{2r}\mathcal{F}_k(g)(z)$. Using H\"{o}lder's inequality and \eqref{eq21},
we get
\begin{align*}
&\Bigl|\int_{\R^d}f(y)\int_{\R^d}e_k(x, z)e_k(-y,
z)|z|^{2r}\mathcal{F}_k(g)(z)\,d\mu_{k}(z)\,d\mu_{k}(y)\Bigr|
\\ &\qquad=\Bigl|\int_{\R^d}f(y)\tau^xG(-y)\,d\mu_{k}(y)\Bigr|\leq \|f\|_{p, d\mu_{k}}\|\tau^xG\|_{p', d\mu_{k}}\leq \|f\|_{p, d\mu_{k}}\|G\|_{p', d\mu_{k}}.
\end{align*}
Thus,
 convolution \eqref{eq12} belongs to $C^{\infty}(\R^d)$.

By Lemma \ref{lem3.2}, the equality in \eqref{eq22} holds for any function $f\in
\mathcal{S}(\R^d)$. If $f\in L^{p}(\mathbb{R}^{d},d\mu_{k})$, $f_n\in
\mathcal{S}(\R^d)$ and $f_n\to f$ in $L^{p}(\mathbb{R}^{d},d\mu_{k})$, then
Minkowski's inequality and \eqref{eq6} give
\begin{equation}
\|((f-f_n)\Ast{\lambda_k}g_0)\|_{p, d\mu_{k}}\leq \|f-f_n\|_{p, d\mu_{k}}\,\|g_0\|_{1, d\nu_{\lambda_{k}}}, \label{eq24}
\end{equation}
while H\"{o}lder's inequality and \eqref{eq21} imply
\[
|((f-f_n)\Ast{k}g)(x)|\leq \|f-f_n\|_{p, d\mu_{k}}\,\|g\|_{p', d\mu_{k}}.
\]
%taking into account
%inequalities \eqref{eq14} and \eqref{eq21} and
By \eqref{eq24}, there is a subsequence $\{n_k\}$ such that
$(f_{n_k}\Ast{\lambda_k}g_0)(x)\to (f\Ast{\lambda_k}g_0)(x)$ a.e., therefore % almost everywhere, therefore
the relation $(f\Ast{\lambda_k}g_0)(x)=(f\Ast{k}g)(x)$
 holds almost everywhere. Since both convolutions are continuous, then it holds everywhere.

To prove the second equation of the lemma, we first remark that
Lemma \ref{lem3.2} implies that \eqref{eq23} holds pointwise for any $f\in \mathcal{S}(\R^d)$.
In the general case, since
$f\in L^{p}(\mathbb{R}^{d},d\mu_{k})$, $\, (f\Ast{\lambda_k}g_0)\in
L^{p}(\mathbb{R}^{d},d\mu_{k})$, and $\mathcal{F}_k(g)\in \mathcal{S}(\R^d)$,
the left and right hand sides of \eqref{eq23} are tempered distributions. Recall that
the Dunkl transform of tempered distribution is defined by
\[
\<\mathcal{F}_k(f),\varphi\>=\<f,\mathcal{F}_k(\varphi)\>,\quad f\in
\mathcal{S}'(\R^d),\quad \varphi\in \mathcal{S}(\R^d).
\]
Let
$f_n\in \mathcal{S}(\R^d)$ and $f_n\to f$ in $L^{p}(\mathbb{R}^{d},d\mu_{k})$,
$\varphi\in \mathcal{S}(\R^d)$. Then
\[
\<\mathcal{F}_k((f-f_n)\Ast{\lambda_k}g_0),\varphi\>=\<((f-f_n)\Ast{\lambda_k}g_0),\mathcal{F}_k(\varphi)\>,
\]
\[
\<\mathcal{F}_k(g)\mathcal{F}_k(f-f_n),\varphi\>=\<(f-f_n),\mathcal{F}_k(\mathcal{F}_k(g)\varphi)\>
\]
and
\[
|\<\mathcal{F}_k((f-f_n)\Ast{\lambda_k}g_0),\varphi\>|\leq \|f-f_n\|_{p,
d\mu_{k}}\,\|g_0\|_{1, d\nu_{\lambda_{k}}}\,\|\mathcal{F}_k(\varphi)\|_{p',
d\mu_{k}},
\]
\[
|\<\mathcal{F}_k(g)\mathcal{F}_k(f-f_n),\varphi\>|\leq \|f-f_n\|_{p, d\mu_{k}}\,\|\mathcal{F}_k(\mathcal{F}_k(g)\varphi)\|_{p', d\mu_{k}}.
\]
Thus, the proof of \eqref{eq23} is now complete.
\end{proof}

\bigskip
\section{Boundedness of the Riesz potential}%\label{pitt-sec}
Recall that $\lambda_k=d/2-1+\sum_{a\in R_+}k(a)$. For
$0<\alpha<2\lambda_k+2$, the weighted Riesz potential $I_{\alpha}^kf$, is
defined on $\mathcal{S}(\R^d)$ (see \cite{ThaXu07}) by
\[
I_{\alpha}^kf(x)=(d_k^{\alpha})^{-1}\int_{\mathbb{R}^d}\tau^{-y}f(x)\frac{1}{|y|^{2\lambda_k+2-\alpha}}\,d\mu_k(y),
\]
where $d_k^{\alpha}=2^{-\lambda_k-1+\alpha}\Gamma(\alpha/2)/\Gamma(\lambda_k+1-\alpha/2)$.
We have, in the sense of tempered distributions,
\begin{equation*}%\label{eq25}
\mathcal{F}_k(I_{\alpha}^kf)(y)=|y|^{-\alpha}\mathcal{F}_k(f)(y).
\end{equation*}
Using \eqref{eq1} and \eqref{eq4}, we obtain
\begin{equation}
I_{\alpha}^kf(x)=(d_k^{\alpha})^{-1}\int_{0}^{\infty}T^{t}f(x)\frac{1}{t^{2\lambda_k+2-\alpha}}\,d\nu_{\lambda_k}(t).\label{eq26}
\end{equation}

To estimate the $L^p$-norm of this operator, we use the maximal
function defined for $f\in \mathcal{S}(\R^d)$ as follows %the maximal function $M_kf$ is defined
(\cite{ThaXu05}):
\[
M_kf(x)=\sup_{r>0}\frac{|(f\Ast{k}\chi_{B_r})(x)|}{\int_{B_r}\,d\mu_k},
\]
where $\chi_{B_r}$ is the characteristic function of the Euclidean ball $B_r$ of radius $r$ centered at 0.

Using \eqref{eq1}, \eqref{eq4}, and \eqref{eq15}, we get
\begin{equation*}
M_kf(x)%=\sup_{r>0}\frac{|\int_{\R^d}\tau^{-y}f(x)\chi_{B_r}(y)\,d\mu_k(y)|}{\int_{B_r}\,d\mu_k}
=
\sup_{r>0}\frac{|\int_{0}^{r}T^{t}f(x)\,d\nu_{\lambda_k}(t)|}{\int_{0}^{r}\,d\nu_{\lambda_k}}.%\label{eq27}
\end{equation*}

It is proved in \cite{ThaXu05} that the maximal function is bounded on
$L^{p}(\mathbb{R}^{d},d\mu_{k})$, $1<p\leq \infty$,
\begin{equation}
\|M_kf\|_{p, d\mu_k}\lesssim \|f\|_{p, d\mu_k} \label{eq28}
\end{equation}
and it is of weak type $(1, 1)$, that is, %for $f\in L^{1}(\mathbb{R}^{d},d\mu_{k})$ and $a>0$,
\begin{equation}
\int_{\{x\colon M_kf(x)>a\}}\,d\mu_k\lesssim\frac{\|f\|_{1,d\mu_{k}}}{a},\qquad a>0. \label{eq29}
\end{equation}

\begin{theorem}\label{thm4.1}
If $1<p<q<\infty$, $0<\alpha<2\lambda_k+2$,
$\frac{1}{p}-\frac{1}{q}=\frac{\alpha}{2\lambda_k+2}$,
then
\begin{equation}
\|I_{\alpha}^kf\|_{q,d\mu_k} \lesssim\|f\|_{p,d\mu_k},\qquad f\in\mathcal{S}(\R^d). \label{eq30}
\end{equation}
The mapping $f\mapsto I_{\alpha}^kf$ is of weak type $(1, q)$, that is,
\begin{equation}
\int_{\{x\colon |I_{\alpha}^kf(x)|>a\}}\,d\mu_k\lesssim
\Bigl(\frac{\|f\|_{1,d\mu_{k}}}{a}\Bigr)^q.
\label{eq31}
\end{equation}
\end{theorem}

\begin{remark}%\label{rem4.1}
In the case $k\equiv 0$, inequality \eqref{eq30} was proved by S.~Soboleff \cite{Sob38}
and G. O.~Thorin \cite{Tho48} and the weighted inequality %\eqref{eq29}
was studied by E. M.~Stein and G.~Weiss \cite{SteWei58}. For the reflection
group $G=\Z_2^d$, Theorem 4.1 was proved in \cite{ThaXu07}. The general case
was obtained in \cite{HasMusSif09}. %{\bf Since this paper is not accessible to the authors, we present
 We give another simple proof based on the $L^p$-boundedness of $T^t$ given
in Theorem \ref{thm3.3} and follow the proof given in \cite{ThaXu07} for
$G=\Z_2^d$.
\end{remark}

\begin{remark}%\label{rem4.2}
In Theorem \ref{thm4.1}, dealing with \eqref{eq30},
 we may assume that $f\in L^{p}(\mathbb{R}^{d},d\mu_{k})$, $1<p<\infty$,
 while proving \eqref{eq31},
 we may assume that $f\in L^{1}(\mathbb{R}^{d},d\mu_{k})$.
\end{remark}

\begin{proof}

Let $R>0$ be fixed. We write \eqref{eq26} as sum of two terms,
\begin{align}
I_{\alpha}^kf(x)&=(d_k^{\alpha})^{-1}
\int_{0}^{R}T^{t}f(x)\frac{1}{t^{2\lambda_k+2-\alpha}}\,d\nu_{\lambda_k}(t)\notag\\
&\qquad +(d_k^{\alpha})^{-1}\int_{R}^{\infty}T^{t}f(x)\frac{1}{t^{2\lambda_k+2-\alpha}}\,d\nu_{\lambda_k}(t)
=J_1+J_2. \label{eq32}
\end{align}
Integrating $J_1$ by parts, we obtain
\begin{align}
d_k^{\alpha}J_1&=\int_{0}^{R}t^{-(2\lambda_k+2-\alpha)}\,d
\Bigl(\int_{0}^{t}T^{s}f(x)\,d\nu_{\lambda_k}(s)\Bigr)\notag\\
&=R^{\alpha}\cdot
R^{-(2\lambda_k+2)}\int_{0}^{R}T^{s}f(x)\,d\nu_{\lambda_k}(s)\notag\\ &\qquad
+(2\lambda_k+2-\alpha)\int_{0}^{R}t^{-(2\lambda_k+2)}\int_{0}^{t}T^{s}f(x)\,d\nu_{\lambda_k}(s)\, t^{\alpha-1}\,dt.\label{eq33}
\end{align}
Here we have used that
\[
\lim_{\varepsilon\to 0+0}\varepsilon^{\alpha}\cdot
\varepsilon^{-(2\lambda_k+2)}\int_{0}^{\varepsilon}T^{s}f(x)\,d\nu_{\lambda_k}(s)=0,
\]
since
\[
\varepsilon^{\alpha}\cdot\varepsilon^{-(2\lambda_k+2)}
\Bigl|\int_{0}^{\varepsilon}T^{s}f(x)\,d\nu_{\lambda_k}(s)\Bigr|
\lesssim
\varepsilon^{\alpha}\sup_{\varepsilon>0}\frac{|\int_{0}^{\varepsilon}T^{t}f(x)\,d\nu_{\lambda_k}(t)|}{\int_{0}^{\varepsilon}\,d\nu_{\lambda_k}}
=\varepsilon^{\alpha}M_kf(x).
\]
In light of \eqref{eq33}, we have
\begin{equation}
|J_1|\lesssim R^{\alpha}M_kf(x)+\int_{0}^{R}M_kf(x)t^{\alpha-1}\,dt\lesssim R^{\alpha}M_kf(x). \label{eq31+}
\end{equation}

To estimate $J_2$, we use H\"{o}lder's inequality, the relation
$\frac{1}{p}-\frac{1}{q}=\frac{\alpha}{2\lambda_k+2}$ and \eqref{eq18}:
\begin{align*}
|J_2|&\leq (d_k^{\alpha})^{-1}
\Bigl(\int_{R}^{\infty}t^{-(2\lambda_k+2-\alpha)p'}\,d\nu_{\lambda_k}(t)\Bigr)^{1/p'}\|T^tf(x)\|_{p,
d\nu_{\lambda_k}}\\
&\lesssim R^{-(2\lambda_k+2)q}\|f\|_{p, d\mu_k}.
\end{align*}
This, \eqref{eq32} and \eqref{eq31+} yield
\[
|I_{\alpha}^kf(x)|\lesssim R^{\alpha}M_kf(x)+R^{-(2\lambda_k+2)q}\|f\|_{p, d\mu_k},
\]
for any $R>0$.
Choosing $R=\bigl(M_kf(x)/\|f\|_{p, d\mu_k}\bigr)^{-q/(2\lambda_k+2)}$ implies the inequality
\begin{equation}
|I_{\alpha}^kf(x)|\lesssim (M_kf(x))^{p/q}(\|f\|_{p, d\mu_k})^{1-p/q}\label{eq32+}
\end{equation}
for any $1\le p<q$.
Integrating \eqref{eq32+} and using \eqref{eq28}, we have
\[
\|I_{\alpha}^kf\|_{q,d\mu_k}\lesssim \|M_kf\|_{p, d\mu_k}^{p/q}\|f\|_{p,
d\mu_k}^{1-p/q}\lesssim \|f\|_{p, d\mu_k},\qquad p>1.
\]

Finally, we use inequality \eqref{eq29} for the maximal
function and inequality \eqref{eq32+} with $p=1$ to obtain
\[
\int_{\{x\colon |I_{\alpha}^kf(x)|>a\}}\,d\mu_k\leq \int_{\{x\colon
(M_kf(x))^{1/q}(\|f\|_{1, d\mu_k})^{1-1/q}\gtrsim a\}}\,d\mu_k\lesssim
%\Bigl(\frac{\|f\|_{1,d\mu_{k}}^{1-1/q}}{a}\Bigr)^q\|f\|_{1, d\mu_k}=
\Bigl(\frac{\|f\|_{1,d\mu_{k}}}{a}\Bigr)^q.
\qedhere
\]
\end{proof}

\bigskip
%\section{Direct theorems of approximation theory}
\section{%Some classes of entire functions of exponential type
Entire functions
of exponential type and
Plancherel--Polya--Boas-type inequalities}

Let $\mathbb{C}^d$ be the complex Euclidean space of $d$ dimensions. Let also
$z=(z_1,\dots,z_d)\in \mathbb{C}^d$,
$\mathrm{Im}\,z=(\mathrm{Im}\,z_1,\dots,\mathrm{Im}\,z_d)$, and $\sigma>0$.

In this section we define several classes of entire functions of exponential
type and study their interrelations. Moreover, we prove the
Plancherel--Polya--Boas-type estimates and the Paley--Wiener-type theorems.
These classes will be used later to study the approximation of functions on
$\mathbb{R}^d$ by entire functions of exponential type.

First, we define two classes of entire functions: $B_{p, k}^\sigma$ and
$\widetilde{B}_{p, k}^\sigma$. We say that a function $f\in B_{p, k}^\sigma$
if $f\in L^{p}(\mathbb{R}^{d},d\mu_{k})$ is such that its analytic
continuation to $\mathbb{C}^d$ satisfies
\begin{equation*}%\label{eq34}
|f(z)|\leq c_{\varepsilon}e^{(\sigma+\varepsilon)|z|},\quad \forall
\varepsilon>0,\ \forall z\in \mathbb{C}.
\end{equation*}
{ The smallest $\sigma=\sigma_{f}$ in this inequality is called a spherical type of $f$.}
In other words, the class $B_{p, k}^\sigma$ is the collection of all entire
functions of spherical type at most $\sigma$.

We say that a function $f\in
\widetilde{B}_{p, k}^\sigma$ if $f\in L^{p}(\mathbb{R}^{d},d\mu_{k})$ is such
that its analytic continuation to $\mathbb{C}^d$ satisfies
\begin{equation*}%\label{eq35}
|f(z)|\leq c_{f}e^{\sigma|\mathrm{Im}\,z|},\quad \forall z\in \mathbb{C}^d.
\end{equation*}
%In other words, the class $B_{p, k}^\sigma$ is the collection of all entire
%functions of spherical type at most $\sigma$, while,
Historically, functions from $\widetilde{B}_{p, k}^\sigma$ were basic objects
in the Dunkl harmonic analysis. It is clear that $\widetilde{B}_{p,
k}^\sigma\subset B_{p, k}^\sigma$. Moreover, if $k\equiv 0$, then both classes
coincide (see, e.g., \cite{NesWil78}). Indeed, if $f\in B_{p, 0}^\sigma$,
$1\leq p<\infty$, then Nikol'skii's inequality \cite[3.3.5]{Nik75}
\[
\|f\|_{\infty}\leq 2^d\sigma^{d/p}\|f\|_{p,d\mu_0}
\]
and the inequality \cite[3.2.6]{Nik75}
\[
\|f(\Cdot+iy)\|_{\infty}\leq e^{\sigma |y|}\|f\|_{\infty},\quad y\in \R^{d},
\]
imply that, for $z=x+iy\in \mathbb{C}^{d}$,
\[
|f(z)|\leq 2^d\sigma^{d/p} \|f\|_{p,d\mu_0}e^{\sigma |\mathrm{Im}\,z|},
\]
i.e., $f\in \widetilde{B}_{p, 0}^\sigma$.

In fact, the classes $B_{p, k}^\sigma$ and $\widetilde{B}_{p, k}^\sigma$
coincide in the weighted case ($k\neq 0$) as well. To see that it is enough to
show that functions from $B_{p, k}^\sigma$ are bounded on $\mathbb{R}^d$.

\begin{theorem}\label{thm5.1}
 If $0<p<\infty$, then $B_{p, k}^\sigma=\widetilde{B}_{p, k}^\sigma$.
\end{theorem}
We will actually prove the more general statement. Let $m\in \Z_+,$
$\alpha^1,\dots,\alpha^m\in\mathbb{R}^d\setminus\{0\}$, $k_0\ge 0,$ $k_1,\dots,k_m>0$,
and
\begin{equation}\label{eq35+}
v(x)=|x|^{k_0}\prod_{j=1}^m|\<\alpha^j,x\>|^{k_j}
\end{equation}
be the power weight.
The Dunkl weight is a particular case of such weighted functions.
{The weighted function \eqref{eq35+} arises in the study of the generalized Fourier transform (see, e.g., \cite{BenKobOrs12}).}

Let $L^{p,v}(\mathbb{R}^d)$, $0<p<\infty$, be the space of complex-valued Lebesgue measurable functions $f$ for which
\[
\|f\|_{p,v}=\Bigl(\int_{\mathbb{R}^d}|f(x)|^pv(x)\,dx\Bigr)^{1/p}<\infty.
\]

Let $\bsigma=(\sigma_1,\dots,\sigma_d)$, $\sigma_1,\dots,\sigma_d>0$.
 %\textbf{Briefly we write} $\bsigma>0$.
 Again, let us define three \textit{anisotropic} classes of entire
functions: $B^{\bsigma}$, $B_{p,v}^{\bsigma}$, and
$\widetilde{B}_{p,v}^{\bsigma}$.

We say that a function $f$ defined on $\mathbb{R}^d$ belongs to $B^{\bsigma}$ if
its analytic continuation to $\mathbb{C}^d$ satisfies
\[
|f(z)|\leq
c_{\varepsilon}e^{(\sigma_1+\varepsilon)|z_1|+\dots+(\sigma_d+\varepsilon)|z_d|},\quad
\forall \varepsilon>0,\ \forall z\in \mathbb{C}^d.
\]
We say that a function $f\in B_{p,v}^{\bsigma}$ if $f\in L^{p}(\mathbb{R}^{d},d\mu_{k})$ is such that its analytic continuation to $\mathbb{C}^d$ belongs to $B^{\bsigma}$.

We say that a function $f\in \widetilde{B}_{p,v}^{\bsigma}$ if $f\in L^{p}(\mathbb{R}^{d},d\mu_{k})$ is such that its analytic continuation to $\mathbb{C}^d$ satisfies
\[
|f(z)|\leq c_fe^{\sigma_1|\mathrm{Im}\,z_1|+\dots+\sigma_d|\mathrm{Im}\,z_d|},\quad
\forall z\in\mathbb{C}^d.
\]
We will use the notation $L^p(\mathbb{R}^d)$, $\|\Cdot\,\|_{p}$,
$B_p^{\bsigma}$ and $\widetilde{B}_p^{\bsigma}$ in the case of the unit weight,
i.e., $v\equiv1$.

\begin{theorem}\label{thm5.2}
If $0<p<\infty$, then

\textup{(1)}~$B_{p,v}^{\bsigma}\subset B_{p}^{\bsigma}$,

\textup{(2)}~$B_{p,v}^{\bsigma}=\widetilde{B}_{p,v}^{\bsigma}$,

\textup{(3)}~$B_{p,v}^{\sigma}=\widetilde{B}_{p,v}^{\sigma}$.
\end{theorem}

\begin{remark}%\label{rem5.1}
(i) Part (3) of Theorem \ref{thm5.2} implies Theorem \ref{thm5.1}.

\smallbreak
(ii) Note that in some particular cases ($k_0=0$ and $p\geq 1$) a similar
result was discussed in \cite{IvaSmiLiu14}.
% Since the paper \cite{IvaSmiLiu14} is not easily
%accessible, we give a detailed proof of Theorem~\ref{thm5.2}.}
\end{remark}

Parts (2) and (3) of Theorem \ref{thm5.2} follows from (1). Indeed, the
embedding in (1) implies that $B_{p,v}^{\bsigma}\subset B_{p}^{\bsigma}\subset
B_{\infty}^{\bsigma}$. Hence, a function $f\in B_{p,v}^{\bsigma}$ is bounded on
$\R^{d}$ and then $f\in \widetilde{B}_{p,v}^{\bsigma}$, which gives (2).
Further, there holds $B_{p,v}^{\sigma}\subset B_{p,v}^{\bsigma}$, where
$\bsigma=(\sigma,\dots,\sigma)\in \R_{+}^{d}$ since $|z|\le
|z_{1}|+\dots+|z_{d}|$. Hence, similar to the above, we have
$B_{p,v}^{\sigma}\subset B_{\infty}^{\bsigma}$ and (3) follows. Thus, to prove
Theorem \ref{thm5.2}, it is sufficient to verify part (1).

The main difficulty to prove Theorem \ref{thm5.2} is that the weight $v(x)$
vanishes. In order to overcome this problem we will first prove two-sided
estimates of the $L^p$ norm of an entire functions in terms of the
weighted $l_p$ norm, $\bigl(\sum_{n}
v(\lambda^{(n)})|f(\lambda^{(n)})|^p\bigr)^{1/p}$, $0<p<\infty$, where $v$ does
not vanish at $\{\lambda^{(n)}\}\subset \mathbb{R}^d$.

Such estimates are of their own interest. They
generalize the Plancherel--Polya inequality (\cite{PlaPol37}, \cite[Chapt. 6, 6.7.15]{Boa54})
\[
\sum_{k\in\mathbb{Z}}|f(\lambda_k)|^p\leq
c(\delta,\sigma,p)\int_{-\infty}^{\infty}|f(x)|^p\,dx,\quad 0<p<\infty,
\]
where $\lambda_k$ is increasing sequence such that $\lambda_{k+1}-\lambda_k\geq
\delta>0$, and $f$ is an entire function of exponential type at most $\sigma$,
and the Boas inequality \cite{Boa52}, \cite[Chapt. 10, 10.6.8]{Boa54},
\begin{equation}
\int_{-\infty}^{\infty}|f(x)|^p\,dx\leq
C(\delta,L,\sigma,p)\sum_{k\in\mathbb{Z}}|f(\lambda_k)|^p,\quad 0<p<\infty,
\label{eq36}
\end{equation}
where, additionally, $\left|\lambda_k-\frac{\pi}{\sigma}\,k\right|\leq L$ and
the type of $f$ is $<\sigma$.

We write
$\bsigma'=(\sigma_1',\dots,\sigma_d')<\bsigma=(\sigma_1,\dots,\sigma_d)$ if
$\sigma_1'<\sigma_1,\dots,\sigma_d'<\sigma_d$. Let $n=(n_1,\dots,n_d)\in
\mathbb{Z}^d$ and $\lambda^{(n)}\colon \mathbb{Z}^d\rightarrow \mathbb{R}^d$.
In what follows we consider the sequences of the following type:
\begin{equation}
\lambda^{(n)}=(\lambda_1(n_1), \lambda_2(n_1, n_2),\dots, \lambda_d(n_1,\dots, n_d)), \label{eq37}
\end{equation}
where $\lambda_{i}^{(n)}=\lambda_i(n_1,\dots,n_i)$ are sequences increasing with
respect to $n_i$, $i=1,\dots,d$ for fixed $n_1,\dots,n_{i-1}$.

\begin{definition} We say that the sequence $\lambda^{(n)}$ satisfies
\textit{the separation condition} $\Omega_\mathrm{sep}[\delta]$, $\delta>0$, if, for
any $n\in \mathbb{Z}^d$,
\[
\lambda_i(n_1,\dots,n_{i-1},n_i+1)-\lambda_i(n_1,\dots,n_{i-1},n_i)\geq\delta,\quad
i=1,\dots,d.
\]
\end{definition}
Note that if the sequence $\lambda^{(n)}$ satisfies the separation condition
$\Omega_\mathrm{sep}[\delta]$ then it also satisfies the condition $\inf_{n\ne
m}|\lambda^{(n)}-\lambda^{(m)}|> 0$.

\begin{definition} We say that the sequence $\lambda^{(n)}$ satisfies
\textit{the close-lattice condition} $\Omega_\mathrm{lat}[\ba,L]$,
$\ba=(a_1,\dots,a_d)>0$, $L>0$, if, for any $n\in \mathbb{Z}^d$,
\[
\Bigl|\lambda_i(n_1,\dots,n_i)-\frac{\pi n_i}{a_i}\Bigr|\leq L,\quad
i=1,\dots,d.
\]
\end{definition}
We start with the Plancherel--Polya-type inequality.
\begin{theorem}\label{thm5.3}
Assume that $\lambda^{(n)}$ satisfies
the condition $\inf_{n\ne m}|\lambda^{(n)}-\lambda^{(m)}|> 0$. Then for $f\in B^{\bsigma}_p$, $0<p<\infty$, we have
\[
\sum_{n\in\mathbb{Z}^d}|f(\lambda^{(n)})|^p\lesssim\int_{\mathbb{R}^d}|f(x)|^p\,dx.
\]
\end{theorem}

\begin{proof}
For the simplicity we prove this result for $d=2$. The proof in the general
case is similar.

The function $|f(z)|^p$ is plurisubharmonic, and therefore for any
$x=(x_1, x_2)\in \mathbb{R}^2$ one has \cite{Ron71}
\[
|f(x_1,x_2)|^p\leq \frac{1}{(2\pi)^2}\int_0^{2\pi}\int_0^{2\pi}|f(x_1+\rho_1e^{i\theta_1},\, x_2+\rho_2e^{i\theta_2}|^p\,d\theta_1 d\theta_2,
\]
where $\rho_1,\,\rho_2>0$. Following \cite[3.2.5]{Nik75}, for $\delta>0$ and
$\xi+i\eta=(\xi_1+i\eta_1, \xi_2+i\eta_2)$ we obtain that
\begin{equation}
|f(x_1,x_2)|^p\leq
\frac{1}{(\pi\delta^2)^2}\int_{-\delta}^{\delta}\int_{-\delta}^{\delta}
\int_{x_1-\delta}^{x_1+\delta}\int_{x_2-\delta}^{x_2+\delta}|f(\xi+i\eta)|^p\,d\xi_1\,d\xi_2\,d\eta_1\,d\eta_2.
\label{eq38}
\end{equation}
The separation condition implies that for some $\delta>0$ the boxes
$[\lambda_1^{(n)}-\delta, \lambda_1^{(n)}+\delta]\times [\lambda_2^{(n)}-\delta,
\lambda_2^{(n)}+\delta]$ do not overlap for any $n$.

Since
\[
f(x+iy)=\sum_{k\in \Z_+^2}\frac{f^{(k)}(x)}{k!}\,(iy)^{k},
\]
where $f^{(k)}$ is a partial derivative $f$ of order $k=(k_1,k_2)$,
$k!=k_1!\,k_2!$, and $(iy)^k=(iy_1)^{k_{1}}(iy_{2})^{k_{2}}$, then, applying
Bernstein's inequality (see \cite[3.2.2 and 3.3.5]{Nik75} and \cite{RahSch90}), we
derive that
\[
\|f(\Cdot+iy)\|_{p}\lesssim e^{\sigma_1|y_1|+\sigma_2|y_2|}\|f\|_{p}.
\]
Using this and \eqref{eq38} we derive that
\begin{align*}
\sum_{n\in\mathbb{Z}^2}|f(\lambda^{(n)})|^p&\leq
\frac{1}{(\pi\delta^2)^2}\int_{-\delta}^{\delta}\int_{-\delta}^{\delta}
\int_{-\infty}^{\infty}\int_{-\infty}^{\infty}|f(\xi+i\eta)|^p\,d\xi_1\,d\xi_2\,d\eta_1\,d\eta_2\\
&\lesssim\int_{-\delta}^{\delta}\int_{-\delta}^{\delta}e^{p(\sigma_1|\eta_1|+\sigma_2|\eta_d|)}\,d\eta_1\,d\eta_2
\int_{-\infty}^{\infty}\int_{-\infty}^{\infty}|f(\xi)|^p\,d\xi_1\,d\xi_2\\
&\lesssim \int_{\mathbb{R}^2}|f(x)|^p\,dx.
\qedhere
\end{align*}
\end{proof}

\begin{theorem}\label{thm5.4}
Let the sequence $\lambda^{(n)}$ of form \eqref{eq37} satisfy the conditions
$\Omega_\mathrm{sep}[\delta]$ and $\Omega_\mathrm{lat}[\bsigma,L]$. Assume that $f\in
B^{\bsigma'}$, $\bsigma'<\bsigma$, is such that
$\sum_{n\in\mathbb{Z}^d}|f(\lambda^{(n)})|^p<\infty$, $0<p<\infty$. Then $f\in
L^p(\mathbb{R}^d)$ and
\begin{equation*}
\int_{\mathbb{R}^d}|f(x)|^p\,dx\lesssim\sum_{n\in\mathbb{Z}^d}|f(\lambda^{(n)})|^p. %\label{eq39}
\end{equation*}
\end{theorem}

\begin{remark}%\label{rem5.2}
 For $p\geq 1$, a similar two-sided Plancherel--Polya--Boas-type inequality was
   % Theorem \ref{thm5.4} can be
  obtained from \cite{Pes07}.
\end{remark}

\begin{proof} For the simplicity we consider the case $d=2$.
Integrating $|f(x_1,x_2)|^p$ at $x_1$ and applying inequality \eqref{eq36}, we get, for any $x_2$,
\[
\int_{-\infty}^{\infty}|f(x_1,x_2)|^p\,dx_1\lesssim\sum_{n_1\in\mathbb{Z}}
|f(\beta_1(n_1),x_{2})|^p.
\]
Since by \eqref{eq36}, for any $n_1$,
\[
\int_{-\infty}^{\infty}|f(\beta_1(n_1),x_2)|^p\,dx_2\lesssim\sum_{n_2\in\mathbb{Z}}
|f(\beta_1(n_1),\beta_2(n_1, n_2))|^p,
\]
then
\begin{align*}
\int_{-\infty}^{\infty}\int_{-\infty}^{\infty}|f(x_1,x_2)|^p\,dx_1\,dx_2&\lesssim
\sum_{n_1\in\mathbb{Z}}\int_{-\infty}^{\infty}|f(\beta_1(n_1),x_2)|^p\,dx_2
\\ &\lesssim\sum_{n_1\in\mathbb{Z}}\sum_{n_2\in\mathbb{Z}}
|f(\beta_1(n_1),\beta_2(n_1, n_2))|^p<\infty.
\qedhere
\end{align*}
\end{proof}

Using Theorems~\ref{thm5.3} and \ref{thm5.4} we arrive at the following statement.

\begin{theorem}%\label{cor5.5}
Let the sequence $\{\lambda^{(n)}\}$ of form \eqref{eq37} satisfy the conditions
$\Omega_\mathrm{sep}[\delta]$ and $\Omega_\mathrm{lat}[\bsigma,L]$.
If $f\in B^{\bsigma'}$, $\bsigma'<\bsigma$, then, for $0<p<\infty$,
\[
\sum_{n\in\mathbb{Z}^d}|f(\lambda^{(n)})|^p
\lesssim\int_{\mathbb{R}^d}|f(x)|^p\,dx\lesssim\sum_{n\in\mathbb{Z}^d}|f(\lambda^{(n)})|^p.
\]
\end{theorem}

%\begin{remark}
%{ For $p\geq 1$ Theorem~\ref{cor5.5} also follows from \cite{Pes07}.}
%\end{remark}

We will need the weighted version of the Plancherel--Polya--Boas equivalence. We start with three auxiliary lemmas.

\begin{lemma}\label{lem5.6}\textnormal{\cite{GorIvaVep14}}
If $\gamma\geq -1/2$, then there exists an even entire function
$\omega_\gamma(z)$, $z\in \mathbb{C}$, of exponential type~$2$ such that,
uniformly in $x\in \mathbb{R}_+$,
\[
\omega_\gamma(x)\asymp
\begin{cases}x^{2k+2},& 0 \leq x\leq 1,\\ x^{2\gamma+1},& x\geq 1,\end{cases}
\]
where $k=[\gamma+1/2]$ and $[a]$ is the integral part of $a$. In particular,
we can take
\[
\omega(z)=z^{2k+2}j_{k-\gamma}(z+i)j_{k-\gamma}(z-i).
\]
\end{lemma}

\begin{lemma}\label{lem5.7}
Let $m\in \mathbb{N}$, $j=1,\dots,m$,
$b^j=(b_1^j,\dots,b_d^j)\in\mathbb{R}^d\setminus\{0\}$, and either $|b_i^j|\geq
1$, or $b_i^j=0$, $i=1,\dots,d$. Then there exists a sequence
$\{\rho^{(n)}\}\subset \mathbb{Z}^d\setminus \{0\}$ of the form \eqref{eq37}
such that, for any $j=1,\dots,m$ and $i=1,\dots,d$,
\begin{equation}
|\rho_i(n_1,\dots,n_i)-n_i|\leq m, \label{eq44}
\end{equation}
\begin{equation}
|\<b^j, \rho^{(n)}\>|\geq 1/2. \label{eq45}
\end{equation}
\end{lemma}

\begin{proof} To construct a desired sequence
\[
\rho^{(n)}=(\rho_1(n_1), \rho_2(n_1, n_2),\dots,\rho_d(n_1,\dots,n_d))\in \mathbb{Z}^d,
\]
we will use the following simple remark. If we
throw out $m$ points
from $\mathbb{Z}$, then the rest can be numbered such that the obtained sequence will be increasing and
\eqref{eq44} holds.

Let $J_1=\{j\colon b_1^j\neq 0,\, b_2^j=\dots=b_d^j=0\}$. If
$J_1=\varnothing$, then we set $\rho_1(n_1)=n_1$. If
$J_1\neq\varnothing$, then $\rho_1(n_1)$ is increasing sequence formed
from $\mathbb{Z}\setminus \{0\}$. In both cases \eqref{eq44} is valid and,
moreover, for $j\in J_1$ and any
$\rho_2(n_1,n_2),\dots,\rho_d(n_1,\dots,n_d)$, one has \eqref{eq45} since
$|\<b^j, \rho^{(n)}\>|=|b_1^j\rho_1(n_1)|\geq 1$.

Let $J_2=\{j\colon b_2^j\neq 0,\, b_3^j=\dots=b_d^j=0\}$, $n_1\in \mathbb{Z}$.
If $J_2=\varnothing$, then we set $\rho_2(n_1, n_2)=n_2$. Let
$J_2\neq\varnothing$. If $j\in J_2$ and $b_1^j\rho_1(n_1)+b_2^jt_j=0$, then
$t_j=l_j+\varepsilon_j$, $l_j\in \mathbb{Z}$, $|\varepsilon_j|\leq 1/2$. Here
$l_j$ is the nearest integer to $t_j$. Note that if $\rho_2\neq l_j$, then
$|b_1^j\rho_1(n_1)+b_2^j\rho_2|=|b_2^j(\rho_2-l_j-\varepsilon_j)|\geq 1/2$.

Let $\rho_2(n_1, n_2)$ be an increasing sequence at $n_2$
formed from $\mathbb{Z}\backslash
\{l_j\colon j\in J_2\}$. For this sequence \eqref{eq44} holds and,
 for $j\in J_2$ and any $\rho_3(n_1,n_2,n_3),\dots,\rho_d(n_1,\dots,n_d)$, one has
\[
|\<b^j, \rho^{(n)}\>|=|b_1^j\rho_1(n_1)+b_2^j\rho_2(n_1, n_2)|\geq 1/2,
\]
that is, \eqref{eq45} holds as well.

Assume that we have constructed the sets $J_1,\dots,J_{d-1}$, and the
sequence $(\rho_1(n_1), \rho_2(n_1, n_2),\dots,\rho_{d-1}(n_1,\dots,n_{d-1}))\in \mathbb{Z}^{d-1}$.

Let $J_d=\{j\colon b_d^j\neq 0\}$, $(n_1,\dots,n_{d-1})\in
\mathbb{Z}^{d-1}$. If $J_d=\varnothing$, then we set $\rho_d(n_1,\dots,n_{d-1},n_d)=n_d$.
Assume now that $J_d\neq\varnothing$. If $j\in J_d$ and
\[
b_1^j\rho_1(n_1)+\cdots+b_{d-1}^j\rho_{d-1}(n_1,\dots,n_{d-1})+b_d^jt_j=0,
\]
then $t_j=l_j+\varepsilon_j$, $|\varepsilon_j|\leq
1/2$. Note that if $\rho_d\neq l_j$, then
\[
|b_1^j\rho_1(n_1)+\cdots+b_{d-1}^j\rho_{d-1}(n_1,\dots,n_{d-1})+b_d^j\rho_d|=|b_d^j(\rho_d-l_j-\varepsilon_j)|\geq 1/2.
\]

Let $\rho_d(n_1,\dots,n_d)$ be an increasing sequence in $n_d$ formed from
$\mathbb{Z}\setminus \{l_j\colon j\in J_d\}$, $\rho^{(n)}=(\rho_1(n_1),
\rho_2(n_1, n_2),\dots,\rho_d(n_1,\dots,n_d))$. For the sequence
$\rho_d(n_1,\dots,n_d)$ inequality \eqref{eq44} holds and, for $j\in J_d$, one has $|\<b^j,
\rho^{(n)}\>|\geq 1/2$.

Thus, we construct the desired sequence since, for any $j\in\{1,\dots,m\}$ and
some $i\in\{1,\dots,d\}$, there holds $b^j\in J_i$.
\end{proof}

An important ingredient of the proof of Theorem \ref{thm5.2} is the following
corollary of Lemma \ref{lem5.7}.

\begin{lemma}\label{lem5.8}
If $\ba>0$, $\alpha^1,\dots,\alpha^m\in\mathbb{R}^d\setminus\{0\}$, then there
exists a sequence $\lambda^{(n)}$ of the form \eqref{eq37} such that for some
$\delta,\,L>0$ the conditions $\Omega_\mathrm{sep}[\delta]$,
$\Omega_\mathrm{lat}[\ba, L]$, and $\xi_{j}(\lambda^{(n)})\ge \delta$,
$j=0,1,\dots,m$, $n\in \mathbb{Z}^d$, hold, where
\begin{equation}\label{eq46++}
\xi_0(x)=|x|,\quad \xi_j(x)=|\<\alpha^j,x\>|,\ j=1,\dots,m.
\end{equation}
\end{lemma}

Indeed, for $m\ge 1$ it is enough to define
\begin{align}
\lambda^{(n)}&=(\lambda_1(n_1), \lambda_2(n_1, n_2),\dots, \lambda_d(n_1,\dots, n_d))\notag\\
&:=\Bigl(\frac{\pi \rho_1(n_1)}{a_1}, \frac{\pi \rho_2(n_1, n_2)}{a_2},\dots,
\frac{\pi \rho_d(n_1,\dots, n_d)}{a_d}\Bigr),
\label{lambdan}
\end{align}
where $\rho^{(n)}$ is the sequence defined in Lemma~\ref{lem5.7}. For $m=0$ in
\eqref{lambdan}, we can take $\{\rho^{(n)}\}=\Z^{d}\setminus \{0\}$.

We are now in a position to state the Plancherel--Polya--Boas inequalities with weights.
\begin{theorem}\label{thm5.9}
Let $f\in B^{\bsigma}$ and $\lambda^{(n)}$ be the sequence satisfying all
conditions of Lemma~\ref{lem5.8} with some $\ba>\bsigma$,
 then, for $0<p<\infty$,
\[
\sum_{n\in\mathbb{Z}^d}v(\lambda^{(n)})|f(\lambda^{(n)})|^p
\lesssim\int_{\mathbb{R}^d}|f(x)|^pv(x)\,dx\lesssim
\sum_{n\in\mathbb{Z}^d}v(\lambda^{(n)})|f(\lambda^{(n)})|^p.
\]
\end{theorem}

\begin{proof}
Recall that $v(x)=\prod_{j=0}^mv_j(x)$, where
$v_j(x)=\xi_j^{k_j}(x)$, $j=0,1,\dots,m$ (see \eqref{eq35+} and \eqref{eq46++}).

By Lemma \ref{lem5.6}, we construct entire function of exponential type
$$w(z)=\prod_{j=0}^m w_j(z),$$ where $w_0(z)=\omega_{\gamma_{0}}(|z|)$,
$w_j(z)=\omega_{\gamma_j}(\<\alpha^j, z\>)$, $j=1,\dots,m$, and
\[
\gamma_j=\frac{k_j}{2p}-\frac{1}{2},\quad j=0,1,\dots,m.
\]
For $j=0,1,\dots,m$, we have $w_j\in B^{2\bmu^j}$, where
\[
\bmu^0=(1,\dots,1)\in \R^{d},\quad \bmu^j=
(|\alpha_1^j|,\dots,|\alpha_d^j|),\ j=1,\dots,m,
\]
and $w\in B^{2\bmu}$, $\bmu=\sum_{j=0}^m\bmu^j$. Moreover,
for any $j=0,1,\dots,m$,
\begin{equation}\label{eq46}
\begin{aligned}
w_j^p(x)&\lesssim v_j(x),\quad x\in \mathbb{R}^d,\\
w_j^p(x)&\gtrsim v_j(x)\gtrsim 1,\quad \text{for}\quad \xi_j(x)\geq \delta>0.
\end{aligned}
\end{equation}

Let $f\in B^{\bsigma}_{p,v}$, $0<p<\infty$, $\bsigma<\ba$, and $\lambda^{(n)}$
be the sequence satisfying all conditions of Lemma~\ref{lem5.8}, then, for some
$s>0$ such that $\bsigma+2s\bmu<\ba$, we have that $f(x)w(sx)\in B^{\bsigma+2s\bmu}$.

Using Theorem \ref{thm5.3}, and properties \eqref{eq46}, we derive
\begin{align*}
\sum_{n\in\mathbb{Z}^d}v(\lambda^{(n)})|f(\lambda^{(n)})|^p&\lesssim
\sum_{n\in\mathbb{Z}^d}|f(\lambda^{(n)})w(\lambda^{(n)})|^p\\
&\lesssim \int_{\mathbb{R}^d}|f(x)w(x)|^p\,dx\lesssim\int_{\mathbb{R}^d}|f(x)|^pv(x)\,dx.
\end{align*}

Let $\delta>0$, $J\subset J_{m}:=\{0,1,\dots,m\}$ or $J=\varnothing$,
\[
E_{\delta}(J)=\{x\in\R^d\colon \xi_{j}(x)\geq\delta,\ j\in J\ \text{and}\
\xi_{j}(x)\leq\delta,\ j\in J_{m}\setminus J\}.
\]

Since $f(x)\prod_{j\in J}w_j(sx)\in B^{\bsigma+2s\bmu}$, then using
Theorems~\ref{thm5.3}, \ref{thm5.4} and properties \eqref{eq46} for $\delta$
from Lemma~\ref{lem5.8}, we obtain
\begin{align*}
\int_{\mathbb{R}^d}|f(x)|^pv(x)\,dx&=\sum_{J}\int_{E_{\delta}(J)}|f(x)|^pv(x)\,dx\lesssim
\sum_{J}\int_{E_{\delta}(J)}|f(x)|^p\prod_{j\in J}v_j(sx)\,dx\\ &\lesssim
\sum_{J}\int_{E_{\delta}(J)}\Bigl|f(x)\prod_{j\in J}w_j(sx)\Bigr|^p\,dx\lesssim
\sum_{J}\int_{\R^d}\Bigl|f(x)\prod_{j\in J}w_j(sx)\Bigr|^p\,dx\\ &\lesssim
\sum_{n}|f(\lambda^{(n)})|^p\sum_{J}\prod_{j\in
J}w_j^p(s\lambda^{(n)})
\lesssim\sum_{n}|f(\lambda^{(n)})w(s\lambda^{(n)})|^p
\\
&\lesssim\sum_{n}|f(\lambda^{(n)})|^pv(s\lambda^{(n)})
\lesssim\sum_{n}|f(\lambda^{(n)})|^pv(\lambda^{(n)}),
\end{align*}
where we have assumed that $\prod_{j\in\varnothing}=1$.
\end{proof}

\begin{proof}[Proof of Theorem~\ref{thm5.2}]
Recall that it is enough to show that $B_{p,v}^{\bsigma}\subset B_{p}^{\bsigma}$ and the latter follows from
 $B^{\bsigma}_{p,v}\subset L^{p}(\mathbb{R}^d)$.

Let $f\in B^{\bsigma}_{p,v}$, $0<p<\infty$, $\ba>\bsigma$, and $\lambda^{(n)}$
be the sequence satisfying all conditions of Lemma \ref{lem5.8}. Using Theorem
\ref{thm5.3}, and properties \eqref{eq46} as in Theorem~\ref{thm5.9} we have
\begin{align*}
\int_{\mathbb{R}^d}|f(x)|^p\,dx&\lesssim
\sum_{n\in\mathbb{Z}^d}|f(\lambda^{(n)})|^p\lesssim
\sum_{n\in\mathbb{Z}^d}|w(\lambda^{(n)})f(\lambda^{(n)})|^p\\
&\lesssim \int_{\mathbb{R}^d}|f(x)w(x)|^p\,dx\lesssim\int_{\mathbb{R}^d}|f(x)|^pv(x)\,dx.
\qedhere
\end{align*}
\end{proof}

%\begin{remark}\label{rem5.1}
%Bernstein', Nikolskii's inequalities, Plancherel--Polya
%and Boas-type inequalities, as well as Theorems \ref{thm5.3} and
%\ref{thm5.4} hold for $0<p<1$ (see \cite{Boa54, RahSch90}). Then Theorem
%\ref{thm5.2} also holds for $0<p<1$.
%\end{remark}

By the Paley--Wiener theorem for tempered distributions (see \cite{Jeu06,Tri02}) and Theorem \ref{thm5.1}, we arrive at the following result.

\begin{theorem}\label{thm5.10}
A function $f\in B_{p, k}^\sigma$, $1\leq p<\infty$, if and only if $f\in
L^{p}(\mathbb{R}^{d},d\mu_{k})\cap C_b(\R^d)$ and $\supp
\mathcal{F}_k(f)\subset B_\sigma$.
\end{theorem}

The Dunkl transform $\mathcal{F}_k(f)$ in Theorem \ref{thm5.10} is understood as a function for $1\leq
p\leq 2$ and as a tempered distribution for $p>2$.

We conclude this section by presenting the concept of the best approximation.
Let
\[
E_{\sigma}(f)_{p,d\mu_k}=\inf\{\|f-g\|_{p,d\mu_k}\colon g\in B_{p, k}^\sigma\}
\]
be the best approximation of a function $f\in
L^{p}(\mathbb{R}^{d},d\mu_{k})$ by entire functions of spherical exponential
type $\sigma$. We show that the best approximation is achieved.

\begin{theorem}\label{thm5.11}
For any $f\in L^{p}(\mathbb{R}^{d},d\mu_{k})$, $1\leq p\leq\infty$, there exist a
function $g^{\ast}\in B_{p, k}^\sigma$ such that
$E_{\sigma}(f)_{p,d\mu_k}=\|f-g^{\ast}\|_{p,d\mu_k}$.
\end{theorem}
\begin{proof}
The proof is standard.
 Let $g_n$ be a sequence from $B_{p, k}^\sigma$ such that $\|f-g_n\|_{p,d\mu_k}\to
E_{\sigma}(f)_{p,d\mu_k}$. Since it is bounded in
$L^{p}(\mathbb{R}^{d},d\mu_{k})$, then it is also bounded in $C_b(\R^d)$. A
compactness theorem for entire functions \cite[3.3.6]{Nik75} implies
that there exist a subsequence $g_{n_k}$ and an entire function $g^{\ast}$ of
exponential type at most $\sigma$ such that
\[
\lim_{k\to\infty}g_{n_k}(x)=g^{\ast}(x),\quad x\in \R^d,
\]
and, moreover, convergence is uniform on compact sets. Therefore, for any $R>0$,
%and the characteristic function $\chi_{B_R}$
\[
\|g^{\ast}\chi_{B_R}\|_{p,d\mu_k}=\lim_{k\to\infty}\|g_{n_k}\chi_{B_R}\|_{p,d\mu_k}\leq M.
\]
Letting $R\to\infty$, we have that $g^{\ast}\in B_{p, k}^\sigma$. In light of
\begin{align*}
\|(f-g^{\ast})\chi_{B_R}\|_{p,d\mu_k}&=\lim_{k\to\infty}\|(f-g_{n_k})\chi_{B_R}\|_{p,d\mu_k}\\
&\leq
\lim_{k\to\infty}\|f-g_{n_k}\|_{p,d\mu_k}=E_{\sigma}(f)_{p,d\mu_k},
\end{align*}
we have $\|f-g^{\ast}\|_{p,d\mu_k}\leq E_{\sigma}(f)_{p,d\mu_k}$.
\end{proof}

\bigskip
\section{Jackson's inequality and equivalence of modulus of smoothness and $K$-functional% Smoothness and structural characteristics
}

\subsection{Smoothness characteristics and $K$-functional}
Let $\mathcal{S}'(\R^d)$ be the space of tempered distributions, $r\in\N$.
We can multiply tempered distributions on functions from $C^{\infty}_{\Pi}(\mathbb{R}^d)$.
Observe that $|x|^{2r},e_k(x,\Cdot),j_{\lambda_k}(|x|)\in C^{\infty}_{\Pi}(\mathbb{R}^d)$.

Further, using multipliers from $C^{\infty}_{\Pi}(\mathbb{R}^d)$, Laplacian, and Dunkl transform, we define several functionals on the Schwartz space. If a sequence $\{\varphi_l\}\subset\mathcal{S}(\R^d)$ converges to zero in topology of $\mathcal{S}(\mathbb{R}^{d})$, then the sequences $\{(-\Delta_k)^r\varphi_l\}$, $\{\mathcal{F}_k(\varphi_l)\}$, $\{g\varphi_l\}$, $g\in C^{\infty}_{\Pi}(\mathbb{R}^d)$, also converge to zero in topology of $\mathcal{S}(\mathbb{R}^{d})$. Hence all functionals will be continue.

First we define the Dunkl transform for tempered distribution
\[
\<\mathcal{F}_k(f),\varphi\>=\<f,\mathcal{F}_k(\varphi)\>,\quad f\in
\mathcal{S}'(\R^d),\quad \varphi\in \mathcal{S}(\R^d).
\]
If $\check{\varphi}(y)=\varphi(-y)$, the inverse  Dunkl transform will be
\[
\<\mathcal{F}_k^{-1}(f),\varphi\>=\<f,\mathcal{F}_k^{-1}(\varphi)\>=\<f,\mathcal{F}_k(\check{\varphi})\>,\quad f\in
\mathcal{S}'(\R^d),\quad \varphi\in \mathcal{S}(\R^d).
\]

Let $f,g\in\mathcal{S}'(\R^d)$. We have $\mathcal{F}_k^{-1}(\mathcal{F}_k(f))=\mathcal{F}_k(\mathcal{F}_k^{-1}(f))$, and $f=g$ iff $\mathcal{F}_k(f)=\mathcal{F}_k(g)$.

We define the $r$-th power of the Dunkl Laplacian $(-\Delta_k)^rf$ as follows
\[
\<(-\Delta_k)^rf,\varphi\>=\<f,(-\Delta_k)^r\varphi\>=\<f,\mathcal{F}_k^{-1}|\Cdot|^{2r}\mathcal{F}_k(\varphi))\>,\quad f\in
\mathcal{S}'(\R^d),\quad \varphi\in \mathcal{S}(\R^d).
\]

Let $W^{2r}_{p, k}$ be the Sobolev space, that is,
\[
W^{2r}_{p, k}=\{f\in L^{p}(\mathbb{R}^{d},d\mu_{k})\colon (-\Delta_k)^rf\in
L^{p}(\mathbb{R}^{d},d\mu_{k})\}
\]
equipped with the  Banach norm
\[
\|f\|_{W^{2r}_{p,k}}=\|f\|_{p,d\mu_k}+\|(-\Delta_k)^rf\|_{p,d\mu_k}.
\]
Note that $(-\Delta_k)^rf\in \mathcal{S}(\R^d)$ whenever $f\in \mathcal{S}(\R^d)$.

For $f\in\mathcal{S}'(\R^d)$ the generalized translation operators $\tau^yf,T^tf\in\mathcal{S}'(\R^d)$ are defined as follows
\begin{equation*}
\begin{gathered}
\<\tau^yf,\varphi\>=\<f,\tau^{-y}\varphi\>=\<f,\mathcal{F}_k^{-1}(e_k(-y,\Cdot)\mathcal{F}_k(\varphi))\>,
\quad \varphi\in \mathcal{S}(\R^d),\quad y\in\mathbb{R}^{d},\\
\<T^tf,\varphi\>=\<f,T^t\varphi\>=\<f,\mathcal{F}_k^{-1}(j_{\lambda_k}(t|\Cdot|)\mathcal{F}_k(\varphi))\>,
\quad \varphi\in \mathcal{S}(\R^d),\quad t\in\mathbb{R}_+.
\end{gathered}
\end{equation*}

For the Dunkl transform of the considered operators and their compositions we have the following easily verifiable equalities
\begin{equation}\label{distributiontransform}
\begin{gathered}
\mathcal{F}_k((-\Delta_k)^{r}f)=|\Cdot|^{2r}\mathcal{F}_k(f),\quad \mathcal{F}_k(\tau^yf)=e_k(y,\Cdot)\mathcal{F}_k(f),\\
\mathcal{F}_k((-\Delta_k)^{r}\tau^yf)=|\Cdot|^{2r}e_k(y,\Cdot)\mathcal{F}_k(f), \quad
\mathcal{F}_k(T^tf)=j_{\lambda_k}(t|\Cdot|)\mathcal{F}_k(f),\\
\mathcal{F}_k((-\Delta_k)^{r}T^tf)=|\Cdot|^{2r}j_{\lambda_k}(t|\Cdot|)\mathcal{F}_k(f),\\
\mathcal{F}_k(T^t(\tau^yf))=j_{\lambda_k}(t|\Cdot|)e_k(y,\Cdot)\mathcal{F}_k(f).
\end{gathered}
\end{equation}
This implies the commutativity of these compositions.

Let $\varphi\in\mathcal{S}(\R^d)$, $\check{\varphi}(y)=\varphi(-y)$. We call $f\in\mathcal{S}'(\R^d)$ even if $\<f,\check{\varphi}\>=\<f,\varphi\>$. Note that
$f\in\mathcal{S}'(\R^d)$ is even iff $\mathcal{F}(f)$ is even.

Let $N_k$ be a set of even $f\in\mathcal{S}'(\R^d)$ for which $\mathcal{F}_k(f)\in C^{\infty}_{\Pi}(\mathbb{R}^d)$.
For $f\in N_k$ and $\varphi\in\mathcal{S}(\R^d)$ we set
\[
(f\Ast{k}\varphi)(x)=\<\tau^{x}f,\check{\varphi}\>=\<f,\tau^{-x}\check{\varphi}\>.
\]

If $g\in N_k$, $\varphi\in\mathcal{S}(\R^d)$, then $(g\Ast{k}\varphi)\in\mathcal{S}(\R^d)$ and
\begin{equation}\label{convolution1}
\mathcal{F}_k(g\Ast{k}\varphi)(y)=\mathcal{F}_k(g)(y)\mathcal{F}_k(\varphi)(y).
\end{equation}
Indeed, we have
\begin{align*}
\tau^{-x}\check{\varphi}(y)&=\mathcal{F}_k^{-1}(e_k(-x,\Cdot)\mathcal{F}_k(\check{\varphi}))(y)\\&
=\int_{\mathbb{R}^d}e_k(-x,z)e_k(y,z)\mathcal{F}_k(\check{\varphi})(z)\,d\mu_k(z)\\&
=\int_{\mathbb{R}^d}e_k(-x,z)e_k(y,z)\mathcal{F}_k(\varphi)(-z)\,d\mu_k(z)\\&
=\int_{\mathbb{R}^d}e_k(x,z)e_k(-y,z)\mathcal{F}_k(\varphi)(z)\,d\mu_k(z)\\&
=\mathcal{F}_k(e_k(x,\Cdot)\mathcal{F}_k(\varphi))(y)=\tau^{x}\varphi(-y)\in\mathcal{S}(\R^d).
\end{align*}
Hence, by definition we get
\begin{align*}
(g\Ast{k}\varphi)(x)&=\<g,\tau^{-x}\check{\varphi}\>=\<g,\mathcal{F}_k(e_k(x,\Cdot)\mathcal{F}_k(\varphi))\>
=\<\mathcal{F}_k(g),e_k(x,\Cdot)\mathcal{F}_k(\varphi)\>\\&
=\int_{\mathbb{R}^d}e_k(x,z)\mathcal{F}_k(g)(z)\mathcal{F}_k(\varphi)(z)\,d\mu_k(z)\\&
=\mathcal{F}_k^{-1}(\mathcal{F}_k(g)\mathcal{F}_k(\varphi))(x)\in\mathcal{S}(\R^d)
\end{align*}
and the equality \eqref{convolution1}.

Now we can define a convolution $(f\Ast{k}g)\in\mathcal{S}'(\R^d)$ for $f\in\mathcal{S}'(\R^d)$ and $g\in N_k$ as follows
\begin{equation}\label{convolution2}
\<(f\Ast{k}g),\varphi\>=\<f,(g\Ast{k}\varphi)\,\>,\quad \varphi\in\mathcal{S}(\R^d).
\end{equation}
Applying \eqref{convolution1} and $\mathcal{F}_k(\mathcal{F}_k(\varphi))=\check{\varphi}$ for $\varphi\in\mathcal{S}(\R^d)$, we obtain
\begin{align*}
\<\mathcal{F}_k(f\Ast{k}g),\varphi\>&=\<(f\Ast{k}g),\mathcal{F}_k(\varphi)\>=\<f,(g\Ast{k}\mathcal{F}_k(\varphi))\>\\&
=\<f,\mathcal{F}_k^{-1}(\mathcal{F}_k(g)\mathcal{F}_k(\mathcal{F}_k(\varphi)))\>=\<f,\mathcal{F}_k^{-1}(\mathcal{F}_k(g)\check{\varphi})\>\\&
=\<f,\mathcal{F}_k(\mathcal{F}_k(g)\varphi)\>=\<\mathcal{F}_k(f),\mathcal{F}_k(g)\varphi\>=\<\mathcal{F}_k(g)\mathcal{F}_k(f),\varphi\>,
\end{align*}
hence
\begin{equation}\label{transformconvolution2}
\mathcal{F}_k(f\Ast{k}g)=\mathcal{F}_k(g)\mathcal{F}_k(f).
\end{equation}

The distribution $|\Cdot|^{2r}\in\mathcal{S}'(\R^d)$ is even,
$G_r=\mathcal{F}_k^{-1}(|\Cdot|^{2r})\in N_k$ and $(-\Delta)^{r}f=(f\Ast{k}G_r)$ for $f\in\mathcal{S}'(\R^d)$.If $g_1,g_2\in N_k$, then $(g_1\Ast{k}g_2)\in N_k$ and $(g_1\Ast{k}g_2)=(g_2\Ast{k}g_1)$. If $f\in\mathcal{S}'(\R^d)$ and $g_1,g_2\in N_k$ then
\begin{equation*}%\label{assconv}
(f\Ast{k}(g_1\Ast{k}g_2))=((f\Ast{k}g_1)\Ast{k}g_2).
\end{equation*}

We have
\begin{equation*}%\label{fracconv}
(-\Delta_k)^{r}(f\Ast{k}g)=((-\Delta_k)^{r}f\Ast{k}g).
\end{equation*}
Indeed, by \eqref{distributiontransform}, \eqref{transformconvolution2},
\begin{align*}
\mathcal{F}_k((-\Delta_k)^{r}(f\Ast{k}g))&=|\Cdot|^{2r}\mathcal{F}_k(f\Ast{k}g)=|\Cdot|^{2r}\mathcal{F}_k(f)\mathcal{F}_k(g)\\&
=\mathcal{F}_k((-\Delta_k)^{r}f)\mathcal{F}_k(g)=\mathcal{F}_k((-\Delta_k)^{r}f\Ast{k}g).
\end{align*}

\begin{lemma}\label{lem6.1}
If $f\in L^{p}(\mathbb{R}^{d},d\mu_{k})$, $g\in L^{1}_{\mathrm{rad}}(\mathbb{R}^d,d\mu_{k})$, $\mathcal{F}_k(g)\in N_k$, then both convolutions
\eqref{eq12} and \eqref{convolution2} of these functions coincide.
\end{lemma}

\begin{proof} Set
\[
(f\Ast{k}g)(x)=\int_{\mathbb{R}^d}f(y)\tau^{x}g(-y)\,d\mu_{k}(y).
\]
By \eqref{eq13} $(f\Ast{k}g)\in L^{p}(\mathbb{R}^{d},d\mu_{k})$ and $(f\Ast{k}g)\in \mathcal{S}'(\R^d)$.
It is sufficiently to prove the equality $\mathcal{F}_k(f\Ast{k}g)=\mathcal{F}_k(g)\mathcal{F}_k(f)$ in $\mathcal{S}'(\R^d)$. For any $\varphi\in \mathcal{S}(\R^d)$
we have
\begin{align*}
\<\mathcal{F}_k(f\Ast{k}g),\varphi\>&=\<(f\Ast{k}g),\mathcal{F}_k(\varphi)\>\\&=
\int_{\mathbb{R}^d}\int_{\mathbb{R}^d}f(y)\tau^{x}g(-y)\,d\mu_{k}(y)\mathcal{F}_k(\varphi)(x)\,d\mu_{k}(x)\\&
=\int_{\mathbb{R}^d}f(y)\int_{\mathbb{R}^d}\tau^{-y}g(x)\mathcal{F}_k(\varphi)(x)\,d\mu_{k}(x)d\mu_{k}(y).
\end{align*}
Since
\begin{align*}
&\int_{\mathbb{R}^d}\tau^{-y}g(x)\mathcal{F}_k(\varphi)(x)\,d\mu_{k}(x)\\&=
\int_{\mathbb{R}^d}\int_{\mathbb{R}^d}e_k(-y,z)e_k(x,z)\mathcal{F}_k(g)(z)\,d\mu_{k}(z)\mathcal{F}_k(\varphi)(x)\,d\mu_{k}(x)
\\&=\int_{\mathbb{R}^d}e_k(-y,z)\mathcal{F}_k(g)(z)\varphi(z)\,d\mu_{k}(z)=\mathcal{F}_k(\mathcal{F}_k(g)\varphi)(y),
\end{align*}
then
\begin{align*}
\<\mathcal{F}_k(f\Ast{k}g),\varphi\>&=\int_{\mathbb{R}^d}f(y)\mathcal{F}_k(\mathcal{F}_k(g)\varphi)\,d\mu_{k}
=\<f,\mathcal{F}_k(\mathcal{F}_k(g)\varphi)\>
\\&=\<\mathcal{F}_k(f),\mathcal{F}_k(g)\varphi\>=\<\mathcal{F}_k(g)\mathcal{F}_k(f),\varphi\>.
\end{align*}
Lemma~\ref{lem6.1} is proved.
\end{proof}

Define the $K$-functional for the couple $(L^{p}(\mathbb{R}^{d},d\mu_{k}), W^{2r}_{p, k})$
as follows\[
K_{2r}(t,
f)_{p,d\mu_k}=\inf\{\|f-g\|_{p,d\mu_k}+t^{2r}\|(-\Delta_k)^rg\|_{p,d\mu_k}\colon
g\in W^{2r}_{p, k}\}.
\]
Note that
for any $f_1,f_2\in L^{p}(\mathbb{R}^{d},d\mu_{k})$ and $g\in W^{2r}_{p, k}$, we have
\begin{align*}
&\|f_1-g\|_{p,d\mu_k}+t^{2r}\|(-\Delta_k)^rg\|_{p,d\mu_k}\\
&\qquad \leq \|f_2-g\|_{p,d\mu_k}+t^{2r}\|(-\Delta_k)^rg\|_{p,d\mu_k}+\|f_1-f_2\|_{p,d\mu_k}
\end{align*}
and hence,
\begin{equation}
|K_{2r}(t, f_1)_{p,d\mu_k}-K_{2r}(t, f_2)_{p,d\mu_k}|\leq \|f_1-f_2\|_{p,d\mu_k}.\label{eq48}
\end{equation}
If $f\in W^{2r}_{p, k}$, then $K_{2r}(t, f)_{p,d\mu_k}\leq t^{2r}\|(-\Delta_k)^rf\|_{p,d\mu_k}$ and
$\lim_{t\to 0}K_{2r}(t, f)_{p,d\mu_k}=0$. This and \eqref{eq48} imply that, for any
 $f\in L^{p}(\mathbb{R}^{d},d\mu_{k})$,
\begin{equation}
\lim_{t\to 0}K_{2r}(t, f)_{p,d\mu_k}=0.\label{eq49}
\end{equation}

Another important property of the $K$-functional is
\begin{equation}
K_{2r}(\lambda t, f)_{p,d\mu_k}\leq \max\{1,\,\lambda^{2r}\}K_{2r}(t, f)_{p,d\mu_k}. \label{eq50}
\end{equation}

Let $I$ be an identical operator and $m\in \mathbb{N}$. Consider the following three differences:
\begin{equation}
\varDelta_t^mf(x)=(I-T^t)^mf(x)= \sum_{s=0}^m(-1)^s\binom{m}{s}(T^{t})^sf(x),\label{eq51}
\end{equation}
\begin{equation}
\aDelta_t^mf(x)=\sum_{s=0}^m(-1)^s\binom{m}{s}T^{st}f(x),\label{eq52}
\end{equation}
\begin{equation}
\aaDelta_t^mf(x)=\binom{2m}{m}^{-1}\sum_{s=-m}^m(-1)^s
\binom{2m}{m-s}T^{st}f(x).\label{eq53}
\end{equation}
Differences \eqref{eq51} and \eqref{eq52} coincide with the classical
difference for the translation operator $T^{t}f(x)=f(x+t)$ and correspond to
the usual definition of the modulus of smoothness of order $m$. Difference
\eqref{eq53} can be seen as follows. Define $\mu_s=(-1)^s\binom{m}{s}$, $s\in
\Z$. Then the convolution $\mu*\mu$ is given by
\[
\nu_s:=(\mu*\mu)_s=\sum_{l\in\Z}\mu_l\mu_{s+l}=(-1)^s\binom{2m}{m-s}.
\]
 Note that $\nu_s\ne 0$ if $|s|\leq m$. Moreover, if $k\equiv0$, then
\[
\frac{1}{\nu_0}\sum_{s=-m}^m\nu_sT^{st}f(x)=f(x)+\frac{2}{\nu_0}\sum_{s=1}^m\nu_sS^{st}f(x)=f(x)-V_{m, t}f(x),
\]
where the operator $S^{t}$ was given in \eqref{s-operator} and the averages
\[
V_{m, t}f(x)=\frac{-2}{\nu_0}\sum_{s=1}^m\nu_sS^{st}f(x)
\]
were defined by F.~Dai and Z.~Ditzian in \cite{DaiDit04}.

\begin{definition}\label{moduli}
The moduli of smoothness of a function $f\in L^{p}(\mathbb{R}^{d},d\mu_{k})$ are defined by
\begin{equation}
\omega_m(\delta, f)_{p,d\mu_k}=\sup_{0<t\leq\delta}\|\varDelta_t^mf(x)\|_{p,d\mu_k},\label{eq54}
\end{equation}
\begin{equation}
\aomega_m(\delta, f)_{p,d\mu_k}=\sup_{0<t\leq\delta}\|\aDelta_t^mf(x)\|_{p,d\mu_k}, \label{eq55}
\end{equation}
\begin{equation}
\aaomega_m(\delta, f)_{p,d\mu_k}=\sup_{0<t\leq\delta}\|\aaDelta_t^mf(x)\|_{p,d\mu_k}. \label{eq56}
\end{equation}
\end{definition}
Let us mention some basic properties of these moduli of smoothness. Define by
$\Omega_m(\delta,
f)_{p,d\mu_k}$ any of the three moduli in Definition \ref{moduli}. %\eqref{eq54}--\eqref{eq56}.
Using
 the triangle inequality, estimate \eqref{eq6} reveals
\begin{equation}\label{eq57}
\begin{gathered}
\Omega_m(\delta, f_1+f_2)_{p,d\mu_k}\leq \Omega_m(\delta, f_1)_{p,d\mu_k}+\Omega_m(\delta, f_2)_{p,d\mu_k},\\
\Omega_m(\delta, f)_{p,d\mu_k}\lesssim \|f\|_{p,d\mu_k}.
\end{gathered}
\end{equation}
If $f\in \mathcal{S}^{'}(\R^d)$, then, by \eqref{distributiontransform},
\begin{equation}
\begin{aligned}
\mathcal{F}_k(\varDelta_t^mf)(y)&=j_{\lambda_{k},m}(t|y|)\mathcal{F}_k(f)(y),
\\
\mathcal{F}_k(\aDelta_t^rf)(y)&=j_{\lambda_{k},m}^*(t|y|)\mathcal{F}_k(f)(y),
\\
\mathcal{F}_k(\aaDelta_t^mf)(y)&=\aaj_{\lambda_{k},m}(t|y|)\mathcal{F}_k(f)(y),
\end{aligned}
\label{eq58}
\end{equation}
where $\lambda_k=d/2-1+\sum_{a\in R_+}k(a)>-1/2$,
\begin{align*}
j_{\lambda_{k},m}(t)&=\sum_{s=0}^m(-1)^s\binom{m}{s}\bigl(j_{\lambda_{k}}(t)\bigr)^s=(1-j_{\lambda_{k}}(t))^m,
\\
j_{\lambda_{k},m}^*(t)&=\sum_{s=0}^m(-1)^s\binom{m}{s}j_{\lambda_{k}}(st),
\end{align*}
and
\begin{align}
\aaj_{\lambda_{k},m}(t)&=\binom{2m}{m}^{-1}\sum_{s=-m}^m(-1)^s\binom{2m}{m-s}j_{\lambda_{k}}(st)\notag\\
&=1+2\binom{2m}{m}^{-1}\sum_{s=1}^m(-1)^s\binom{2m}{m-s}j_{\lambda_{k}}(st). \label{eq59}
\end{align}

Since $j_{\lambda_{k},m}(t|y|),j_{\lambda_{k},m}^*(t|y|),\aaj_{\lambda_{k},m}(t|y|)\in C^{\infty}_{\Pi}(\mathbb{R}^d)$, then all differences and their Dunkl transforms are tempered distributions.

Let us prove the following remark, which will be important further in Theorem \ref{thm6.2}.

\begin{remark}\label{rem6.1}
 The functions $j_{\lambda_{k},m}(t)$ and $\aaj_{\lambda_{k},m}(t)$
have zero of order $2m$ at the origin, while the function
$\aj_{\lambda_{k},m}(t)$ has zero of order $m+1$ if $m$ is odd and of order $m$
if $m$ is even.
\end{remark}
Indeed, first we study $j_{\lambda_{k},m}(t)=(1-j_{\lambda_{k}}(t))^{m}$. Since, for
any $t$,
\begin{equation}\label{j-ser}
j_{\lambda}(t)=\sum_{k=0}^{\infty}\frac{(-1)^{k}\Gamma(\lambda+1)(t/2)^{2k}}{k!\,\Gamma(k+\lambda+1)},
\end{equation}
we get $j_{\lambda_{k},m}(t)\asymp t^{2m}$ as $t\to 0$.
Second, since
\[
\sum_{s=0}^m(-1)^s\binom{m}{s}s^{2k}=0, \qquad 0\le 2k\le m-1,
\]
(see \cite[Sect.~4.2]{prud}), using \eqref{j-ser}, we obtain
that
$j_{\lambda_{k},m}^*(t)\asymp t^{2[(m+1)/2]}$.
Finally, taking into account
% $\sum_{s=1}^m(-1)^s\binom{2m}{m-s}=-\frac{1}{2}\binom{2m}{m}$ and
%$\sum_{s=1}^m(-1)^s\binom{2m}{m-s}s^{2k}=0$, $k=1,\dots,m-1$ (see там же),
%using \eqref{j-ser}, we have $\aaj_{\lambda_{k},m}(t)\asymp t^{2m}$.
\begin{equation*}
\begin{aligned}
&\sum_{s=1}^m(-1)^s\binom{2m}{m-s}=-\frac{1}{2}\binom{2m}{m},
\\
&\sum_{s=1}^m(-1)^s\binom{2m}{m-s}s^{2k}=0,\qquad k=1,\dots,m-1,
\end{aligned}
\end{equation*}
 (see \cite[Sect.~4.2]{prud}) and using again \eqref{j-ser},
 we arrive at $\aaj_{\lambda_{k},m}(t)\asymp t^{2m}$.
Some of these properties were known (see \cite{Pla07,Pla09,DaiDit04}).
% These properties define properties of moduli
%of smoothness.
\begin{remark}\label{rem6.2}
In the paper \cite{DaiDit04}, the authors obtained that $\aaj_{\lambda_{k},m}(t)>0$ for $t>0$.
\end{remark}

\begin{lemma}\label{lem6.2}
If $m,r\in \N$, $1\leq p\leq\infty$, and $f\in L^{p}(\mathbb{R}^{d},d\mu_{k})$, then
\begin{equation}\label{omega-inequality2}
\omega_{m+r}(\delta, f)_{p,d\mu_k}\leq 2^r\omega_{m}(\delta, f)_{p,d\mu_k}.
\end{equation}
\end{lemma}

\begin{proof} Since for $f\in\mathcal{S}'(\R^d)$ by~\eqref{distributiontransform}, \eqref{eq58},
\begin{align*}
\mathcal{F}_k(\varDelta_t^{m+r}f)&=(1-j_{\lambda_{k}}(t))^{m+r}\mathcal{F}_k(f)\\&
=(1-j_{\lambda_{k}}(t))^{r}(1-j_{\lambda_{k}}(t))^{m}\mathcal{F}_k(f)\\&
=\sum_{s=0}^{r}(-1)^s\binom{r}{s}(j_{\lambda_{k}}(t))^s\mathcal{F}_k(\varDelta_t^{m}f),
\end{align*}
then
\[
\varDelta_t^{m+r}f=\sum_{s=0}^{r}(-1)^s\binom{r}{s}(T^{t})^s(\varDelta_t^{m}f).
\]
Using for $f\in L^{p}(\mathbb{R}^{d},d\mu_{k})$ Theorem~\ref{thm3.3} and \eqref{eq57}, we get
\[
\|\varDelta_t^{m+r}f\|_{p,d\mu_k}\leq \sum_{s=0}^{r}\binom{r}{s}\|\varDelta_t^{m}f\|_{p,d\mu_k}=2^r\|\varDelta_t^{m}f\|_{p,d\mu_k}.
\]
\end{proof}

\subsection{Main results}
First we state the Jackson-type inequality.

\begin{theorem}\label{thm6.1}
Let $\sigma> 0$, $1\leq p \leq \infty$, $r\in \mathbb{Z}_+$, $m\in \mathbb{N}$. We have,
for any $f\in W_{p, k}^{2r}$,
\begin{equation}
E_{\sigma}(f)_{p,d\mu_k} \lesssim
\frac{1}{\sigma^{2r}}\,\Omega_m\Bigl(\frac{1}{\sigma},
(-\Delta_k)^rf\Bigr)_{p,d\mu_k}, \label{eq60}
\end{equation}
where $\Omega_m$ is any of the three moduli of smoothness \eqref{eq54}--\eqref{eq56}.
\end{theorem}
\begin{remark}\label{rem6.3}
(i) For radial functions inequality \eqref{eq60} is the Jackson inequality in $L^{p}(\R_{+},d\nu_{\lambda_k})$.
In this case it was obtained in \cite{Pla07,Pla09} for moduli
\eqref{eq54} and \eqref{eq55}. For $k\equiv 0$ and the modulus of smoothness \eqref{eq56},
inequality \eqref{eq60}
was obtained by F. Dai and Z. Ditzian \cite{DaiDit04}, see also the paper \cite{DaiDitTik08}.

(ii) From the proof of Theorem
\ref{thm6.1} we will see that inequality \eqref{eq60} for
moduli \eqref{eq54} and \eqref{eq56} can be equivalently written as
\[
E_{\sigma}(f)_{p,d\mu_k}
\lesssim
\frac{1}{\sigma^{2r}}
\|\varDelta_{1/\sigma}^m ((-\Delta_k)^rf) \|_{p,d\mu_k},
\]
\[
E_{\sigma}(f)_{p,d\mu_k}
\lesssim
\frac{1}{\sigma^{2r}}
\|\aaDelta_{1/\sigma}^m ((-\Delta_k)^rf) \|_{p,d\mu_k}.
\]
\end{remark}

The next theorem provides an equivalence between moduli of smoothness and the $K$-functional.
\begin{theorem}\label{thm6.2}
If $\delta> 0$, $1\leq p\leq \infty$, $r\in \mathbb{N}$, then for any $f\in
L^{p}(\mathbb{R}^{d},d\mu_{k})$

\begin{equation}\label{eq61}
\begin{aligned}
K_{2r}(\delta, f)_{p,d\mu_k}&\asymp \omega_r(\delta, f)_{p,d\mu_k}\asymp
\aaomega_r(\delta, f)_{p,d\mu_k}\\
&\asymp\aomega_{2r-1}(\delta,
f)_{p,d\mu_k}\asymp\aomega_{2r}(\delta, f)_{p,d\mu_k}.
\end{aligned}
\end{equation}

\end{theorem}
\begin{remark}\label{rem6.4}
If $k\equiv0$, the equivalence between the classical modulus of smoothness and
the $K$-functional is well known \cite{johnen, DevoreLorentz}, while the
equivalence between modulus \eqref{eq56} and the $K$-functional was shown in
\cite{DaiDit04}. For radial functions a partial result of \eqref{eq61}, more
precisely, an equivalence of the $K$-functional and moduli of smoothness
\eqref{eq54} and \eqref{eq55} was proved in \cite{Pla07,Pla09}.
\end{remark}
\begin{remark}\label{rem6.5}
One can continue equivalence \eqref{eq61} as follows (see also Remark \ref{rem6.8})
\[
\ldots \asymp
\|\varDelta_\delta^r f \|_{p,d\mu_k}
\asymp
\|\aaDelta_\delta^r f \|_{p,d\mu_k}.
\]
\end{remark}

We give the proof for the difference \eqref{eq53} and the modulus of smoothness
\eqref{eq56}. We partially follow the proofs in \cite{LiSuIva11,Pla07,Pla09}
which are different from those given  in \cite{DaiDit04}. For moduli of smoothness
\eqref{eq54} and \eqref{eq55} the proofs are similar and will be omitted here
(see also \cite{Pla07,Pla09}). The proof makes use of radial multipliers and
based on boundedness of the translation operator $T^t$.

\subsection{Properties of the de la Vall\'{e}e Poussin type operators }

Let $\eta\in \mathcal{S}_\text{rad}(\R^d)$ be such that $\eta(x)=1$ if $|x|\leq 1$,
$\eta(x)>0$ if $|x|<2$, and $\eta(x)=0$ if $|x|\geq 2$. Denote
\[
\eta_r(x)=\frac{1-\eta(x)}{|x|^{2r}},\quad \widehat{\eta}_{k,r}(y)=\mathcal{F}_{k}(\eta_r)(y),
\]
where $\mathcal{F}_{k}(\eta_r)$ is a tempered distribution.
If $t=|x|$, $\eta_0(t)=\eta(x)$, and $\eta_{r 0}(t)=\eta_r(x)$, then
$\mathcal{F}_{k}(\eta_r)(y)=\mathcal{H}_{\lambda_k}(\eta_{r 0})(|y|)$.

\begin{lemma}\label{lem6.3}
We have $\widehat{\eta}_{k,r}\in L^{1}(\R^d, d\mu_k)$, where $r>0$.
\end{lemma}

\begin{proof} It is sufficient to prove that
 $\mathcal{H}_{\lambda_k}(\eta_{r 0})\in L^{1}(\R_{+},d\nu_{\lambda_k})$. In the case $r\geq 1$ this was proved in \cite[(4.25)]{Pla09}.
We give the proof for any $r>0$.

Letting $u_j(t)=(1+t^2)^{-j}$ and taking into account that
\[
\frac{1}{t^{2r}}=\frac{1}{(1+t^2)^r}\Bigl(1-\frac{1}{1+t^2}\Bigr)^{-r}=
\sum_{j=0}^{\infty}\binom{j+r-1}{j}\frac{1}{(1+t^2)^{j+r}},\quad t\ne 0,
\]
we obtain, for any $M\in\N$ and $t\ge 0$,
\begin{align*}
\eta_{r 0}(t)&=\sum_{j=0}^{\infty}\binom{j+r-1}{j}(1-\eta_0(t))u_{j+r}(t)
\\ &=\sum_{j=0}^{M-1}\binom{j+r-1}{j}u_{j+r}(t)-
\eta_0(t)\sum_{j=0}^{M-1}\binom{j+r-1}{j}u_{j+r}(t)\\ &+
\sum_{j=M}^{\infty}\binom{j+r-1}{j}(1-\eta_0(t))u_{j+r}(t)
=:\psi_1(t)+\psi_2(t)+\psi_3(t).
\end{align*}

Since for any $r>0$ we have $\mathcal{H}_{\lambda_k}(u_r)\in
L^{1}(\R_{+},d\nu_{\lambda_k})$ (see \cite[Lemma 3.2]{Pla09}, \cite[Chapt 5,
5.3.1]{Ste70}, \cite[Chapt 8, 8.1]{Nik75}), then
$\mathcal{H}_{\lambda_k}(\psi_1)\in L^{1}(\R_{+},d\nu_{\lambda_k})$. Because of
$\psi_2\in\mathcal{S}(\R_+)$, then $\mathcal{H}_{\lambda_k}(\psi_2)\in
L^{1}(\R_{+},d\nu_{\lambda_k})$. Thus, we are left to show that, for
sufficiently large $M$, $\mathcal{H}_{\lambda_k}(\psi_3)\in
L^{1}(\R_{+},d\nu_{\lambda_k})$.

Let $M+r> \lambda_k+1$, $t\geq 1$. Since
$\frac{\Gamma(j+r)}{\Gamma(j+1)}\lesssim j^{r-1}$, we have
\[
|\psi_3(t)|\leq \frac{M^{r-1}}{(1+t^2)^{M+r}}\sum_{j=0}^{\infty}\frac{(1+j)^{r-1}}{2^{j}}\lesssim \frac{1}{(1+t^2)^{M+r}}\lesssim\frac{1}{t^{2M+2r}},
\]
and
\[
\int_0^{\infty}|\psi_3(t)|\,d\nu_{\lambda_k}(t)\lesssim \int_1^{\infty}t^{-(2M+2r-2\lambda_k-1)}\,dt<\infty.
\]
Thus, $\psi_3\in L^{1}(\R_{+},d\nu_{\lambda_k})$, $\mathcal{H}_{\lambda_k}(\psi_3)\in C(\R_{+})$, and $\mathcal{H}_{\lambda_k}(\psi_3)\in L^{1}([0,2],d\nu_{\lambda_k})$.

Recall that the Bessel differential operator is defined by
\[
\mathcal{B}_{\lambda_k}=\frac{d^2}{dt^2}+\frac{(2\lambda_k+1)}{t}\frac{d}{dt}.
\]
Using $\psi_3\in C^{\infty}(\R_{+})$, we have, for any $s\in \N$, $\mathcal{B}_{\lambda_k}^s\psi_3\in L^{1}([0,2],d\nu_{\lambda_k})$.

If $t\geq 2$, then $(1-\eta_0(t))u_{j+r}(t)=u_{j+r}(t)$ and
\[
\mathcal{B}_{\lambda_k}u_{j+r}(t)=4(j+r)(j+r-\lambda_k)u_{j+r+1}(t)-4(j+r)(j+r+1)u_{j+r+2}(t).
\]
This gives \[
|\mathcal{B}_{\lambda_k}u_{j+r}(t)|\leq 2^3(j+r+\lambda_k+1)^2u_{j+r+1}(t).
\]
By induction on $s$,
\[
|\mathcal{B}_{\lambda_k}^su_{j+r}(t)|\leq 2^{3s}(j+r+2s+\lambda_k-1)^{2s}u_{j+r+s}(t),
\]
and then, for $t\geq 2$,
\[
|\mathcal{B}_{\lambda_k}^s\psi_3(t)|\lesssim \frac{1}{(1+t^2)^{M+r+s}}\sum_{j=0}^{\infty}\frac{(1+j)^{r+2s-1}}{5^{j}}\lesssim \frac{1}{(1+t^2)^{M+r+s}}\lesssim\frac{1}{t^{2M+2r+2s}},
\]
and $\mathcal{B}_{\lambda_k}^s\psi_3\in L^{1}([2,\infty),d\nu_{\lambda_k})$.
Thus, we have
$\mathcal{B}_{\lambda_k}^s\psi_3\in L^{1}(\R_+,d\nu_{\lambda_k})$ for any $s$. Choosing $s>\lambda_k+1$ and using the inequality
\[
|\mathcal{H}_{\lambda_k}(\psi_3)(\tau)|\leq\frac{1}{\tau^{2s}}
\int_{0}^{\infty}|\mathcal{B}_{\lambda_k}^s\psi_3(t)|\,d\nu_{\lambda_k}(t)\lesssim \frac{1}{\tau^{2s}},
\]
we arrive at $\mathcal{H}_{\lambda_k}(\psi_3)\in L^{1}([2,\infty),d\nu_{\lambda_k})$.
Finally, we obtain that $\mathcal{H}_{\lambda_k}(\psi_3)\in L^{1}(\R_+,d\nu_{\lambda_k})$.
\end{proof}

For $m, r\in\mathbb{N}$ and $m\geq r$, we set
\[
g_{m, r}^*(y):= |y|^{-2r}\aaj_{\lambda_{k},m}(|y|),\quad
g_{m,r}(x):=\mathcal{F}_{k}(g_{m, r}^*)(x),
\]
\[
g_{m,r}^t(x) :=t^{2r-2\lambda_k-2}g_{m,r}\Bigl(\frac{x}{t}\Bigr).
\]
Since
\[
g_{m,r}^*(y)=\aaj_{\lambda_{k},m}(|y|)\eta_r(y)+\frac{\aaj_{\lambda_{k},m}(|y|)}{|y|^{2r}}\,\eta(y),
\]
\[
\frac{\aaj_{\lambda_{k},m}(|y|)}{|y|^{2r}}\in
C^{\infty}_{\Pi}(\R^d),\quad\frac{\aaj_{\lambda_{k},m}(|y|)}{|y|^{2r}}\,\eta(y)\in
\mathcal{S}(\R^d),
\]
and
\[
\mathcal{F}_{k}(\aaj_{\lambda_{k},m}\eta_r)(x)=\binom{2m}{m}^{-1}
\sum_{s=-m}^m(-1)^s\binom{2m}{m-s}T^{s}\widehat{\eta}_{\lambda_k,r}(x),
\]
then boundedness of the operator $T^s$ in $L^{1}(\R^d,d\mu_k)$ and Lemma
\ref{lem6.3} imply that
\begin{equation}\label{eq62}
\begin{gathered}
g_{m,r},\,g_{m,r}^t\in L^{1}(\R^d,d\mu_{k}),\quad \|g_{m,r}^t\|_{1,d\mu_{k}}=t^{2r}\|g_{m,r}\|_{1,d\mu_{k}},
\\
\mathcal{F}_{k}^{-1}(g_{m,r}^t)(y)=
\mathcal{F}_{k}(g_{m,r}^t)(y)=t^{2r}g_{m,r}^*(ty)=|y|^{-2r}\aaj_{\lambda_{k},m}(t|y|).
\end{gathered}
\end{equation}

\begin{lemma}\label{lem6.4}
Let $m, r \in \mathbb{N}$, $m \geq r$, $1\leq p\leq\infty$, and $f\in
W_{p, k}^{2r}$. We have
\begin{equation}
\aaDelta_t^mf = (-\Delta_k)^rf\Ast{k}g_{m,r}^t \label{eq63}
\end{equation}
and
\begin{equation}
\|\aaDelta_t^mf\|_{p,d\mu_k}\lesssim t^{2r}\|(-\Delta_k)^rf\|_{p,d\mu_k}.\label{eq64}
\end{equation}
\end{lemma}

\begin{proof}
Let $f\in \mathcal{S}^{'}(\R^d)$. Applying \eqref{eq58}, \eqref{distributiontransform}, \eqref{transformconvolution2}, and \eqref{eq62}, we obtain
\begin{align*}
\mathcal{F}_k(\aaDelta_t^mf)(y)&=\aaj_{\lambda_{k},m}(t|y|)\mathcal{F}_k(f)(y)
=|y|^{2r}\mathcal{F}_k(f)(y)\frac{\aaj_{\lambda_{k},m}(t|y|)}{|y|^{2r}}\\
&=\mathcal{F}_k((-\Delta_k)^rf)(y)
\mathcal{F}_k(g_{m,r}^t)(y)
\end{align*}
and \eqref{eq63}. If $f\in W_{p, k}^{2r}$, then $(-\Delta_k)^rf\in L^{p}(\mathbb{R}^{d},d\mu_{k})$. Inequality \eqref{eq64} follows from \eqref{eq62}, \eqref{eq63}, Lemma
\ref{lem6.1}, and \eqref{eq13}. Note that a constant in \eqref{eq64} can be taken as $\|g_{m,r}\|_{1,d\mu_{k}}$.
\end{proof}

%\begin{remark}\label{rem6.6}
%Since $\mathcal{S}(\R^d)$ is dense in $W_{p, k}^{2r}$, in light of \eqref{eq57}, inequality
%\eqref{eq64} holds for any function from $W_{p, k}^{2r}$.
%\end{remark}

Let $f\in\mathcal{S}^{'}(\R^d)$. We set $\theta(x)=\mathcal{F}_k(\eta)(x)$ and
$\theta_\sigma(x) = \mathcal{F}_k(\eta(\Cdot/\sigma))(x)$. Then $\theta$, $\theta_\sigma\in
\mathcal{S}(\R^d)$.
The de la Vall\'{e}e Poussin type operator is given by $P_\sigma (f)= f\Ast{k}\theta_\sigma$.
By \eqref{transformconvolution2},
\begin{equation}\label{approxtransform}
\mathcal{F}_k(P_\sigma(f))=\eta(\Cdot/\sigma)\mathcal{F}_k (f)(y).
\end{equation}

\begin{lemma}\label{lem6.5}
If $\sigma >0$, $1 \leq p\leq\infty$, $f\in L^{p}(\mathbb{R}^{d},d\mu_{k})$, then

\smallbreak
\textup{(1)} $\|P_\sigma (f)\|_{p,d\mu_{k}} \lesssim \|f\|_{p,d\mu_{k}}$;

\smallbreak
\textup{(2)} $P_\sigma (f)\in B_{p, k}^{2\sigma}$ and $P_\sigma(g)=g$ for any $g\in B_{p, k}^\sigma$;

\smallbreak
\textup{(3)} $\|f-P_\sigma(f)\|_{p,d\mu_{k}} \lesssim E_{\sigma} (f)_{p,d\mu_{k}}$.
\end{lemma}

\begin{remark}\label{rem6.7}
Property (3) in this lemma means that $P_\sigma(f)$ is \textit{the near best approximant} of
$f$ in $L^p(\R^d,d\mu_{k})$.
\end{remark}
\begin{proof}

(1) Since $f\in L^{p}(\mathbb{R}^{d},d\mu_{k})$, $\eta(\Cdot/\sigma),\theta_\sigma\in \mathcal{S}(\R^d)$ then using Lemma~\ref{lem6.1}, \eqref{eq13}, and the equality
$\|\theta\|_{1,d\mu_{k}}=\|\theta_\sigma \|_{1,d\mu_{k}}$, we get
\begin{align*}
\|P_\sigma (f)\|_{p,d\mu_{k}}&=\|f\Ast{k}\theta_\sigma\|_{p,d\mu_{k}}\leq \|\theta_\sigma \|_{1,d\mu_{k}}
\|f\|_{p,d\mu_{k}}\\
&=\|\theta\|_{1,d\mu_{k}} \|f\|_{p,d\mu_{k}}\lesssim\|f\|_{p,d\mu_{k}}.
\end{align*}

(2) We observe that $\supp \eta(\Cdot/\sigma)\subset B_{2\sigma}$ and then
$\supp \mathcal{F}_k(P_\sigma(f))\subset B_{2\sigma}$. Theorem \ref{thm5.10}
yields $P_\sigma (f)\in B_{p, k}^{2\sigma}$. If $g\in B_{p,k}^\sigma$, then by
Theorem~\ref{thm5.10}, $\supp \mathcal{F}_k(g)\subset B_{\sigma}$ and
$\mathcal{F}_k(P_\sigma(g))(y)=\eta(y/\sigma)\mathcal{F}_k(g)(y)=\mathcal{F}_k
(g)(y).$ Hence, $P_\sigma(g)=g$.

(3) Using Theorem \ref{thm5.11}, there exists an entire function $g^*\in B_{p, k}^\sigma$ such that
 $\|f-g^*\|_{p,d\mu_{k}}=E_\sigma(f)_{p,d\mu_{k}}$. Then using
$P_\sigma(g^*)=g^*$ implies
\begin{align*}
\|f- P_\sigma(f)\|_{p,d\mu_{k}} &= \|f-g^*+P_\sigma (g^*-f)\|_{p,d\mu_{k}}\\
&\leq \|f-g^*\|_{p,d\mu_{k}}+\|P_\sigma(f-g^*)\|_{p,d\mu_{k}} \lesssim E_\sigma (f)_{p,d\mu_{k}}.
\qedhere
\end{align*}
\end{proof}

In the proof of the next lemma we will use the estimate
\begin{equation}
|j^{(n)}_{\lambda}(t)|\lesssim (|t|+1)^{-(\lambda+1/2)},\quad t\in \R,\ \lambda\geq-1/2,\ n\in \Z_+,
\label{eq65}
\end{equation}
which follows, by induction on $n$, from the known properties of the Bessel function (\cite{BatErd53})
\[
|j_{\lambda}(t)|\lesssim (|t|+1)^{-(\lambda+1/2)},\quad
j'_{\lambda}(t)=-\frac{t}{2(\lambda+1)}j_{\lambda+1}(t).
\]

\begin{lemma}\label{lem6.6}
If $\sigma >0$, $1 \leq p\leq\infty$, $m\in \N$, $r\in \Z_+$, $f\in W_{p, k}^{2r}$, then
\begin{equation}
\|f-P_{\sigma/2}(f)\|_{p,d\mu_{k}} \lesssim \sigma^{-2r}\|\aaDelta_{a/\sigma}^m((-\Delta_k)^rf)\|_{p,d\mu_k}\label{eq66}
\end{equation}
for some $a=a(\lambda_k, m)>0$.
\end{lemma}

\begin{proof}
If $f\in \mathcal{S}^{'}(\R^d)$, then by \eqref{approxtransform}, \eqref{eq58}, and \eqref{distributiontransform}
\begin{align*}
\mathcal{F}_k(f-P_{\sigma/2}(f))&=(1-\eta(2\Cdot/\sigma))\mathcal{F}_kf\notag\\
&=\sigma^{-2r}\frac{1-\eta(2\Cdot/\sigma)}{(|\Cdot|/\sigma)^{2r}
\aaj_{\lambda_{k},m}(a|\Cdot|/\sigma)}\mathcal{F}_k(\aaDelta_{a/\sigma}^m
((-\Delta_k)^r f))\notag\\ &=\sigma^{-2r} \varphi
(\Cdot/\sigma)\mathcal{F}_k(\aaDelta_{a/\sigma}^m ((-\Delta_k)^r
f)),\label{eq67}
\end{align*}
where
\begin{equation}
\varphi (y)= \frac{1-\eta(2y)}{|y|^{2r} \aaj_{\lambda_{k},m}(a|y|)}, \quad
\aaDelta_{a/\sigma}^m ((-\Delta_k)^r f)\in L^{p}(\mathbb{R}^{d},d\mu_{k}).\label{eq68}
\end{equation}
From here and by \eqref{transformconvolution2} we have
\begin{equation}\label{eq67}
f-P_{\sigma/2}(f)=(\aaDelta_{a/\sigma}^m ((-\Delta_k)^r
f)\Ast{k}\mathcal{F}_k(\varphi(\Cdot/\sigma))).
\end{equation}

Setting $\aaj_{\lambda_{k},m}(t)= 1- \tau_0(t)$, in light of \eqref{eq59}
and \eqref{eq65}, we observe that $\aaj_{\lambda_{k},m}(t)\to 1$ as $t\to
\infty$. Then we can choose $a>0$ such that $|\tau_0(t)| \leq 1/2$ for $|t|
\geq a/2$. For such $a=a(\lambda_k, m)$, we have that $\varphi(y)=0$ for
$|y|\leq 1/2$, $\varphi(y)>0$ for $|y|> 1/2$, and $\varphi\in
C^{\infty}_{\Pi}(\mathbb{R}^d)$.

%Moreover, the derivatives $\varphi^{(k)}(y)$ grow at
%infinity not faster than $|y|^{a_k}$, which yields $\varphi\in
%\mathcal{S}'(\R^d)$.
%{\bf и ее производные растут на $\infty$ не быстрее некоторой степени $|y|$,
%в частности, $\varphi\in \mathcal{S}'(\R^d)$.}

We will use the following decomposition
\[
\varphi(y)=\varphi_1(|y|)+\varphi_2(|y|),
\]
where
\[
\varphi_1(|y|)=2^{2r}\eta_{r}(2y)\Bigl(\frac{1}{1- \tau_0(a|y|)}-S_N(\tau_0(a|y|)\Bigr)
\]
and
\[
\varphi_2(|y|)=2^{2r}\eta_{r}(2y)S_N(\tau_0(a|y|)), \quad
\eta_{r}(y)=\frac{1-\eta(y)}{|y|^{2r}}, \quad
S_N(t)=\sum_{j=0}^{N-1}t^j.
\]

First, we show that $\mathcal{F}_k(\varphi_1(|\Cdot|))\in L^{1}(\R^d, d\mu_k)$.
Since for a radial function we have
\[
\Delta_k\varphi_1(|y|)=\varphi_1''(|y|)+\frac{2\lambda_k+1}{|y|}\varphi_1'(|y|)
\]
and, for $|t|\leq1/2$,
\[
(1-t)^{-1}-\sum_{j=0}^{N-1}t^j=(1-t)^{-1}-S_N(t)=\frac{t^N}{1-t},
\]
then, by \eqref{eq59} and \eqref{eq65}, we obtain
\[
\Delta_k^s\varphi_1(|y|)=O(|y|^{-2r-N(\lambda_k+1/2)}),\quad |y|\geq 1/2,\ s\in \Z_+.
\]
Hence, for a fixed $N\geq 2+2/(2\lambda_k+1)$, we have $\Delta_k^s\varphi_1(|y|)\in
L^{1}(\R^d, d\mu_k)$, where $s\in \Z_+$. Since $\mathcal{F}_k((-\Delta_k)^s\varphi_1(|\Cdot|))(x)=|x|^{2s}\mathcal{F}_k(\varphi_1(|\Cdot|))(x)$, then
\begin{align*}
|\mathcal{F}_k(\varphi_1(|\Cdot|))(x)|&=|x|^{-2s}\bigl|\int_{\R^d}e_k(-x,y)(-\Delta_k)^s\varphi_1(|y|)\,d\mu_k(y)\bigr|
\\&\leq\frac{\|(-\Delta_k)^s \varphi_1(|y|)\|_{1,d\mu_k}}{|x|^{2s}}.
\end{align*}
Setting $s>\lambda_k+1$ yields $\mathcal{F}_k(\varphi_1(|\Cdot|))\in L^{1}(\R^d, d\mu_k)$.

Second, let us show that $\mathcal{F}_k(\varphi_2(|\Cdot|))\in L^{1}(\R^d, d\mu_k)$ for $r\in \N$.
Let
\[
\tau_0(t)=\sum_{s=1}^m\nu_sj_{\lambda_{k}}(st),\quad \psi_r(x)=2^{2r}\mathcal{F}_k(\eta_{r}(2\Cdot))(x),
\]
\[
A^af(x)=\sum_{s=1}^m\nu_sT^{as}f(x),\quad B^af(x)=\sum_{j=0}^{N-1}(A^a)^jf(x).
\]
Boundedness of the operator $T^t$ in $L^{p}(\R^d, d\mu_k)$ implies
\[
\|A^a\|_{p\to p}=\sup\{\|Af\|_{p,d\mu_k}\colon \|f\|_{p,d\mu_k}\leq 1\}\leq \sum_{s=1}^m|\nu_s|
\]
and
\begin{equation}
\|B^a\|_{p\to p}\leq \sum_{j=0}^{N-1}(\|A\|_{p\to p})^j\leq
N\Bigl(1+\sum_{s=1}^m|\nu_s|\Bigr)^{N-1},\quad 1\leq p<\infty.\label{eq69}
\end{equation}
Then for $p=1$, taking into account Lemma \ref{lem6.3}, we have
\[
\|\mathcal{F}_k(\varphi_2(|\Cdot|))\|_{1,d\mu_k}=\bigl\|B^a\psi_r\bigr\|_{1,d\mu_k}\leq
N\Bigl(1+\sum_{s=1}^m|\nu_s|\Bigr)^{N-1}\|\psi_r\|_{1,d\mu_k}<\infty.
\]
 Thus, $
\mathcal{F}_k(\varphi)
\in L^{1}(\R^d, d\mu_k)$.
Combining Lemma \ref{lem6.1}, relations
\eqref{eq13}, \eqref{eq67}, \eqref{eq68}, and the formula
$\|\mathcal{F}_k((\Cdot/\sigma))\|_{1,d\mu_k}=\|\mathcal{F}_k(\varphi)\|_{1,d\mu_k}$,
we obtain inequality \eqref{eq66} for $r\in \N$.

Let now $r=0$. { Define} the operators $A_1$ and $A_2$ as follows:
\[
\mathcal{F}_k(A_1g)(y)=\varphi_1(|y|/\sigma)\mathcal{F}_k(g)(y)
\]
and
\[
\mathcal{F}_k(A_2g)(y)=\varphi_2(|y|/\sigma)\mathcal{F}_k(g)(y),\quad \varphi_2(|y|)=(1-\eta(2y))S_N(\tau_0(a|y|)).
\]
Since $\mathcal{F}_k(\varphi_1(|\Cdot|))\in L^{1}(\R^d, d\mu_k)$, then by \eqref{eq13} for $1\leq p\leq\infty$
\begin{equation}
\|A_1g\|_{p,d\mu_k}\leq
\|\mathcal{F}_k(\varphi_1(|y|))\|_{1,d\mu_k}\,\|g\|_{p,d\mu_k}\lesssim
\|g\|_{p,d\mu_k},\quad g\in L^{p}(\mathbb{R}^{d},d\mu_{k}). \label{eq70}
\end{equation}

We are left to show that
\[
\|A_2g\|_{p,d\mu_k}\lesssim \|g\|_{p,d\mu_k},\quad 1\leq p\leq\infty,\ g\in L^{p}(\mathbb{R}^{d},d\mu_{k}).
\]
We have
\begin{align*}
\mathcal{F}_k(A_2g)(y)&=(1-\eta(2y/\sigma))S_N(\tau_0(a|y|/\sigma))\mathcal{F}_k(g)(y)\\
&=(1-\eta(2y/\sigma))\mathcal{F}_k(B^{a/\sigma}g)(y)\\
&=\mathcal{F}_k(B^{a/\sigma}g-P_{\sigma/2}(B^{a/\sigma}g))(y).
\end{align*}
Since $B^{a/\sigma}g\in L^{p}(\mathbb{R}^{d},d\mu_{k})$, using Lemma \ref{lem6.5} and inequality
 \eqref{eq69}, we get
\begin{equation}
\|A_2g\|_{p,d\mu_k}\lesssim \|B^{a/\sigma}g\|_{p,d\mu_k}\leq
N\bigl(1+\sum_{s=1}^m|\nu_s|\bigr)^{N-1}\|g\|_{p,d\mu_k}\lesssim
\|g\|_{p,d\mu_k}. \label{eq71}
\end{equation}
Using \eqref{eq70} and \eqref{eq71} with $g= \aaDelta_{a/\sigma}^m
f$, we finally obtain \eqref{eq66} for $r=0$.
\end{proof}

\begin{lemma}\label{lem6.7}
If $\sigma >0$, $1 \leq p\leq\infty$, $m\in \N$, $f\in L^{p}(\R^d, d\mu_k)$, then
\begin{equation}
\|((-\Delta_k)^{m}P_\sigma(f)\|_{p,d\mu_k} \lesssim \sigma^{2m}
\|\aaDelta_{a/(2\sigma)}^{m} f\|_{p,d\mu_k}, \label{eq72}
\end{equation}
where $a=a(\lambda_k, m)>0$ is given in Lemma \ref{lem6.6}.
\end{lemma}

\begin{proof}
We have\begin{align*}
\mathcal{F}_k(((-\Delta_k)^{m}P_\sigma(f))(y)&=|y|^{2m}\eta(y/\sigma)\mathcal{F}_k(f)(y)\\
&=\sigma^{2m}\varphi(y/\sigma)\aaj_{\lambda_{k},m}(a/(2\sigma))\mathcal{F}_k(f)(y)\\
&=\sigma^{2m}\varphi(y/\sigma)\mathcal{F}_k(\aaDelta_{a/(2\sigma)}^{m} f)(y),
\end{align*}
where
\[
\varphi(y)=\frac{|y|^{2m}\eta(y)}{\aaj_{\lambda_{k},m}(a|y|/2)}.
\]

Since $\aaj_{\lambda_{k},m}(a|y|/2)/|y|^{2m}>0$ for $|y|>0$, we observe that
$\varphi\in \mathcal{S}(\R^d)$ and $\mathcal{F}_k(\varphi)\in L^{1}(\R^d,
d\mu_k)$. Then estimate \eqref{eq72} follows from Lemma \ref{lem6.1}, \eqref{eq13}, $\aaDelta_{a/(2\sigma)}^{m}f\in L^{p}(\R^d, d\mu_k)$, and $\|\mathcal{F}_k(\varphi(\Cdot/\sigma))\|_{1,d\mu_k}=\|\mathcal{F}_k(\varphi)\|_{1,d\mu_k}$.
\end{proof}

\subsection{Proofs of Theorem \ref{thm6.1} and \ref{thm6.2} }

\begin{proof}[Proof of Theorem \ref{thm6.2}]
In connection with Lemma \ref{lem6.4}, observe that, for $f\in
L^{p}(\R^d, d\mu_k)$ and $g\in W_{p, k}^{2r}$,
\begin{align*}
\|\aaDelta_{\delta}^rf\|_{p,d\mu_k}&\leq\aaomega_r(\delta,
f)_{p,d\mu_k} \leq \aaomega_r(\delta, f-g)_{p,d\mu_k}
+\aaomega_r(\delta, g)_{p,d\mu_k}\\
&\lesssim (\|f-g\|_{p,d\mu_k}+\delta^{2r}\|(-\Delta_k)^{r}g\|_{p,d\mu_k}).
\end{align*}
Then
\begin{equation}
\|\aaDelta_{\delta}^rf\|_{p,d\mu_k}\leq \aaomega_r(\delta, f)_{p,d\mu_k} \lesssim K_{2r}(\delta, f)_{p,d\mu_k}. \label{eq73}
\end{equation}
On the other hand, $P_\sigma(f)\in W_{p, k}^{2r}$ and
\begin{equation}
K_{2r}(\delta, f)_{p,d\mu_k} \leq \|f-P_\sigma (f)\|_{p,d\mu_k}+
\delta^{2r}\|(-\Delta_k)^{r}P_\sigma(f) \|_{p,d\mu_k}. \label{eq74}
\end{equation}
In light of Lemma \ref{lem6.6},
\begin{equation*}%\label{eq75}
\|f-P_{\sigma}(f)\|_{p,d\mu_{k}} \lesssim \|\aaDelta_{a/(2\sigma)}^rf\|_{p,d\mu_k}.
\end{equation*}
Further, Lemma \ref{lem6.7} yields
\begin{equation}
\|((-\Delta_k)^{r}P_\sigma(f)\|_{p,d\mu_k} \lesssim \sigma^{2r}
\|\aaDelta_{a/(2\sigma)}^{r} f\|_{p,d\mu_k}.\label{eq76}
\end{equation}
Setting $\sigma=a/(2\delta)$, from \eqref{eq74}--\eqref{eq76} we arrive at
\begin{equation}
K_{2r}(\delta, f)_{p,d\mu_k}\lesssim \|\aaDelta_{\delta}^{r}
f\|_{p,d\mu_k}\lesssim \aaomega_r(\delta, f)_{p,d\mu_k}.
\label{eq77}
\end{equation}
\end{proof}

\begin{proof}[Proof of Theorem \ref{thm6.1}]
Using property \eqref{eq50}, inequalities
\eqref{eq73} and \eqref{eq77}, we obtain
\begin{align}
E_\sigma (f)_{p,d\mu_k}&\leq \|f-P_{\sigma/2}(f)\|_{p,d\mu_{k}} \lesssim
\sigma^{-2r}\|\aaDelta_{a/\sigma}^m((-\Delta_k)^rf)\|_{p,d\mu_k}\notag\\
&\lesssim \frac{1}{\sigma^{2r}}\,K_{2m}\Bigl(\frac{a}{\sigma},
(-\Delta_k)^rf\Bigr)_{p,d\mu_k}\lesssim\frac{1}{\sigma^{2r}}\,K_{2m}\Bigl(\frac{1}{\sigma},
(-\Delta_k)^rf\Bigr)_{p,d\mu_k}\notag\\
&\lesssim \frac{1}{\sigma^{2r}}\,\|\aaDelta_{1/\sigma}^m((-\Delta_k)^rf)\|_{p,d\mu_k}
\lesssim \frac{1}{\sigma^{2r}}\,\aaomega_m\Bigl(\frac{1}{\sigma}, (-\Delta_k)^rf\Bigr)_{p,d\mu_k}.
\label{eq78}
\end{align}
\end{proof}

\begin{remark}\label{rem6.8}
The proofs of estimates \eqref{eq77} and \eqref{eq78} for the difference
\eqref{eq53} is based on the fact that the parameter $a$ in Lemmas
\ref{lem6.6} and \ref{lem6.7} is the same. It is possible due to the fact that
$\aaj_{\lambda_{k},m}(t)>0$ for $t>0$, see Remark \ref{rem6.2}. This estimate
is valid for the difference \eqref{eq51} as well, since
$j_{\lambda_{k},m}(t)=(1-j_{\lambda_{k}}(t))^m>0$ for $t>0$.

Therefore, the moduli of smoothness \eqref{eq54} and \eqref{eq56} in
inequalities \eqref{eq60} and \eqref{eq61} can be replaces by the norms of the
corresponding differences \eqref{eq51} and \eqref{eq53}. For the modulus of
smoothness \eqref{eq55} this observation is not valid since
$\aj_{\lambda_{k},m}(t)$ does not
 keep its sign.
\end{remark}
\begin{remark}\label{rem.6.9}
Properties \eqref{eq49} and \eqref{eq50} of the $K$-functional, inequality \eqref{omega-inequality2} and the equivalence \eqref{eq61} imply the following
properties of moduli of smoothness
\begin{align*}
&\textup{(1)}\ \lim_{\delta\to 0+0}\omega_m(\delta,
f)_{p,d\mu_k}=\lim_{\delta\to 0+0} \aomega_m(\delta,
f)_{p,d\mu_k}=\lim_{\delta\to 0+0} \aaomega_m(\delta,
f)_{p,d\mu_k}=0;
\\
&\textup{(2)}\ \omega_m(\lambda\delta, f)_{p,d\mu_k}\lesssim\max\{1,
\lambda^{2m}\}\omega_m(\delta, f)_{p,d\mu_k};
\\
&\textup{(3)}\ \aomega_{l}(\lambda\delta, f)_{p,d\mu_k}\lesssim\max\{1,
\lambda^{2m}\}\aomega_l(\delta, f)_{p,d\mu_k},\quad l=2m-1,\,2m;
\\
&\textup{(4)}\ \aaomega_m(\lambda\delta, f)_{p,d\mu_k}\lesssim\max\{1,
\lambda^{2m}\}\aaomega_m(\delta, f)_{p,d\mu_k};
\\
&\textup{(5)}\ \aomega_{m+r}(\delta, f)_{p,d\mu_k}\lesssim\aomega_m(\delta, f)_{p,d\mu_k};
\\
&\textup{(6)}\ \aaomega_{m+r}(\delta, f)_{p,d\mu_k}\lesssim\aaomega_m(\delta, f)_{p,d\mu_k}.
\end{align*}
\end{remark}

\bigskip
\section{Some inequalities for entire functions}
In this section, we study weighted analogues of the inequalities for entire
functions. In particular, we obtain Nikolskii's inequality (\cite{Nik75}, see
Theorem \ref{thm7.1} below), Bernstein's inequality (\cite{Nik75}, Theorem
\ref{thm7.2}), Nikolskii--Stechkin's inequality (\cite{Nik48,Ste48}, Theorem
\ref{thm7.4}), and Boas-type inequality (\cite{Boa48}, Theorem \ref{thm7.5}).

\begin{theorem}\label{thm7.1}
If $\sigma>0$, $0<p\leq q\leq\infty$, $f\in B_{p, k}^\sigma$, then
\begin{equation}
\|f\|_{q,d\mu_k} \lesssim \sigma^{(2\lambda_k+2)(1/p-1/q)}\|f\|_{p,d\mu_k}.\label{eq79}
\end{equation}
\end{theorem}
\begin{remark}%\label{rem7.1}
Observe that the obtained Nikolskii inequality is sharp, i.e., we actually have
\[
\sup_{f\in B_{p, k}^\sigma, f\ne 0}\frac{\|f\|_{q,d\mu_k}}
{\|f\|_{p,d\mu_k}}\asymp \sigma^{(2\lambda_k+2)(1/p-1/q)},
\]
and an extremizer can be taken as %accuracy is achieved on radial functions
\[
f_{\sigma,m}(x)=\frac{\sin^{2m}(\theta|x|)}{|x|^{2m}},\quad \theta=\frac{\sigma}{2m},
\]
for sufficiently large $m\in\N$.
\end{remark}

\begin{proof}
Let $f\in B_{p, k}^\sigma$, $p\geq 1$, $q=\infty$. By Theorem \ref{thm5.10}, we have
$\supp \mathcal{F}_k(f)\subset B_\sigma$, and then
\begin{equation}
\mathcal{F}_k(f)(y)=\eta(y/\sigma)\mathcal{F}_k(f)(y),\quad \eta(y)=\eta_0(|y|).
\label{eq79+}
\end{equation}
Lemma \ref{lem3.6} implies
\[
f(x)=(f\Ast{\lambda_k}\mathcal{H}_{\lambda_k}(\eta_0(\Cdot/\sigma)))(x)=
\int_0^{\infty}T^tf(x)\mathcal{H}_{\lambda_k}(\eta_0(\Cdot/\sigma))(t)\,d\nu_{\lambda_k}(t).
\]
Taking into account that
\[
\mathcal{H}_{\lambda_k}(\eta_0(\Cdot/\sigma))(t)=\sigma^{2\lambda_k+2}\mathcal{H}_{\lambda_k}(\eta_0)(\sigma
t),
\]
\[
\|\mathcal{H}_{\lambda_k}(\eta_0)(\sigma
t)\|_{p',d\mu_k}=\sigma^{-\frac{2\lambda_k+2}{p'}}\|\mathcal{H}_{\lambda_k}(\eta_0)(t)\|_{p',d\mu_k},
\]
H\"{o}lder's inequality and Theorem \ref{thm3.3} yield
\begin{align*}
|f(x)|&\leq \sigma^{2\lambda_k+2}\|T^tf(x)\|_{p,d\nu_{\lambda_k}}\|\mathcal{H}_{\lambda_k}(\eta_0)(\sigma
t)\|_{p',d\mu_k}\\
&\leq \sigma^{(2\lambda_k+2)/p}\|\mathcal{H}_{\lambda_k}(\eta_0)(t)\|_{p',d\mu_k}\|f\|_{p,d\mu_k}
\lesssim\sigma^{(2\lambda_k+2)/p}\|f\|_{p,d\mu_k},
\end{align*}
i.e.,
\eqref{eq79} holds.

Let $f\in B_{p, k}^\sigma$, $0<p<1$, $q=\infty$. By Theorem~\ref{thm5.1}, $f$ is bounded and $f\in B_{1, k}^\sigma$. We have
\[
\|f\|_{1,d\mu_k} = \| |f|^{1-p} |f|^p\|_{1,d\mu_k} \le \||f|^{1-p}\|_{\infty} \||f|^p\|_{1,d\mu_k} = \|f\|_\infty^{1-p} \|f\|_{p,d\mu_k}^p.
\]
Using \eqref{eq79} with $p=1$ and $q=\infty$,
\[
\|f\|_{1,d\mu_k}\lesssim
\sigma^{2\lambda_k+2}\|f\|_{1,d\mu_k}
\|f\|_\infty^{-p} \|f\|_{p,d\mu_k}^p,
\]
which gives
\[
 \|f\|_\infty \lesssim \sigma^{(2\lambda_k+2)/p}\|f\|_{p,d\mu_k}.
\]
Thus, the proof of \eqref{eq79} for $q=\infty$ is complete.

If $0<p\leq q<\infty$, we obtain
\begin{align*}
\|f\|_{q,d\mu_k} &= \||f|^{1-p/q}|f|^{p/q}\|_{q,d\mu_k} \le \|f\|_\infty^{1-p/q} \|f\|_{p,d\mu_k}^{p/q}
\\ &\le \sigma^{(2\lambda_k+2)(1-p/q)/p} \|f\|_{p,d\mu_k}^{1-p/q} \|f\|_{p,d\mu_k}^{p/q} = \sigma^{(2\lambda_k+2)(1/p-1/q)}\|f\|_{p,d\mu_k}.
\qedhere
\end{align*}
\end{proof}

\begin{theorem}\label{thm7.2}
If $\sigma>0$, $r\in \N$, $1 \leq p\leq \infty$, $f\in B_{p, k}^\sigma$, then
\begin{equation}
\|(-\Delta_k)^rf\|_{p,d\mu_k} \lesssim \sigma^{2r}\|f\|_{p,d\mu_k}.\label{eq80}
\end{equation}
\end{theorem}

\begin{proof}
It is enough to consider the case $r=1$. As in the previous theorem, we use \eqref{eq79+} to obtain
\[
\mathcal{F}_k((-\Delta_k)f)(y)=|y|^2\eta(y/\sigma)\mathcal{F}_k(f)(y)=\sigma^2\varphi_0(|y|/\sigma)\mathcal{F}_k(f)(y),
\]
where $\varphi_0(t)=t^2\eta_0(t)\in \mathcal{S}(\R_+)$. Combining Lemma \ref{lem3.6}, inequality \eqref{eq13}, and
$\|\mathcal{F}_k(\varphi_0(|\Cdot|/\sigma))\|_{1,d\mu_k}=\|\mathcal{F}_k(\varphi_0(|\Cdot|))\|_{1,d\mu_k}$, we arrive at
\[
\|(-\Delta_k)f\|_{p,d\mu_k}\leq \sigma^2 \|\mathcal{F}_k(\varphi_0(|\Cdot|))\|_{1,d\mu_k}\|f\|_{p,d\mu_k}\lesssim\sigma^2 \|f\|_{p,d\mu_k}.
\qedhere
\]
\end{proof}

The next result follows from Lemma \ref{lem6.4}, and Theorem \ref{thm7.2}.

\begin{corollary}\label{cor7.3}
If $\sigma,\,\delta>0$, $m\in \N$, $1 \leq p\leq \infty$, $f\in B_{p, k}^\sigma$, then
\begin{align*}
\omega_m(\delta, f)_{p,d\mu_k}&\lesssim (\sigma\delta)^{2m}\|f\|_{p,d\mu_k},
\\ \aomega_{l}(\delta, f)_{p,d\mu_k}&\lesssim(\sigma\delta)^{2m}\|f\|_{p,d\mu_k},\quad l=2m-1,\,2m,
\\ \aaomega_m(\delta, f)_{p,d\mu_k}&\lesssim (\sigma\delta)^{2m}\|f\|_{p,d\mu_k},
\end{align*}
where constants do not depend on $\sigma, \delta,$ and $f$.
\end{corollary}

\begin{theorem}\label{thm7.4}
If $\sigma>0$, $m\in \N$, $1 \leq p\leq \infty$, $0<t\leq 1/(2\sigma)$, $f\in B_{p, k}^\sigma$, then
\begin{equation}
\|(-\Delta_k)^mf\|_{p,d\mu_k} \lesssim t^{-2m}\|\aaDelta_{t}^{m} f\|_{p,d\mu_k}.\label{eq82}
\end{equation}
\end{theorem}
\begin{remark}%\label{rem7.2}
By Remark \ref{rem6.5}, this inequality can be equivalently written as
\[
\|(-\Delta_k)^mf\|_{p,d\mu_k} \lesssim t^{-2m}
K_{2m}(t, f)_{p,d\mu_k}.
%\|\varDelta_{t}^{m} f\|_{p,d\mu_k},\label{eq81}
\]
\end{remark}
\begin{proof}
%We prove only \eqref{eq82}, since the proof of \eqref{eq81} is similar.
We have
\[
\mathcal{F}_k((-\Delta_k)^mf)(y)=\frac{|y|^{2m}\eta(y/\sigma)}{\aaj_{\lambda_{k},m}(t|y|)}\,
\aaj_{\lambda_{k},m}(t|y|)\mathcal{F}_k(f)(y).
\]
Since for $0<t\leq 1/(2\sigma)$
\[
\eta(y/\sigma)=\eta(y/\sigma)\eta(ty),
\]
we obtain that
\[
\mathcal{F}_k((-\Delta_k)^mf)(y)=t^{-2m}\eta(y/\sigma)\varphi(ty)
\aaj_{\lambda_{k},m}(t|y|)\mathcal{F}_k(f)(y),
\]
where
\[
\varphi(y)=\frac{|y|^{2m}\eta(y)}{\aaj_{\lambda_{k},m}(|y|)}\in \mathcal{S}(\R^d).
\]
Using \[
\aaj_{\lambda_{k},m}(t|\Cdot|)\mathcal{F}_k(f)=\mathcal{F}_k(\aaDelta_{t}^{m}f),\quad
\aaDelta_{t}^{m}f\in L^{p}(\R^d, d\mu_k),
\]
and
\[
\|\mathcal{F}_k(\eta(\Cdot/\sigma))\|_{1,d\mu_k}=\|\mathcal{F}_k(\eta)\|_{1,d\mu_k},\quad
\|\mathcal{F}_k(\varphi(t\Cdot))\|_{1,d\mu_k}=\|\mathcal{F}_k(\varphi)\|_{1,d\mu_k},
\]
and combining Lemma \ref{lem3.6} and inequality \eqref{eq13}, we have
\begin{align*}
\|(-\Delta_k)^mf\|_{p,d\mu_k}&\leq t^{-2m}\|\mathcal{F}_k(\eta(\Cdot/\sigma))\|_{1,d\mu_k}\|\mathcal{F}_k(\varphi(t\Cdot))\|_{1,d\mu_k}
\|\aaDelta_{t}^{m} f\|_{p,d\mu_k}\\
&=t^{-2m}\|\mathcal{F}_k(\eta)\|_{1,d\mu_k}\|\mathcal{F}_k(\varphi)\|_{1,d\mu_k}
\|\aaDelta_{t}^{m} f\|_{p,d\mu_k}\\
&\lesssim t^{-2m}\|\aaDelta_{t}^{m} f\|_{p,d\mu_k}.
\qedhere
\end{align*}
\end{proof}

\begin{theorem}\label{thm7.5}
If $\sigma>0$, $m\in \N$, $1 \leq p\leq \infty$, $0<\delta\leq t\leq 1/(2\sigma)$, $f\in B_{p, k}^\sigma$, then
\begin{equation}
\delta^{-2m}\|\aaDelta_{\delta}^{m} f\|_{p,d\mu_k} \lesssim t^{-2m}\|\aaDelta_{t}^{m} f\|_{p,d\mu_k}.\label{eq84}
\end{equation}
\end{theorem}

\begin{remark}%\label{rem7.3}
Using Remark \ref{rem6.5}, Theorem \ref{thm7.4}, and taking into account that
$\delta^{-2m} K_{2m}(\delta, f)_{p,d\mu_k}$ is decreasing in $\delta$ (see
\eqref{eq50}), inequality \eqref{eq84} can be equivalently written as
\begin{align*}
&
\|(-\Delta_k)^mf\|_{p,d\mu_k} \asymp
\delta^{-2m}\|\aaDelta_{\delta}^{m} f\|_{p,d\mu_k} \asymp t^{-2m}\|\aaDelta_{t}^{m} f\|_{p,d\mu_k},
\\
&
\|(-\Delta_k)^mf\|_{p,d\mu_k} \asymp
\delta^{-2m}
 K_{2m}(\delta, f)_{p,d\mu_k}
 \asymp t^{-2m}K_{2m}(t, f)_{p,d\mu_k}.
\end{align*}
\end{remark}

\begin{proof}
We have
\begin{align*}
\mathcal{F}_k(\aaDelta_{\delta}^{m} f)(y)&=\aaj_{\lambda_{k},m}(\delta|y|)
\mathcal{F}_k(f)(y)\\
&=\eta(y/\sigma)\frac{\aaj_{\lambda_{k},m}(\delta|y|)\eta(ty)}{\aaj_{\lambda_{k},m}(t|y|)}\mathcal{F}_k(\aaDelta_{t}^{m}
f)(y)\\
&=\theta^{2m}\eta(y/\sigma)\varphi_{\theta}(ty)\mathcal{F}_k(\aaDelta_{t}^{m}
f)(y),
\end{align*}
where $\theta=\delta/t\in (0, 1]$,
\[
\varphi_{\theta}(y)=\frac{\psi(\theta y)\eta(y)}{\psi(y)}\in
\mathcal{S}(\R^d),\quad
\psi(y)=\frac{\aaj_{\lambda_{k},m}(|y|)}{|y|^{2m}}\in C^{\infty}(\R^d).
\]
Using Lemma \ref{lem3.6} and estimate \eqref{eq13}, we arrive at inequality \eqref{eq84}:
\begin{align*}
\|\aaDelta_{\delta}^{m} f\|_{p,d\mu_k}&\leq
\theta^{2m}\|\mathcal{F}_k(\eta)\|_{1,d\mu_k}\max_{0\leq \theta\leq
1}\|\mathcal{F}_k(\varphi_{\theta})\|_{1,d\mu_k}\|\aaDelta_{t}^{m}
f\|_{p,d\mu_k}\\
&\lesssim \Bigl(\frac{\delta}{t}\Bigr)^{2m}\|\aaDelta_{t}^{m} f\|_{p,d\mu_k},
\end{align*}
provided that the function $n(\theta)=\|\mathcal{F}_k(\varphi_{\theta})\|_{1,d\mu_k}$ is continuous on $[0, 1]$.
Let us prove this.

Set
 $\varphi_{\theta}(y)=\varphi_{\theta 0}(|y|)$, $r=|y|$, $\rho=|x|$. Then
\begin{align*}
n(\theta)&=\int_{\R^d}\Bigl|\int_{\R^d}\varphi_{\theta}(y)e_k(x,y)\,d\mu_k(y)\Bigr|\,d\mu_k(x)\\
&=\int_{0}^{\infty}\Bigl|\int_{0}^2\varphi_{\theta 0}(r)j_{\lambda_k}(\rho
r)\,d\nu_{\lambda_k}(r)\Bigr|\,d\nu_{\lambda_k}(\rho)\\
&=b_{\lambda_k}^2\int_{0}^{\infty}\Bigl|\int_{0}^2\varphi_{\theta 0}(r)j_{\lambda_k}(\rho
r)r^{2\lambda_k+1}\,dr\Bigr|\rho^{2\lambda_k+1}\,d\rho.
\end{align*}
The inner integral continuously depends on $\theta$. Let us show that the outer
integral converges uniformly in $\theta\in [0,1]$. Since
\cite[Sect.~7.2]{BatErd53}
\[
\frac{d}{dr}\bigl(j_{\lambda_k+1}(\rho r)r^{2\lambda_k+2}\bigr)=(2\lambda_k+2)j_{\lambda_k}(\rho r)r^{2\lambda_k+1},
\]
integrating by parts implies
\begin{align*}
&\int_{0}^2\varphi_{\theta 0}(r)j_{\lambda_k}(\rho
r)r^{2\lambda_k+1}\,dr=\int_{0}^2\varphi_{\theta 0}(r)\,d
\Bigl(\int_{0}^{r}j_{\lambda_k}(\rho \tau)\tau^{2\lambda_k+1}\Bigr)\\
&\qquad =-\frac{1}{2\lambda_k+2}\int_{0}^2\frac{(\varphi_{\theta 0}(r))'}{r}j_{\lambda_k+1}(\rho
r)r^{2\lambda_k+3}\,dr=\dots\\
&\qquad =(-1)^s\Bigl(\prod_{j=1}^s(2\lambda_k+2s)\Bigr)^{-1}
\int_{0}^2
\varphi_{\theta 0}^{[s]}(r)j_{\lambda_k+s}(\rho r)r^{2\lambda_k+2s+1}\,dr,
\end{align*}
where
\[
\varphi_{\theta 0}^{[s]}(r):=\frac{\frac{d}{dr}\varphi_{\theta 0}^{[s-1]}(r)}{r}\in C^{\infty}(\mathbb{R}_+\times [0,1]),
\]
since $\varphi_{\theta 0}(r)$ is even in $r$ and $\varphi_{\theta 0}\in C^{\infty}(\mathbb{R}_+\times [0,1])$. This and \eqref{eq65} give
\[
\Bigl|\int_{0}^2\varphi_{\theta 0}(r)j_{\lambda_k}(\rho
r)r^{2\lambda_k+1}\,dr\Bigr|\leq \frac{c_1(\lambda_k, m,
s)}{(\rho+1)^{\lambda_k+s+1/2}}
\]
and, for $s>\lambda_k+3/2$,
\[
n(\theta)\leq c_2(\lambda_k, m, s)\int_{0}^{\infty}(1+\rho)^{-(s-\lambda_k-1/2)}\,d\rho\leq c_3(\lambda_k, m, s),
\]
completing the proof.
\end{proof}

\begin{remark}%\label{rem7.4}
Combining \eqref{eq79} and \eqref{eq80}, the following Bernstein--Nikolskii inequality is valid
\[
\|(-\Delta_k)^rf\|_{q,d\mu_k}\lesssim \sigma^{2r+(2\lambda_k+2)(1/p-1/q)}\|f\|_{p,d\mu_k},\quad 1\leq p\leq q\leq\infty.
\]
\end{remark}

\begin{remark}%\label{rem7.5}
For radial functions, Nikolskii inequality \eqref{eq79}, Bernstein \eqref{eq80}, Nikolskii--Stechkin \eqref{eq82}, and Boas inequality \eqref{eq84} follow from corresponding estimates in the space $L^{p}(\R_{+},d\nu_{\lambda})$ proved in \cite{Pla07}.
\end{remark}

\bigskip
\section{Realization of $K$-functionals and moduli of smoothness}

In the non-weighted case ($k\equiv0$) the equivalence between
 the classical
modulus of smoothness and the $K$-functional between $L^p$ and
the Sobolev space $W_p^{r}$
%equipped with the semi-norm $\Vert f \Vert_{\dot W_p^{\a}}= \Vert f^{(\a)} \Vert_{p} $:
is well known \cite{DevoreLorentz, johnen}: $1\le p \le\infty$,
\textit{for any integer
%This is given by the following
$r$ one has
\begin{equation*}%\label{eq.th6.0+}
\omega_r(t,f)_{{L^p(\mathbb{R})}}%\widetilde{K}_\psi(f,\delta)_p
\asymp {K}_{r}(f,t)_p,\qquad 1\le p \le\infty,
\end{equation*}
where}
\[
 {K}_{r}(f,t)_p:=\inf_{g\in\dot W_p^{r}} \bigl(\|f-g\|_p+ t^r \| g\|_{
{\dot W_p^{r}}}\bigr).
%\Vert f \Vert_{\dot W_p^{\a}}=
%\Vert f^{(\a)} \Vert_{p}.
%\sum_{|\nu|=\a}\Vert D^\nu f\Vert_p.
\]

Starting from the paper \cite{DHI} (see also \cite[Lemma 1.1]{gogat} for the
fractional case), the following equivalence between the modulus of smoothness
and %is equivalent to the
%This equivalence fails for $0<p<1$ since the $K$-functional cannot
%be used in this case: $K_\a(f,\d)\equiv 0$, see \cite{DHI}. A
%suitable substitute for the $K$-functional for $p<1$ is the
%realization concept.
the realization of the $K$-functional is widely used in approximation theory:
\begin{equation*}%\label{eqKf2}
\omega_r(t,f)_{{L^p(\mathbb{R})}
}%\widetilde{K}_\psi(f,\delta)_p
\asymp \mathcal{R}_{r}(t,f)_p=\inf_{g}\{\Vert
f-g\Vert_p+t^{r}\Vert g^{(r)}\Vert_p\},
\end{equation*}
where $g$ is an entire function of exponential type $1/t.$
%The next result was proved in the paper~\cite{DHI} for $\b\in\N$. For any positive $\b$,
%the proof follows from Theorem \ref{lemNSB} below. % in~\cite{K11} .

Let the realization of the $K$-functional $K_{2r}(t, f)_{p,d\mu_k}$ be given as follows:
\[
\mathcal{R}_{2r}(t, f)_{p,d\mu_k}=\inf\{\|f-g\|_{p,d\mu_k}+t^{2r}\|(-\Delta_k)^rg\|_{p,d\mu_k}\colon
g\in B_{p, k}^{1/t}\}
\]
and
\[
\mathcal{R}^*_{2r}(t, f)_{p,d\mu_k}=\|f-g^*\|_{p,d\mu_k}+t^{2r}\|(-\Delta_k)^rg^*\|_{p,d\mu_k},
\]
where $g^*\in B_{p, k}^{1/t}$ is a near best approximant.

\begin{theorem}%\label{thm8.1}
If $t> 0$, $1\leq p\leq\infty$, $r\in \mathbb{N}$, then for any $f\in
L^{p}(\mathbb{R}^{d},d\mu_{k})$

\begin{equation*}%\label{eq61+}
\begin{aligned}
\mathcal{R}_{2r}(t, f)_{p,d\mu_k}&\asymp
\mathcal{R}^*_{2r}(t, f)_{p,d\mu_k}
\asymp
K_{2r}(t, f)_{p,d\mu_k}\asymp \omega_r(t, f)_{p,d\mu_k}
\\
&
\asymp
\aaomega_r(t, f)_{p,d\mu_k}\asymp\aomega_{2r-1}(t,
f)_{p,d\mu_k}\asymp\aomega_{2r}(t, f)_{p,d\mu_k}.
\end{aligned}
\end{equation*}
\end{theorem}
\begin{proof}By Theorem \ref{thm6.2},
\begin{equation*}
\begin{aligned}
\omega_r(t, f)_{p,d\mu_k}&\asymp
\aaomega_r(t, f)_{p,d\mu_k}
\asymp\aomega_{2r-1}(t,
f)_{p,d\mu_k}\asymp\aomega_{2r}(t, f)_{p,d\mu_k}
\\
&\asymp
K_{2r}(t, f)_{p,d\mu_k}\leq
\mathcal{R}_{2r}(t, f)_{p,d\mu_k}\leq
\mathcal{R}^*_{2r}(t, f)_{p,d\mu_k},%\asymp \omega_r(\delta, f)_{p,d\mu_k}\asymp
\end{aligned}
\end{equation*}
where we have used the fact that
$B_{p, k}^{1/t}\subset W_{p, k}^{2r}$, which follows from Theorem \ref{thm7.2}.

Therefore, it is enough to show that
\[
\mathcal{R}^*_{2r}(t, f)_{p,d\mu_k} \le C\omega_r(t, f)_{p,d\mu_k}.
\]
Indeed, for $g^*$ being the best approximant (or near best approximant), the Jackson inequality given in
Theorem~\ref{thm6.1} implies that
\begin{equation}\label{vsp}
\|f-g^{\ast}\|_{p,d\mu_k}\lesssim E_{1/t}(f)_{p,d\mu_k}\lesssim
 \omega_r(t, f)_{p,d\mu_k}.
\end{equation}
Using the first inequality in Theorem \ref{thm7.4} and taking into account \eqref{vsp}, we have
 %Moreover, taking into account the first inequality in Theorem \ref{thm7.4}, we have
\begin{align*}
\|(-\Delta_k)^rg^*\|_{p,d\mu_k}
&\lesssim t^{-2r}\|\varDelta_{t/2}^{r} g^*\|_{p,d\mu_k}
\\ &\lesssim t^{-2r}\|\varDelta_{t/2}^{r} (g^*-f)\|_{p,d\mu_k}+t^{-2r}\|\varDelta_{t/2}^{r} f\|_{p,d\mu_k}
\\
&\lesssim
t^{-2r}\|g^*-f\|_{p,d\mu_k}+t^{-2r}\omega_r(t/2, f)_{p,d\mu_k}.
\end{align*}
Using again \eqref{vsp}, we arrive at
\[
\|f-g^*\|_{p,d\mu_k}+t^{2r}\|(-\Delta_k)^rg^*\|_{p,d\mu_k}\lesssim \omega_r(t, f)_{p,d\mu_k},
\]
completing the proof.
\end{proof}

The next result answers the following important question (see, e.g.,
\cite{timan,Iva11}): when does the relation
\begin{equation}\label{eq61++}
\omega_m\Bigl(\frac1n, f\Bigr)_{p,d\mu_k} \asymp E_n(f)_{p,d\mu_k}
\end{equation}
(or similar relations with concepts in Theorem \ref{thm8.2})
hold?
\begin{theorem}\label{thm8.2}
Let $1\leq p\leq\infty$ and $m\in \mathbb{N}$. We have that \eqref{eq61++} is
valid if and only if
\begin{equation}\label{eq61+++}
\omega_m\Bigl(\frac1n, f\Bigr)_{p,d\mu_k} \asymp \omega_{m+1}\Bigl(\frac1n, f\Bigr)_{p,d\mu_k}.
\end{equation}
\end{theorem}
\begin{proof}
We prove only the non-trivial part that \eqref{eq61+++} implies \eqref{eq61++}.
Another part of the Theorem~\ref{thm8.2} follows from \eqref{omega-inequality2}, Jackson's
inequality \eqref{eq60} and Remark~\ref{rem.6.9}.

Since, by \eqref{eq50} we have $\omega_m(nt, f)_{p,d\mu_k} \lesssim n^{2m}
\omega_m(t, f)_{p,d\mu_k}$, relation \eqref{eq61+++} implies that
\begin{equation}\label{eq61-}
\omega_{m+1}(nt, f)_{p,d\mu_k} \lesssim n^{2m} \omega_{m+1}(t, f)_{p,d\mu_k}.
\end{equation}
This and Jackson's inequality give
\begin{align*}
&\frac{1}{n^{2(m+1)}} \sum_{j=0}^n
(j+1)^{2(m+1)-1} E_j(f)_{p,d\mu_k}
\\ &\qquad \lesssim
\frac{1}{n^{2(m+1)}} \sum_{j=0}^n
(j+1)^{2(m+1)-1}
\omega_{m+1}\Bigl(\frac{1}{j+1}, f\Bigr)_{p,d\mu_k}
\\ &\qquad \lesssim
\omega_{m+1}\Bigl(\frac1n, f\Bigr)_{p,d\mu_k}.
%\frac{n^{2m}}{n^{2(m+1)}} \sum_{j=0}^n
%(j+1)^{2(m+1)-1-2m}
\end{align*}
Moreover, Theorem \ref{thm9.1} below implies
\begin{align*}
\omega_{m+1}\Bigl(\frac1{l n}, f\Bigr)_{p,d\mu_k}
&\lesssim \frac{1}{(ln)^{2(m+1)}} \sum_{j=0}^{ln}
(j+1)^{2(m+1)-1} E_j(f)_{p,d\mu_k}
\\ &\lesssim \frac{1}{l^{2(m+1)}}\,\omega_{m+1}\Bigl(\frac1n, f\Bigr)_{p,d\mu_k}
\\ &\qquad +\frac{1}{(ln)^{2(m+1)}} \sum_{j=n+1}^{ln}
(j+1)^{2(m+1)-1} E_j(f)_{p,d\mu_k},
\end{align*}
or, in other words,
\begin{align*}
&\frac{1}{n^{2(m+1)}} \sum_{j=n+1}^{ln}
(j+1)^{2(m+1)-1} E_j(f)_{p,d\mu_k}
\\ &\qquad \gtrsim
C
l^{2(m+1)}
\omega_{m+1}\Bigl(\frac1{l n}, f\Bigr)_{p,d\mu_k}-
\omega_{m+1}\Bigl(\frac1n, f\Bigr)_{p,d\mu_k}.
\end{align*}
Using again \eqref{eq61-}, we obtain
\[
\frac{1}{n^{2(m+1)}} \sum_{j=n+1}^{ln}
(j+1)^{2(m+1)-1} E_j(f)_{p,d\mu_k}\gtrsim
(
C
l^{2}-1)
\omega_{m+1}\Bigl(\frac1n, f\Bigr)_{p,d\mu_k}.
\]
Taking into account monotonicity of $ E_j(f)_{p,d\mu_k}$ and choosing $l$
sufficiently large, we arrive at \eqref{eq61++}.
\end{proof}

\bigskip
\section{Inverse theorems of approximation theory}

\begin{theorem}\label{thm9.1}
Let $m, n \in \mathbb{N}$, $1\leq p\leq\infty$, $f \in L^{p}(\R^d, d\mu_k)$. We have
\begin{equation}
K_{2m}\Bigl(\frac{1}{n}, f\Bigl)_{p,d\mu_k}
%\omega_m\Bigl(\frac{1}{n}, f\Bigr)_{p,d\mu_k}
\lesssim \frac{1}{n^{2m}} \sum_{j=0}^n
(j+1)^{2m-1} E_j(f)_{p,d\mu_k}. \label{eq85}
\end{equation}
\end{theorem}

\begin{remark}
By Remark \ref{rem6.5}, $K_{2m}\bigl(\frac{1}{n}, f\bigl)_{p,d\mu_k}$ in this
inequality can be equivalently replaced by $\omega_m\bigl(\frac{1}{n},
f\bigr)_{p,d\mu_k}$, $\aaomega_m\bigl(\frac{1}{n}, f\bigr)_{p,d\mu_k}$, and
$\aomega_l\bigl(\frac{1}{n}, f\bigr)_{p,d\mu_k}$, $l=2m-1,2m$.
%where
%$\Omega_m$ is any of three moduli of smoothness \eqref{eq54}--\eqref{eq56}.
\end{remark}

\begin{proof}
Let us prove \eqref{eq85} for $\omega_m \bigl(\frac{1}{n}, f\bigr)_{p,d\mu_k}$.
By Theorem \ref{thm5.11}, for any $\sigma>0$ there exists $f_\sigma\in B_{p,
k}^\sigma$ such that
\[
\|f-f_\sigma\|_{p,d\mu_k}=E_\sigma(f)_{p,d\mu_k},\quad
E_0(f)_{p,d\mu_k}=\|f\|_{p,d\mu_k}.
\]
For any $s\in\mathbb{Z}_+$,
\begin{align*}
\omega_m(1/n, f)_{p,d\mu_k}&\leq \omega_m(1/n, f-f_{2^{s+1}})_{p,d\mu_k}+ \omega_m(1/n,
f_{2^{s+1}})_{p,d\mu_k}\\
&\lesssim E_{2^{s+1}}(f)_{p,d\mu_k}+\omega_m(1/n, f_{2^{s+1}})_{p,d\mu_k}.
\end{align*}
Using Lemma \ref{lem6.4},
\begin{align*}
&\omega_m(1/n, f_{2^{s+1}})_{p,d\mu_k}\lesssim {n^{-2m}}\|(-\Delta_k)^m
f_{2^{s+1}}\|_{p,d\mu_k}\\
&\qquad \lesssim \frac{1}{n^{2m}} \Bigl(\|(-\Delta_k)^mf_1\|_{p,d\mu_k}+\sum_{j=0}^s \|(-\Delta_k)^mf_{2^{j+1}}-
(-\Delta_k)^mf_{2^j}\|_{p,d\mu_k}\Bigr).
\end{align*}
Then Bernstein inequality \eqref{eq80} implies that
\begin{align*}
\|(-\Delta_k)^mf_{2^{j+1}}-(-\Delta_k)^mf_{2^j}\|_{p,d\mu_k} &\lesssim
2^{2m(j+1)}\|f_{2^{j+1}}-f_{2^j}\|_{p,d\mu_k}\\& \lesssim
2^{2m(j+1)}E_{2^j}(f)_{p,d\mu_k},
\\
\|(-\Delta_k)^mf_1\|_{p,d\mu_k}&\lesssim E_0(f)_{p,d\mu_k}.
\end{align*}
Thus,
\[
\omega_m(1/n, f_{2^{s+1}})_{p,d\mu_k}\lesssim
\frac{1}{n^{2m}}\Bigl(E_0(f)_{p,d\mu_k}+ \sum_{j=0}^s2^{2m(j+1)}E_{2^j}(f)_{p,d\mu_k}\Bigr).
\]
Taking into account that
\begin{equation}
\sum_{l=2^{j-1}+1}^{2^j} l^{2m-1}E_l(f)_{p,d\mu_k} \geq
%2^{(j-1)(2m-1)}E_{2^j}(f)_{p,d\mu_k}\cdot
%2^{j-1}=
2^{2m(j-1)}E_{2^j}(f)_{p,d\mu_k},\label{eq88}
\end{equation}
we have
\begin{align*}
&\omega_m(1/n, f_{2^{s+1}})_{p,d\mu_k}\lesssim
\frac{1}{n^{2m}}\Bigl(E_0(f)_{p,d\mu_k}+2^{2m} E_1(f)_{p,d\mu_k}\\
&\qquad +\sum_{j=1}^s
2^{4m} \sum_{l=2^{j-1}+1}^{2^j} l^{2m-1} E_l(f)_{p,d\mu_k}\Bigr)\lesssim \frac{1}{n^{2m}}\sum_{j=0}^{2^s}(j+1)^{2m-1}E_j(f)_{p,d\mu_k}.
\end{align*}
Choosing $s$ such that $2^s\leq n<2^{s+1}$ implies \eqref{eq85}.
\end{proof}

Theorem \ref{thm9.1} and Jackson's inequality imply the following Marchaud inequality.
\begin{corollary}%\label{cor9.2}
Let $m \in \mathbb{N}$, $1\leq p\leq\infty$, $f \in L^{p}(\R^d, d\mu_k)$. We have
\[
K_{2m}(\delta, f)_{p,d\mu_k}
%\omega_m\Bigl(\frac{1}{n}, f\Bigr)_{p,d\mu_k}
\lesssim
\delta^{2m} \Bigl(
\|f\|_{p,d\mu_k}+
\int_\delta^1
{t^{-2m}}
{K_{2m+2}(t, f)_{p,d\mu_k}}\,\frac{dt}{t}
\Bigr).
\]
\end{corollary}

\begin{theorem}%\label{thm9.3}
Let $1\leq p\leq\infty$, $f \in L^{p}(\R^d, d\mu_k)$ and $r\in\mathbb{N}$ be
such that $\sum _{j=1}^{\infty} j^{2r-1}E_j(f)_{p,d\mu_k} <\infty.$ Then $f \in
W^{2r}_{p, k}$ and, for any $m,n\in\mathbb{N}$, we have
%\begin{equation}
%\omega_m\Bigl(\frac{1}{n}, (-\Delta_k)^rf\Bigr)_{p,d\mu_k} \lesssim\frac{1}{n^{2r}} \sum_{j=0}^n
%(j+1)^{2k+2r-1}E_j(f)_{p,d\mu_k} +\sum_{j=n+1}^\infty j^{2r-1}
%E_j(f)_{p,d\mu_k},\label{eq89}
%\end{equation}
%\begin{equation}
%\aomega_l\Bigl(\frac{1}{n}, (-\Delta_k)^rf\Bigr)_{p,d\mu_k} \lesssim\frac{1}{n^{2r}} \sum_{j=0}^n
%(j+1)^{2k+2r-1}E_j(f)_{p,d\mu_k} +\sum_{j=n+1}^\infty j^{2r-1}
%E_j(f)_{p,d\mu_k} \label{eq90}
%\end{equation}
%and
\begin{equation}
K_{2m}\Bigl(\frac{1}{n}, (-\Delta_k)^rf\Bigl)_{p,d\mu_k}
%\Omega_m\Bigl(\frac{1}{n}, (-\Delta_k)^rf\Bigr)_{p,d\mu_k}
\lesssim\frac{1}{n^{2r}} \sum_{j=0}^n
(j+1)^{2k+2r-1}E_j(f)_{p,d\mu_k} +\sum_{j=n+1}^\infty j^{2r-1}
E_j(f)_{p,d\mu_k}.\label{eq91}
\end{equation}
\end{theorem}
\begin{remark}%\label{rem9.2}
We can replace $K_{2m}\bigl(\frac{1}{n}, (-\Delta_k)^rf\bigl)_{p,d\mu_k}$ by
any of moduli $\omega_m\bigl(\frac{1}{n}, (-\Delta_k)^rf\bigr)_{p,d\mu_k}$,
$\aomega_l\bigl(\frac{1}{n}, (-\Delta_k)^rf\bigr)_{p,d\mu_k},$ and $\aaomega_m
\bigl(\frac{1}{n}, (-\Delta_k)^rf\bigr)_{p,d\mu_k}$, $l=2m-1,2m$.
%where
%$\Omega_m$ is any of three moduli of smoothness \eqref{eq54}--\eqref{eq56}.
\end{remark}

\begin{proof}Let us prove \eqref{eq91} for $\omega_m
\bigl(\frac{1}{n}, (-\Delta_k)^rf\bigr)_{p,d\mu_k}$. Consider
\begin{equation}
(-\Delta_k)^r
f_1+\sum_{j=0}^{\infty}\bigl((-\Delta_k)^rf_{2^{j+1}}-(-\Delta_k)^rf_{2^{j}}\bigr).\label{eq92}
\end{equation}
By Bernstein's inequality \eqref{eq80},
\[
\|(-\Delta_k)^rf_{2^{j+1}}-(-\Delta_k)^rf_{2^j}\|_{p,d\mu_k}\lesssim
2^{(j+1)r}E_{2^j}(f)_{p,d\mu_k}\lesssim\sum_{l=2^{j-1}+1}^{2^j}l^{r-1}E_l(f)_{p,d\mu_k}.
\]
Therefore, series \eqref{eq92} converges to a function $g\in L^{p}(\R^d,
d\mu_k)$. Let us show that $g=(-\Delta_k)^rf$, i.e., $f\in W^{2r}_{p, k}$. Set
\[
S_N=(-\Delta_k)^rf_1+\sum_{j=0}^N \bigl((-\Delta_k)^rf_{2^{j+1}}-(-\Delta_k)^rf_{2^j}\bigr).
\]
Then
\begin{align*}
\<\mathcal{F}_k(g), \varphi\> &= \<g, \mathcal{F}_k(\varphi)\> =\lim_{N\to\infty} \<S_N,
\mathcal{F}_k(\varphi)\>=\\
&= \lim_{N\to\infty}\<\mathcal{F}_k(S_N), \varphi\>=\lim_{N\to\infty}\<|y|^{2r}\mathcal{F}_k(f_{2^{N+1}}), \varphi\>
= \<|y|^{2r} \mathcal{F}_k(f), \varphi\>,
\end{align*}
where $\varphi\in \mathcal{S}(\R^d)$.
Hence, $\mathcal{F}_k(g)(y) =|y|^{2r}\mathcal{F}_k(f)(y)$ and $g=(-\Delta_k)^rf$.

To obtain \eqref{eq91}, we write
\[
\omega_m(1/n, (-\Delta_k)^r f)_{p,d\mu_k}\leq\omega_m(1/n, (-\Delta_k)^r f-S_N)_{p,d\mu_k}+\omega_m(1/n, S_N)_{p,d\mu_k}.
\]
The first term is estimated as follows
\begin{align*}
\omega_m(1/n, (-\Delta_k)^r f-S_N)_{p,d\mu_k} &\lesssim\|(-\Delta_k)^r
f-S_N\|_{p,d\mu_k}\\
& \lesssim \sum_{j=N+1}^{\infty}2^{2r(j+1)}E_{2^j}(f)_{p,d\mu_k}
\lesssim\sum_{l=2^N+1}^{\infty} l^{2r-1}E_l(f)_{p,d\mu_k}.
\end{align*}
Moreover,
by Corollary \ref{cor7.3},
\begin{align*}
\omega_m(1/n, S_N)_{p,d\mu_k} &\leq \omega_m(1/n, (-\Delta_k)^r
f_1)_{p,d\mu_k}\\
&\qquad +\sum_{j=0}^{N}\omega_m\bigl(1/n, (-\Delta_k)^r
f_{2^{j+1}}-(-\Delta_k)^r f_{2^j}\bigr)_{p,d\mu_k}\\
& \lesssim \frac{1}{n^{2r}}\Bigl(E_0(f)_{p,d\mu_k}+ \sum_{j=0}^{N} 2^{2(m+r)(j+1)}E_{2^j}(f)_{p,d\mu_k}\Bigr).
\end{align*}
Using \eqref{eq88} and choosing $N$ such that $2^N\leq n<2^{N+1}$ completes the proof of \eqref{eq91}.
\end{proof}

\end{document}